\documentclass[12pt]{article}
\usepackage{amsmath,amsthm,amsfonts,amssymb,csquotes,epsfig, titlesec, epsfig, float, caption,wrapfig, mathrsfs, mathtools, multicol}
\usepackage[title]{appendix}
\usepackage{authblk}
\usepackage[hidelinks]{hyperref}
\usepackage{tikz}
\usepackage{tikz-cd}
\usetikzlibrary{cd,calc}
\usepackage{xcolor}
\usepackage{accents}
\usepackage{marvosym}
\usepackage[alphabetic,abbrev,lite,backrefs]{amsrefs} 
\usepackage[]{enumerate}
\usepackage[american]{babel}
\usepackage[left=1in,top=1in,right=1in]{geometry}
\usepackage[capitalize]{cleveref}
\usepackage{subcaption}
\usepackage[all]{xy}
\usepackage{enumitem}

\newtheorem{mainthm}{Theorem}
\Crefname{mainthm}{Theorem}{Theorems}

\usepackage{chngcntr}
\counterwithin{figure}{section}
\numberwithin{equation}{section}

\newtheorem{lemma}[equation]{Lemma}
\newtheorem{theorem}[equation]{Theorem}
\newtheorem{prop}[equation]{Proposition}
\newtheorem{cor}[equation]{Corollary}
\newtheorem{thm}[equation]{Theorem} 

\theoremstyle{definition}
\newtheorem{defn}[equation]{Definition}
\newtheorem{conjecture}[equation]{Conjecture}
\newtheorem{convention}[equation]{Convention}
\newtheorem{examplex}[equation]{Example}
\newenvironment{example}
    {\pushQED{\qed}\examplex}
  {\popQED\endexamplex}

\theoremstyle{remark}
\newtheorem{remarkx}[equation]{Remark}
\newenvironment{remark}
    {\pushQED{\qed}\remarkx}
  {\popQED\endremarkx}
\newtheorem*{case*}{Case}

\titleformat*{\section}{\normalsize \bfseries \filcenter}
\titleformat*{\subsection}{\normalsize \bfseries }
\newcommand{\Manoa}{M\=anoa}
\newcommand{\Hawaii}{Hawai\kern.05em`\kern.05em\relax i}

\renewcommand{\gets}{\longleftarrow}
\renewcommand{\geq}{\geqslant}
\renewcommand{\leq}{\leqslant}

\newcommand{\beq}{\begin{displaymath}}
\newcommand{\eeq}{\end{displaymath}}

\newcommand{\defi}[1]{\textsf{#1}} 

\renewcommand{\AA}{\ensuremath{\mathbb{A}}}

\newcommand{\GG}{\mathbb G}

\newcommand{\PP}{\mathbb P}
\newcommand{\QQ}{\mathbb Q}
\newcommand{\R}{\mathbb{R}}
\newcommand{\RR}{\ensuremath{\mathbb{R}}}
\newcommand{\ZZ}{\mathbb Z}

\newcommand{\bM}{\Phi_{\Cox,S}}

\newcommand{\cC}{\mathcal{C}}

\newcommand{\cD}{\mathcal D}

\newcommand{\cE}{\mathcal{E}}
\newcommand{\cF}{\ensuremath{\mathcal{F}}}

\newcommand{\cH}{\mathcal H}

\newcommand{\cO}{{\mathcal O}}

\newcommand\cP{\mathcal{P}}

\newcommand{\cT}{\ensuremath{\mathcal{T}}}

\newcommand{\cW}{\mathcal{W}}

\newcommand{\sX}{{\mathcal X}}
\newcommand{\cX}{\mathcal{X}}
\newcommand\cY{\mathcal{Y}}
\newcommand{\sY}{{\mathcal Y}}

\newcommand{\mff}{\mathfrak f}
\newcommand{\mfX}{\mathfrak X}

\newcommand{\tN}{\widetilde{N}}

\newcommand{\tX}{\widetilde{X}}
\newcommand\tsX{\widetilde{\sX}}
\newcommand{\tY}{\widetilde{Y}}
\newcommand{\tbeta}{\widetilde{\beta}}

\newcommand{\tsigma}{\widetilde{\sigma}}
\newcommand{\tSigma}{\widetilde{\Sigma}}

\newcommand{\ce}{\coloneqq}

\newcommand{\Cl}{\on{Cl}}

\newcommand{\Cox}{\operatorname{Cox}}

\newcommand{\End}{\on{End}}

\newcommand\HHL{\operatorname{H}}

\newcommand{\Hom}{\operatorname{Hom}}

\newcommand{\Id}{\mathrm{id}}

\newcommand{\leftexp}[2]{{\vphantom{#2}}^{#1}{#2}}

\newcommand\on{\operatorname}

 \newcommand{\Perf}{\on{Perf}}
\newcommand{\Pic}{\operatorname{Pic}}

\newcommand{\RHom}{\on{Hom}}
 \newcommand{\Spec}{\operatorname{Spec}}

\newcommand{\weezer}{\leftexp{=}{\kern-0.23em\operatorname{W}}^{\kern-0.21em =}}

\title{ \vspace{-2 em} \large \textbf{King's Conjecture and the Cox category}}
\date{}
\author{\normalsize Matthew R. Ballard, Christine Berkesch, Michael K.~Brown,
Lauren Cranton Heller, 
\linebreak Daniel Erman, David Favero, Sheel Ganatra, Andrew Hanlon, and Jesse
Huang}

\newcommand{\Addresses}{{
  \bigskip
  \footnotesize
\noindent M.R.~Ballard, \textsc{Department of Mathematics, University of South Carolina}\par\nopagebreak
  \noindent \textit{E-mail address}: \texttt{ballard@math.sc.edu}

  \medskip

  \noindent C.~Berkesch, \textsc{School of Mathematics, University of Minnesota}\par\nopagebreak
  \noindent \textit{E-mail address}: \texttt{cberkesc@umn.edu}

  \medskip

  \noindent M.~K.~Brown, \textsc{Department of Mathematics and Statistics, Auburn University}\par\nopagebreak
  \noindent \textit{E-mail address}: \texttt{mkb0096@auburn.edu}

  \medskip

    \noindent L.~Cranton~Heller, \textsc{Department of Mathematics, University of Nebraska}\par\nopagebreak
  \noindent \textit{E-mail address}: \texttt{lheller2@unl.edu}

  \medskip

  \noindent D.~Erman, \textsc{Department of Mathematics, University of \Hawaii \, at \Manoa}\par\nopagebreak
  \noindent \textit{E-mail address}: \texttt{erman@hawaii.edu}

  \medskip

  \noindent D.~Favero, \textsc{School of Mathematics, University of Minnesota}\par\nopagebreak
  \noindent \textit{E-mail address}: \texttt{favero@umn.edu}

  \medskip

  \noindent S.~Ganatra, \textsc{Department of Mathematics, University of Southern California}\par\nopagebreak
  \noindent \textit{E-mail address}: \texttt{sheel.ganatra@usc.edu}

  \medskip
  
  \noindent A.~Hanlon, \textsc{Department of Mathematics, University of Oregon}\par\nopagebreak
  \noindent \textit{E-mail address}: \texttt{ahanlon@uoregon.edu}

  \medskip

  \noindent J.~Huang, \textsc{Department of Pure Mathematics, University of Waterloo}\par\nopagebreak
  \noindent \textit{E-mail address}: \texttt{j654huang@uwaterloo.ca}
}}

\begin{document}

\maketitle

\begin{abstract}
We state and prove a realization of King's Conjecture for
a category glued from the derived categories of all of the toric varieties
arising from a given Cox ring. Our perspective extends ideas of Beilinson and
Bondal to all semiprojective toric varieties. 
\end{abstract}

\section{Introduction}\label{sec:intro}

Beilinson's work on the derived category of $\PP^n$~\cite{beilinson} and the
\defi{Beilinson collection} of line bundles $\cO_{\PP^n}(-n), \ldots,
\cO_{\PP^n}(-1), \cO_{\PP^n}$ broadly set the stage for 
exceptional collections, tilting bundles, and semiorthogonal
decompositions~\cites{bondal-associative, BK89}.
One attempt to generalize Beilinson’s result was King’s Conjecture \cite{king},
which proposed that, like $\PP^n$, every smooth projective toric variety has a
full strong exceptional collection of line bundles. This turned out to be
false~\cites{hille-perling, michalek11, efimov}, but nevertheless, it has
continued to inspire work on exceptional collections for toric
varieties~\cites{altmann-immaculate,
BFK_GIT,BallardEtAl2019ArithmeticToric,BallardEtAl2019ToricFrobenius,BallardEtAl2018arxivDerivedCategories,bernardi-tirabassi,borisov-hua,borisov-orlov-equivariant,borisov-wang,costa-miro-roig0,costa-miro-roig,costa-miro-roig2,dey-lason-michalek,Jerby1,Jerby2,Jerby3,kawamataI,kawamataII,lason-michalek,ohkawa2013frobenius,prabhu-naik,sanchez2023derived,uehara},
\emph{inter alia}.  Notably, Kawamata proved a weakening of King's conjecture by
dropping the line bundle and strong-ness conditions~\cite{kawamataI}. 

Our main result is a realization of King’s Conjecture.\footnote{Le roi est mort, vive le roi!}  The basic
idea originates from the correspondence between Cox rings and toric varieties,
which is generally not a bijection: several distinct toric varieties can
correspond to the same Cox ring.  We introduce a method for gluing the derived
categories of all toric varieties arising from a given Cox ring.  
For a toric variety $X$, we refer to 
this as the \defi{Cox category}, denoted $D_{\Cox}(X)$ (see Definition~\ref{defn:DCox}).  
The exceptional collection is the \defi{Bondal--Thomsen collection}, denoted $\Theta$, which is the natural generalization of the Beilinson collection, for a given Cox ring; see
Figure~\ref{fig:HirzSecondaryFan} and Definition~\ref{defn:BT}.
The following is our main result. 

\begin{mainthm}\label{thm:main1}
Let $X$ be a semiprojective toric variety.
The direct sum of the line bundles in $\Theta$ is a tilting object for $D_{\Cox}(X)$.
If $X$ is projective, then $\Theta$ forms a full strong exceptional collection of line bundles for $D_{\Cox}(X)$ under a natural ordering.
\end{mainthm}

Our theorem also manifests a vision of Bondal.  In 2006, Bondal used the
collection of line bundles $\Theta$ to provide a novel perspective on derived
categories of toric varieties that was especially amenable to homological mirror
symmetry~\cite{Bondal}.\footnote{This same collection was previously described
by Thomsen as Frobenius summands in~\cite{Thomsen}.}  Bondal's
proposal has been immensely influential, inspiring the coherent-constructible
correspondence~\cites{FLTZ11, FLTZ12, FLTZ14} as well as many other results on
derived categories of toric varieties.

\cref{thm:main1} demonstrates that both King's conjecture and Bondal's proposal
always hold---without exception---as long as one synthesizes all of the toric geometry associated to a given Cox ring.
In combination with \cref{thm:resDiagCoxVague} below, our results show that Beilinson's work on $D(\PP^n)$--leveraging an explicit resolution of the diagonal to find a full strong exceptional collection--can be generalized to $D_{\Cox}(X)$ for any projective toric $X$.

These properties of the Cox category offer a novel and concrete instantiation of the  philosophy connecting derived categories and birational geometry.
We consider this analogous to how gluing affine varieties to form projective varieties can yield
simpler structures.

We construct $D_{\Cox}(X)$ using the secondary fan $\Sigma_{GKZ}$, which is the toric version of the Mori chamber decomposition in birational geometry. For each
cone $\Gamma \in \Sigma_{GKZ}$, there is a toric variety
$X_\Gamma$ as described in~\cite{CLSToricVarieties}*{Chapters 14--15}.
The maximal chambers of $\Sigma_{GKZ}$ correspond to simplicial toric varieties,
which we label as $X_1, \dots, X_r$.  Each $X_i$ comes with an irrelevant ideal
$B_i\subseteq S$, and the pair $(S,B_i)$ determines both the toric variety $X_i$
as a GIT quotient of $\Spec(S) - V(B_i)$ and a Deligne--Mumford toric stack
$\sX_i$ as the corresponding stack quotient. These are identical when $X_i$ is
smooth, but when $X_i$ is simplicial but not smooth, the stack $\sX_i$ remains 
smooth, making its derived category better behaved. Accordingly, our central
definition will involve these toric stacks.

\begin{defn}\label{defn:DCox}
Let $\sX_1, \dots, \sX_r$ be the toric stacks corresponding to the maximal chambers of $\Sigma_{GKZ}(X)$.
Let $\tsX$ be any smooth toric stack with proper birational toric morphisms $\pi_i\colon \tsX\to \sX_i$ for all $i$.  The \defi{Cox category} $D_{\Cox}(X)$ is  the full subcategory of $D(\tsX)$ generated by $\pi_i^*D(\sX_i)$ for $1\leq i\leq r$.\footnote{Throughout the paper, we adopt the convention that all functors are derived; for example, $\pi_i^*$ will denote $L\pi_i^*$, $\Hom$ will denote $\operatorname{RHom}$, and so on. See also \cref{conv:functors}.}  We will often refer to $D_{\Cox}(X)$ as simply $D_{\Cox}$. 
\end{defn}

While $D(\tsX)$ provides a convenient way to compare the $D(\sX_i)$, $D_{\Cox}$ itself is independent of the choice of $\tsX$,
and thus $D_{\Cox}(X)$ may be seen as an invariant of the Cox ring $S$.  
Because the functors $\pi_i^*$ are fully faithful, $D_{\Cox}$ 
contains a copy of each $D(\sX_i)$. Naturally, these copies 
overlap. For example, the structure sheaf $\cO_{\sX_i}$ pulls 
back to $\cO_{\tsX}$ for all $i$, meaning that $\cO_{\tsX}$ lies in 
$\pi_i^* D(\sX_i)$ for all $i$.  
How do the copies of $D(\sX_i)$ in $D_{\Cox}$ relate to one another?
The graphs of the birational maps $\sX_i \dashrightarrow \sX_j$ induce
Fourier--Mukai transforms  $\Phi_{ij}\colon D(\sX_i)\to D(\sX_j)$. For any
$\cE\in D(\sX_i)$, we refer to the set of $\Phi_{ij}(\cE)$ (for all $j$) as the
\defi{Fourier--Mukai transforms of $\cE$}.
Continuing with the metaphor that $D(\sX_i)$ is analogous to an ``affine
patch'' of $D_{\Cox}$, each Fourier--Mukai transform corresponds to a
transition function, and each $\pi_{i*}\colon D_{\Cox} \to D(\sX_i)$
represents restriction to the patch.

In Appendix \ref{sec:grothendieckconstr}, we make this
birational-gluing-via-integral-transforms
perspective precise (see \Cref{thm:cox_groth}) by providing an equivalent 
characterization of $D_{\Cox}$ as constructed from a lax functor whose 1-skeleton is the data of 
$(D(\sX_i),\Phi_{ij})$, using the Grothendieck construction \cite{SGA1}*{\S\.VI.8}. Just as a scheme is glued as a colimit of affine
schemes, $D_{\Cox}$ is the pretriangulated envelope of the lax colimit of
derived categories.

\begin{figure}
\centering
\begin{tikzpicture}[scale = .65]
\draw[step=1cm,gray,dashed,very thin] (0,0) grid (7,7);
\draw[fill=gray!60] (4,3)--(0,4.33)--(0,7)--(7,7)--(7,3)--(4,3);
\draw (4,3) -- (4,7);
\filldraw[black] (5,3) circle (3pt);
\filldraw[black] (5,3) circle (0pt) node[anchor=south]{\tiny $d_1$};
\filldraw[black] (3.5,5.5) circle (0pt) node[anchor=east]{{ $\PP(1,1,3)$}};
\filldraw[black] (4.8,5.5) circle (0pt) node[anchor=west]{{ $\cH_3$}};
\filldraw[black] (2,4) circle (3pt) node[anchor=south]{\tiny$d_2$};
\filldraw[black] (3,4) circle (3pt) node[anchor=south]{\tiny$d_3$};
\filldraw[black] (4,4) circle (3pt) node[anchor=south]{\tiny$d_4$};
\filldraw[black] (5,4) circle (3pt) node[anchor=south]{\tiny$d_5$};
\filldraw[black] (4,3) circle (3pt) node[anchor=north]{\tiny$d_0=(0,0)$};
\filldraw[black]  (1,-1) circle (0pt);
\end{tikzpicture}
\qquad
\begin{tikzpicture}[scale = .65]
\draw[step=1cm,gray,dashed,very thin] (0,0) grid (7,7);
\node[anchor=north] at (5,3){\tiny $\deg(x_0)=\deg(x_2)$};
\draw[black, ->, line width = .4mm] (4,3)--(5,3);
\draw[black, ->, line width = .4mm] (4,3)--(4,4);
\draw[black, ->, line width = .4mm] (4,3)--(1,4);
\node at (4,4.25){\tiny $\deg(x_3)$};
\node[anchor=south] at (1,4) {\tiny $\deg(x_1)$};
\filldraw[black]  (1,-1) circle (0pt);
\end{tikzpicture}
\qquad
\begin{tikzpicture}[scale = .65]
\draw[step=1cm,gray,dashed,very thin] (0,0) grid (7,7);
\filldraw[black] (4.3,5.2) circle (0pt) node[anchor=west]{{ $- Z$}};
\filldraw[black] (5,3) circle (3pt) node[anchor=north]{\tiny $d_1$};
\draw[black, line width = .75mm] (6,3)--(4,3)--(1,4);
\draw[black, dashed, line width = .75mm] (1,4) -- (1,5)--(3,5)--(6,4)--(6,3);
\draw[fill = gray, opacity = .3] (4,3)--(1,4) -- (1,5)--(3,5)--(6,4)--(6,3)--(4,3);
\filldraw[black] (2,4) circle (3pt) node[anchor=south]{\tiny$d_2$};
\filldraw[black] (3,4) circle (3pt) node[anchor=south]{\tiny$d_3$};
\filldraw[black] (4,4) circle (3pt) node[anchor=south]{\tiny$d_4$};
\filldraw[black] (5,4) circle (3pt) node[anchor=north]{\tiny$d_5$};
\filldraw[black] (4,3) circle (3pt) node[anchor=north]{\tiny$d_0$};
\filldraw[black]  (1,-1) circle (0pt);
\end{tikzpicture}
\caption{The secondary fan for a Hirzebruch surface $\cH_3$ has two maximal
chambers: one corresponding to $\cH_3$ and the other to $\PP(1,1,3)$.  The
Bondal--Thomsen collection $\Theta$ consists of the degrees in a half-open
zonotope $Z$ determined by the degrees of the variables.   In this
example, $\Theta$ consists of the $6$ degrees $-d_i$, with $d_i$ as marked in
$-Z$.}
\label{fig:HirzSecondaryFan}
\end{figure}
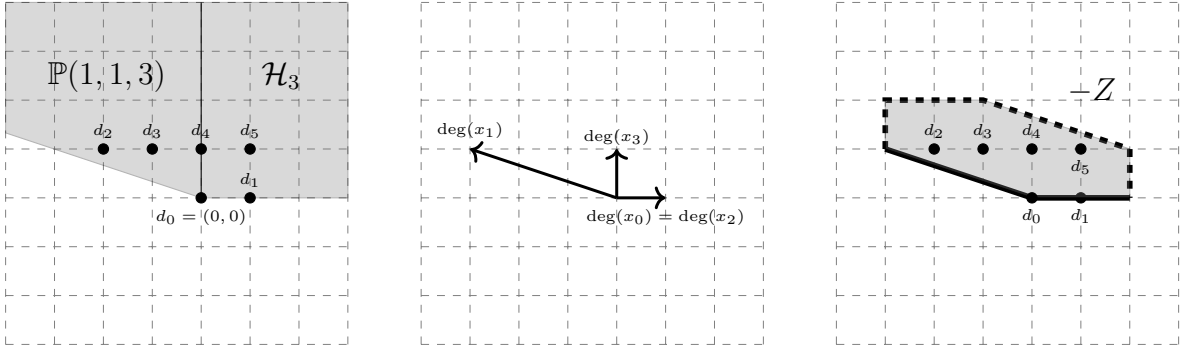

\begin{example}\label{ex:hirzGKZ}
Let $X=\cH_3$ be a Hirzebruch surface of type $3$. Its secondary fan has two maximal chambers (see Figure~\ref{fig:HirzSecondaryFan}): one corresponding to $\cH_3$, and the other to the weighted projective stack $\PP(1,1,3)$.
After choosing $\tsX$ (see \cref{subsec:constructionOfWX}), $D_{\Cox}$ is generated by pullbacks of line bundles from $\cH_3$ and $\PP(1,1,3)$.
Some of these ``glue,'' for example, the pullbacks of $\cO_{\cH_3}(0,d)$ and $\cO_{\PP(1,1,3)}(3d)$ agree for all $d\in \ZZ$.  For other bundles, the correspondence is more subtle. For instance, the Fourier--Mukai transform of the structure sheaf of the exceptional curve on $\cH_3$ is supported on the stacky point of $\PP(1,1,3)$.

The Bondal--Thomsen collection is illustrated in
Figure~\ref{fig:HirzSecondaryFan}. The corresponding bundles generate, but do
not form an exceptional collection, for either $D(\cH_3)$ or $D(\PP(1,1,3))$.
However, Theorem~\ref{thm:main1} implies that these become an exceptional
collection when $D(\cH_3)$ and $D(\PP(1,1,3))$ are glued to form $D_{\Cox}$. Note
that we must assign to each element $-d\in \Theta$ an appropriate object in
$D_{\Cox}$, and we do this according to where the image of $d$ lies in
$\Sigma_{GKZ}$. For instance, referring to Figure~\ref{fig:HirzSecondaryFan},
$d_5$ lies in the chamber for $\cH_3$; thus, we define $\cO_{\Cox}(-d_5)$ as
the pullback of $\cO_{\cH_3}(-d_5)$. Definition~\ref{defn:ThetaCox} provides
the general recipe.
\end{example}

Although it is generally {\em not} possible to choose some $\tsX$ such
that $D_{\Cox} = D(\tsX)$, $D_{\Cox}$ does behave in many ways like the
derived category of a mythical smooth toric stack associated to the ring $S$.
For instance, each $\pi_{i*}\colon D_{\Cox} \to D(\sX_i)$ is a categorical localization, just like
pushforwards via birational toric morphisms. We also have:
\begin{prop}\label{prop:dgThm}
For a semiprojective toric variety $X$, $D_{\Cox}(X)$ is 
self-dual, homologically smooth, 
and of Rouquier dimension equal to $\dim X$. Moreover, if $X$ is projective,
then $D_{\Cox}(X)$ is also proper.
\end{prop}
\noindent We caution the reader, however, that certain natural categories can appear
similar to $D_{\Cox}$ but are not, in fact, equivalent; see
Remark~\ref{rmk:caveats} for a discussion of some subtleties.

\smallskip

As $D_{\Cox}$ incorporates information about
$\Sigma_{GKZ}$, theorems about $D_{\Cox}$ yield uniform results for those
varieties. For instance, a curious feature of the Hanlon--Hicks--Lazarev
resolution of the diagonal for $X$ from \cite{HHL} is that the construction only involves the rays of the fan of $X$, making no use whatsoever of  the higher
dimensional cones.  In other words, this resolution 
depends
solely on the Cox ring.  The following theorem explains this curiosity. 

\begin{thm}\label{thm:resDiagCoxVague}
 The Hanlon--Hicks--Lazarev resolution of the diagonal of $X$ lifts to a complex $\HHL$ in $D_{\Cox}$ with the following properties:
\vspace*{-2.25mm}
\begin{enumerate}[noitemsep]
\item $\HHL$ is a resolution of the diagonal for $D_{\Cox}$ in the sense
that the Fourier--Mukai transform $D_{\Cox}\to D_{\Cox}$ with kernel $\HHL$
is naturally isomorphic to the identity, and
\item  The derived pushforward $(\pi_i \times\pi_i)_{\ast}\HHL$ is homotopic
to the
Hanlon--Hicks--Lazarev resolution of the diagonal for $\sX_i$ for all $1
\leq i \leq r$.
\end{enumerate}
\end{thm}
\noindent 
The Hanlon--Hicks--Lazarev resolutions are uniform across all
birational models precisely because they are the restrictions (via
$\pi_{i\ast}$) of a single resolution of the diagonal. 
A sharper statement is provided in \cref{sec:generation}. We note that an
assumption on the characteristic of the field appears in \cite{HHL}. However,
this is extraneous; see \cref{rmk:characteristic}.

A second example in this vein is \cref{thm:HHL_window}, which demonstrates how
the resolution above provides window categories in the derived category of
graded modules over the Cox ring uniformly across all chambers. These window
categories glue to the homotopy category of the graded modules $S(d)$ for $d \in
\Theta$, as described in \cref{thm:windows_glued}.

A third example arises from the theory of noncommutative or categorical resolutions.  Recall that $X$ is any semiprojective toric variety.  Let $\cT$ be the reflexive sheaf $\bigoplus_{-d\in\Theta} \cO_{X}(-d)$, and consider the underived endomorphism algebra $A_\Theta=\Hom_X^0(\cT,\cT)$.
\begin{thm}\label{thm:noncomm}
The endomorphism algebra $A_\Theta$ is a noncommutative resolution for $X$. The algebra $A_{\Theta}$ is uniform for any $X$ with the same Cox ring, that is, for any semiprojective toric variety whose fan has the same rays as $X$. 
\end{thm}
\noindent Noncommutative resolutions of affine schemes were introduced by Van den Bergh in~\cite{van-den-bergh-nccr}, and in the affine case, \cref{thm:noncomm} recovers the noncommutative resolutions of affine toric varieties given by \v{S}penko--Van den Bergh~\cite{spenko-van-den-bergh-inventiones}*{Proposition 1.3.6} and Faber--Muller--Smith~\cite{faber-muller-smith}.  For the non-affine case, one can either extend the Van den Bergh definition in the natural way or consider categorical resolutions in the sense of~\cite{kuzentsov-lefschetz} (see \cref{thm:noncommMorePrecise} below).  These resolutions are rarely crepant (see Remark~\ref{rmk:crepant}).  
On the other hand, the noncommutative algebra $A_\Theta$ is well understood as a quiver with relations.  In fact, it has a topological description as the path algebra of a quiver embedded in the mirror torus up to homotopy relations,  see \cites{Bondal, FH}.  Furthermore, minimal resolutions of modules over this algebra can also be interpreted topologically \cite{FS}.

A surprising feature of Theorem~\ref{thm:noncomm}---even when $X$ is smooth---is
that we are working with the {\em underived} endomorphism algebra $A_{\Theta}$.
In most cases, $\cT$ is not a tilting bundle for $X$ itself. Nevertheless, there is still a fully faithful functor into $D(A_\Theta)$ provided by $\mathcal E\mapsto \Hom_{D_{\Cox}}(\cT_{\Cox}, \pi^*\cE)$ (see \cref{subsec:noncomm}).

As a fourth example, $D_{\Cox}$ can be viewed from the perspective of multigraded $S$-modules.  Each $\sX_i$ corresponds to a pair $(S,B_i)$, where $B_i$ is an irrelevant ideal.  As $D_{\Cox}$ combines all of the $D(\sX_i)$, it should provide methods for using (complexes of) $S$-modules to {\em simultaneously} study all of the toric varieties from $\Sigma_{GKZ}$.  This is the content of \cref{subsec:BTmonads}, where we introduce and study analogues of Beilinson monads based on $\Theta$.

We  also obtain a sharpened version of Bondal--Thomsen generation.  
Bondal's~\cite{Bondal} inspired followup work 
proving that the Bondal--Thomsen collection $\Theta$ generates $D(X)$ (see~\cites{BallardEtAl2019ToricFrobenius,prabhu-naik,uehara} for partial results and~\cites{FH,HHL} for a full proof). We sharpen these results, showing that by incorporating data about a wall of $\Sigma_{GKZ}$, one can generate $D(X)$ with a smaller subset of $\Theta$. 
\subsection{Symplectic motivation}\label{subsec:symplectic}
While our results, methods, and applications are algebraic, 
some of the motivation for our definition of $D_{\Cox}$ comes from the symplectic side of mirror symmetry, as we  now explain.
There are at least two approaches to homological mirror symmetry for toric varieties in the literature: one is to pass through a microlocal sheaf category initiated in \cites{FLTZ11, FLTZ12, FLTZ14} and completed in \cite{Kuw20}, and the other is to work directly with Floer theory and Fukaya categories as in \cites{abouzaid2006homogeneous, abouzaid2009morse,HHfunctoriality}.
These are tied together by the main result of \cite{GPS}, which implies that both categories on the symplectic side can be modeled with a partially wrapped Fukaya category.
Namely, there is an (appropriately derived) equivalence of categories between each $D(\sX_i)$ and a partially wrapped Fukaya category $\cW(T^*T^{\dim \sX_i}, \mff_i)$, where $\mff_i$ is a closed subset of the boundary at infinity introduced in \cite{FLTZ11}.
Objects in this Fukaya category are supported on exact cylindrical at infinity Lagrangian submanifolds that do not intersect the stop $\mff_i$ at infinity, and morphisms involve Lagrangian Floer cohomology of certain large perturbations of these submanifolds that never cross $\mff_{i}$.\footnote{While Fukaya categories are $A_\infty$-categories, this is not relevant for this motivational discussion.}

Algebraic constructions are often more rigid than analogous constructions in topology and geometry, and this is the case here; there is flexibility on the symplectic side that is not obviously present on the algebraic side.
For example, one can define a wrapped Fukaya category by simply taking the union of all of the stop data:  $\cW(T^*T^{\dim \sX_i}, \bigcup_i \mff_i)$, see Figure~\ref{fig:stop}.
Working out examples suggested that this category had nice properties, and it led us to ask: is there a natural algebraic analogue of this ``union of stop data'' category?

\begin{figure}[h]
\begin{subfigure}{2.1in}
    \centering

\begin{tikzpicture}[scale=.85]
\draw[fill=gray, fill opacity=0.1] (0,0) rectangle (4,4);
\coordinate (Start1) at (0,0);
\coordinate (End1) at (4,1.33);
\coordinate (Start2) at (0,1.33);
\coordinate (End2) at (4,2.66);
\coordinate (Start3) at (0,2.66);
\coordinate (End3) at (4,4);
\foreach \startpoint/\endpoint in {Start1/End1, Start2/End2, Start3/End3} {
    \draw (\startpoint) -- (\endpoint);
    \foreach \i in {1,...,20} {
        \draw[color=blue] ($ (\startpoint)!\i/20!(\endpoint) $) -- ++(-1/15,3/15);
    }
}
\foreach \i in {1,...,20} {
  \draw[color=blue] (\i*0.2,0) -- ++(0,0.2);
  \draw[color=blue] (\i*0.2,0) -- ++(0,-0.2);
}
\foreach \i in {1,...,20} {
  \draw[color=blue] (\i*0.2,4) -- ++(0,-0.2);
  \draw[color=blue] (\i*0.2,4) -- ++(0,0.2);
}
\foreach \i in {1,...,20} {
  \draw[color=blue] (0,\i*0.2) -- ++(0.2,0);
}
\foreach \i in {1,...,20} {
  \draw[color=blue] (4,\i*0.2) -- ++(0.2,0);
}
\def\numlines{10} 
\def\maxangle{360} 
\def\linelength{0.4cm} 


\foreach \point in {Start1} {
\foreach \counter in {0, ..., \numlines} {
    \def\angle{\maxangle/\numlines*\counter}
    \draw[color=blue] (\point) -- (\angle:\linelength);
}
}
\end{tikzpicture}
\captionsetup{margin=0in,width=2in,font=small,justification=centering}
   \caption{$\cH_3$ }
    \label{fig:FF3}
\end{subfigure}
\begin{subfigure}{2.1in}
    \centering
        \begin{tikzpicture}[scale=.85]
\draw[fill=gray, fill opacity=0.1] (0,0) rectangle (4,4);

\coordinate (Start1) at (0,0);
\coordinate (End1) at (4,1.33);
\coordinate (Start2) at (0,1.33);
\coordinate (End2) at (4,2.66);
\coordinate (Start3) at (0,2.66);
\coordinate (End3) at (4,4);

\foreach \startpoint/\endpoint in {Start1/End1, Start2/End2, Start3/End3} {
    \draw (\startpoint) -- (\endpoint);
    \foreach \i in {1,...,20} {
        \draw[color=blue] ($ (\startpoint)!\i/20!(\endpoint) $) -- ++(-1/15,3/15);
    }
}

\foreach \i in {1,...,20} {
  \draw[color=blue] (\i*0.2,0) -- ++(0,0.2);
  \draw[color=blue] (\i*0.2,0) -- ++(0,-0.2);
}
\foreach \i in {1,...,20} {
  \draw[color=blue] (\i*0.2,4) -- ++(0,-0.2);
  \draw[color=blue] (\i*0.2,4) -- ++(0,0.2);
}
\foreach \i in {1,...,20} {
  \draw[color=blue] (0,\i*0.2) -- ++(0.2,0);
}
\foreach \i in {1,...,20} {
  \draw[color=blue] (4,\i*0.2) -- ++(0.2,0);
}
\def\maxangle{360}
\def\subdivision{10}
\def\linelength{0.4cm} 

\foreach \point in {Start1} {
\foreach \counter in {0, ..., 20} {
    \draw[color=blue] (\point) -- (\maxangle/\subdivision*\counter:\linelength);
} }

\foreach \point in {Start2, Start3} {
\foreach \counter in {0, ..., 4} {
    \draw[color=red] (\point) -- ++(\maxangle/\subdivision*\counter:\linelength);
}
}

\end{tikzpicture}
\captionsetup{margin=0in,width=2in,font=small,justification=centering}
    \caption{$D_{\Cox}$}
    \label{fig:glued}
\end{subfigure}
\begin{subfigure}{2.1in}
\centering
\begin{tikzpicture}[scale=.85]
\draw[fill=gray, fill opacity=0.1] (0,0) rectangle (4,4);

\coordinate (Start1) at (0,0);
\coordinate (End1) at (4,1.33);
\coordinate (Start2) at (0,1.33);
\coordinate (End2) at (4,2.66);
\coordinate (Start3) at (0,2.66);
\coordinate (End3) at (4,4);

\foreach \startpoint/\endpoint in {Start1/End1, Start2/End2, Start3/End3} {
    \draw (\startpoint) -- (\endpoint);
    \foreach \i in {1,...,20} {
        \draw[color=red] ($ (\startpoint)!\i/20!(\endpoint) $) -- ++(-1/15,3/15);
    }
}

\foreach \i in {1,...,20} {
  \draw[color=red] (\i*0.2,0) -- ++(0,-0.2);
}
\foreach \i in {1,...,20} {
  \draw[color=red] (\i*0.2,4) -- ++(0,-0.2);
}
\foreach \i in {1,...,20} {
  \draw[color=red] (0,\i*0.2) -- ++(0.2,0);
}
\foreach \i in {1,...,20} {
  \draw[color=red] (4,\i*0.2) -- ++(0.2,0);
}
\def\maxangle{360}
\def\subdivision{10}
\def\linelength{0.4cm} 

\foreach \point in {Start1} {
\foreach \counter in {0, ..., 20} {
    \draw[color=red] (\point) -- (\maxangle/\subdivision*\counter:\linelength);
} }

\foreach \point in {Start2, Start3} {
\foreach \counter in {0, ..., 4} {
    \draw[color=red] (\point) -- ++(\maxangle/\subdivision*\counter:\linelength);
}
}
\end{tikzpicture}
\captionsetup{margin=0in,width=2in,font=small,justification=centering}
    \caption{$\PP(1,1,3)$ }
    \label{fig:PP113}
\end{subfigure}
\caption{The pictures at left and right indicate the boundary conditions for Lagrangians in the cotangent bundle of a torus, drawn as partial conormals to a stratification of the torus, in the correspondence between derived categories of toric varieties and wrapped Fukaya categories.  The union of this ``stop data'' yields a Fukaya category with no obvious algebraic analogue.  This motivated our definition of the Cox category. }
\label{fig:stop}
\end{figure}
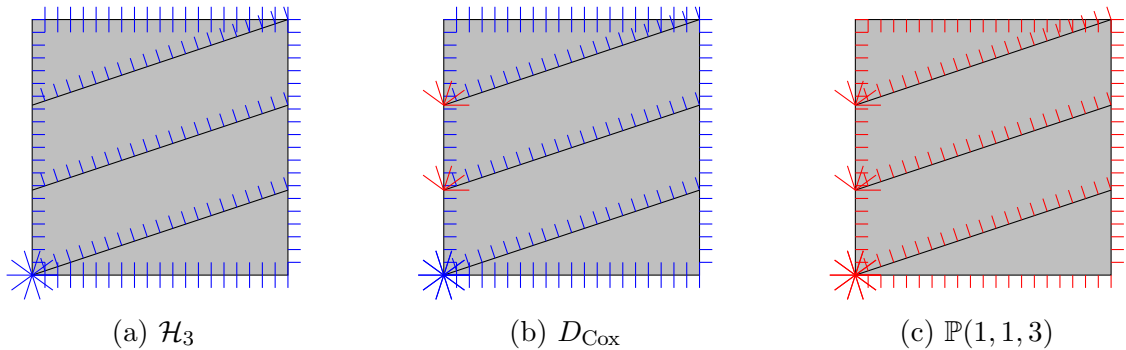

Thus, our definition of $D_{\Cox}$ was motivated by a desire to mimic a manipulation on the symplectic side (though even with that motivation in mind, finding the corresponding algebraic construction is not immediate). 
In addition, some of our arguments were initially inspired by techniques developed in \cite{ganatra2023sectorial}.
We also note that \cite{bai2023rouquier} was a first instance of applying flexible constructions in wrapped Fukaya categories to problems of algebraic origin.
But we emphasize: the symplectic connection is purely motivational, and the results and methods in this work are algebraic.

A parallel motivation came from the commutative algebra
perspective, though with a less clear road map.  Each category $D(\sX_i)$
corresponds to a derived category of graded $S$-modules, modulo the modules
supported on an irrelevant ideal that depends on the particular $\sX_i$.  One
should seek out a category that maintains the geometric content of the
$D(\sX_i)$ but depends only on the Cox ring $S$ and not the additional
data of an irrelevant ideal. \cref{rmk:caveats} addresses some of the challenges in realizing this approach. 

\subsection[Outline of the proof of Theorem A]{Outline of the proof of \cref{thm:main1}}\label{subsec:methods}

After we establish basic facts about $D_{\Cox}$,
the proof of our main result amounts to two major steps.  
The first is to study how the Bondal--Thomsen collection behaves with respect to the Fourier--Mukai transforms $\Phi_{ij}\colon D(\sX_i)\to D(\sX_j)$, as introduced shortly after Definition~\ref{defn:DCox}.

\begin{lemma}[$\Theta$-Transform Lemma]\label{lem:PushPull}
Let $-d\in \Theta_X$ be an element whose image in $\Sigma_{GKZ}$ lies in the chamber corresponding to $\sX_i$.  For any $j$,  $\Phi_{ij}(\cO_{\sX_i}(-d)) = \cO_{\sX_j}(-d)$.
\end{lemma}
This is a subtle property, which fails if one alters the hypotheses even
slightly.  
The challenge has nothing to do with stackiness, as the statement seems equally subtle when both $\sX_i$ and $\sX_j$ are smooth toric varieties.
It does not appear to easily follow from known vanishing results from toric or birational geometry.
\cref{lem:PushPull} captures a unique feature of the Bondal--Thomsen collection and its behavior under birational Fourier--Mukai transforms and so provides the foundation for our understanding of $D_{\Cox}$.  

The $\Theta$-Transform Lemma almost immediately implies that the objects $\cO_{\Cox}(-d)$ form a strong exceptional collection when $X$ is projective, and that they sum to a tilting bundle more generally.
Namely, we must compute $\RHom(\cO_{\Cox}(-d), \cO_{\Cox}(-d'))$ for two elements $-d,-d' \in \Theta$.
By adjunction and the $\Theta$-Transform Lemma, this reduces to computing
\[ \RHom(\cO_{\Cox}(-d), \cO_{\Cox}(-d'))
\cong \RHom_{\sX_i}(\cO_{\sX_i}(-d), \cO_{\sX_i}(-d'))
\cong H^*(\sX_i, \cO_{\sX_i}(d-d')).
\]
The element $d$ lies in the nef cone of $\sX_i$ and since $-d'$ is an element of $\Theta$, it can be written as a toric $\QQ$-divisor with all coefficients in $(-1,0]$.  The vanishing of the higher cohomology groups then follows from a stacky version of Demazure Vanishing for $\QQ$-divisors (\cref{thm:demazure}).  

\smallskip

Fullness of the Bondal--Thomsen collection for each $\sX_i$ was proven in \cite{HHL}*{Corollary D} and \cite{FH}; the second step is to extend those results to show that $\Theta$ generates $D_{\Cox}$.  For this, we combine the fact that $\Theta_{\tsX}$ generates $D(\tsX)$ with \cref{lem:PushPull} and explicit computations regarding the behavior of Bondal--Thomsen bundles with respect to pushforward.  

\subsection{Related results}\label{subsec:relatedResults}
We briefly discuss how our construction and resxults relate to some previous constructions and results in the literature. 
As noted in \cref{subsec:symplectic}, our construction was partially motivated by the work \cites{ganatra2020covariantly,ganatra2023sectorial,GPS} and other aspects of partially wrapped Fukaya categories.

The construction in \cite{CastravetTevelevI} on $D(\overline{M}_{0,n})$ is
similar to the Cox category, as they use pullbacks $\pi\colon \overline{M}_{0,n}
\to \overline{M}_{0,m}$ to provide a full $S_n$-stable exceptional collection, and they develop a general inclusion/exclusion principle for admissible categories \cite{CastravetTevelevI}*{Lemma 3.6}.  
However, their results are not
applicable here, as Tor-independence fails to hold in our setting.

There is a natural link between the elements of $\Theta$ and the collection
introduced by \v{S}penko--Van den Bergh in \cite{spenko-van-den-bergh-inventiones}, and utilized by
Halpern-Leistner--Sam in \cite{HLS}, of the dominant integral points in a
generic perturbation of a half zonotope for a quasi-symmetric representation of
a reductive linear algebraic group. However, in the toric setting, neither
result easily implies the other.  \v{S}penko--Van den Bergh also study
pullbacks of special vector bundles under Kirwan resolutions, which can dominate
all GIT quotients, in \cite{SVDB-KirwanRes}. Their criteria only works
for smaller collections than $\Theta$. For example,
for an Atiyah flop as in Example~\ref{ex:atiyahFlop}, it applies to $R, R(1)$
but not to $R(-1), R, R(1)$.

Window categories also provide a framework for embedding derived categories of toric varieties into an ambient category, specifically, the derived category $D(S)$ of $\Cl(X)$-graded $S$-modules.  These window categories have been widely used in the study of derived categories of (local) GIT quotients. We show in Section~\ref{sec:windows} that $D_{\Cox}$ is equivalent to the 
gluing of natural window categories associated to each 
GKZ chamber. 

Previous literature on King's conjecture includes the cases where $X$ is smooth, toric, and Fano of dimension at most 4 \cites{borisov-hua,bernardi-tirabassi,uehara,prabhu-naik},  where $X$ is smooth of Picard rank at most 2 \cite{costa-miro-roig0}, where $X$ is in some particular families of Picard rank 3 \cites{dey-lason-michalek,lason-michalek}, and more~\cites{costa-miro-roig,costa-miro-roig2,borisov-wang}.  
There are also many more papers inspired by King's Conjecture as mentioned in the first paragraph of the introduction.

The idea of gluing together all the toric varieties in the secondary fan has appeared previously in the context of hypertoric varieties and (topological) cohomology rings \cite{proudfoot2011all}. Although these ideas align philosophically with our definition of the Cox category, we do not currently know how to make a precise mathematical connection.

In a broader context, our results fit into distinct streams of
efforts to better understand derived categories of toric varieties---including
the special role of $\Theta$---in algebraic geometry, commutative algebra, and
homological mirror symmetry. In algebraic geometry, a number of works
investigated the Rouquier dimension of derived categories of toric varieties and
whether the Bondal--Thomsen collection generates $D(X)$. Within commutative
algebra, our results fit into the recent burst of activity on extending
homological results from a standard graded polynomial ring to a multigraded
setting. Homological mirror symmetry for toric varieties is particularly
well-studied due to the central role of toric varieties in the subject. For
instance, large complex structure limit degenerations of Calabi--Yau varieties
can be glued from toric varieties along toric strata. Our work builds on and
extends concrete descriptions of derived categories of toric varieties arising
from work on homological mirror symmetry.

\subsection{Organization}
This paper is organized as follows. Section \ref{sec:background} establishes notation and provides preliminary material. 
Namely, \cref{subsec:basics} reviews the secondary fan of a toric variety; \cref{subsec:cohomology} transfers standard techniques for computing cohomology of toric varieties to the stacks of interest in this paper; \cref{subsec:BTdefinitions} defines and recalls fundamental properties of the Bondal--Thomsen collection; and \cref{subsec:BTGKZ} spells out how we identify the Bondal--Thomsen bundles of $X$ with those of the other spaces in the secondary fan.

Sections \ref{sec:CoxCategory}--\ref{sec:generation} cover the main results.  
Specifically, \cref{subsec:constructionOfWX} gives a construction of a common simplicial refinement of the $\sX_i$ that is then used in \cref{subsec:CoxDefinitions} to define the Cox category $D_{\Cox}$ and give some first properties.
\cref{sec:PushPull} addresses the $\Theta$-Transform Lemma (\cref{lem:PushPull}).
In particular, \cref{subsec:BTforDCox} defines the Bondal--Thomsen collection in the Cox category. 
\cref{subsec:examples} gives a few examples of the Cox category and the $\Theta$-Transform Lemma.
The proof of \cref{lem:PushPull} is confined to \cref{subsec:proofofthetatrans}, and \cref{subsec:thetacomputations} deduces some consequences of that lemma.
Next, \cref{subsec:generation} proves that the Bondal--Thomsen collection generates the Cox category.
Finally, \cref{sec:proofsMainResults} ties all the pieces together into a proof of the main result.

\cref{sec:diagonal} upgrades the generation statement to an explicit short resolution of the diagonal for the Cox category. 
\cref{sec:applications} covers a number of applications and extensions of our
main results, as well as connections with related topics, including window
categories (\cref{sec:windows}), noncommutative resolutions
(\cref{subsec:noncomm}), Bondal--Thomsen monads (\cref{subsec:BTmonads}), and
sharpened generation (\cref{sec:sharpenedBondal}).
Finally, in Appendix \ref{sec:grothendieckconstr}, we show that the Cox category can also be obtained via the Grothendieck construction.

\section{Preliminaries}\label{sec:background}
This section serves two purposes.  First, we establish our notation.  Second, we collect facts about toric varieties and stacks that are needed to establish our results.  In many cases, we need a version of a statement that is not clearly stated elsewhere in the literature.  For instance, standard references like~\cite{CLSToricVarieties} do not cover toric stacks, and so we provide independent proofs of the necessary cohomological vanishing theorems.  The subsections of this section are largely modular and can be read independently on an as-needed basis.

\cref{subsec:basics} recalls some fundamentals of the birational geometry of toric varieties and the secondary fan, following~\cite{CLSToricVarieties}*{Chapters 14--15}. In \cref{subsec:cohomology}, we summarize how to compute cohomology on the toric stacks of interest. \cref{subsec:BTdefinitions} defines the Bondal--Thomsen collection and collects some of its basic properties.  
In \cref{subsec:BTGKZ}, we discuss how to realize modules over the Cox ring as sheaves on all the varieties appearing in the secondary fan, paying particular attention to elements of the Bondal--Thomsen collection.

\smallskip

Throughout, we work over an arbitrary algebraically closed field $\Bbbk$.
As in Theorem~\ref{thm:main1}, $X$ will always denote a semiprojective, normal toric variety, though we remark that our main results about the Cox category $D_{\Cox}(X)$ depend only on the Cox ring.  Since the Cox ring of $X$ is unchanged if one replaces the fan of $X$ by a simplicial refinement with the same rays as in~\cite{CLSToricVarieties}*{Proposition 11.1.7}, little is lost if one focuses on the case where $X$ is simplicial.
We use $N$ to denote a lattice and $M$ its dual lattice. Given a fan $\Sigma$ on $N_\RR = N \otimes_\ZZ \RR$, we use $X_\Sigma$ to denote the corresponding toric variety and $T$ to denote the algebraic torus $N \otimes \GG_m$.
We write $\Sigma(k)$ for the set of $k$-dimensional cones in $\Sigma$.
For any $\rho \in \Sigma(1)$, let $u_\rho$ be the primitive generator of the ray and $D_\rho$ the corresponding torus invariant divisor.

Before going further into the preliminaries, we also record the following conventions.

\begin{convention} \label{conv:functors}
We will denote bounded derived categories as $D(-)$ and homotopy categories 
$K(-)$. For example, we will use
$D(X)$ to denote the bounded derived category of coherent sheaves on $X$.
Similarly, categories of modules will always be bounded and graded.  For
example, the Cox ring $S$ of a toric variety $X$ is naturally $\Cl(X)$-graded,
and we will write $D(S)$ for the bounded derived category of finitely generated,
$\Cl(X)$-graded $S$-modules.
For a graded ring $S$ and a subset $\Omega$ of degrees, 
$K_\Omega(S)$ will denote the homotopy category of complexes whose components 
are summands of free modules concentrated in $\Omega$. 

Unless otherwise specified, all functors are derived. For example, we use
$\pi_*$ for the derived pushforward $\pi_*$ along a morphism $\pi$ and
$R^0\pi_*$ for the underived pushforward.
For subcategories $\mathcal A, \mathcal B$, of a triangulated category, we will
write $\langle \mathcal A, \mathcal B \rangle$ both for the smallest thick
subcategory generated by $\mathcal A$ and $\mathcal B$ and for a possible
semiorthogonal decomposition.
\end{convention}

\subsection{The birational geometry of toric varieties}\label{subsec:basics}
The birational geometry of a semiprojective toric variety $X$ is described by its secondary fan $\Sigma_{GKZ}(X)$ and by an associated GIT problem.
This is covered in detail in~\cite{CLSToricVarieties}*{Chapters 14--15}.
This subsection will establish our notation regarding the secondary fan and provide a terse summary of the main ideas and results that we will need.

The Cox ring of $X$ is $S\ce\bigoplus_{d\in \Cl(X)} H^0(X,\cO_X(d))$, and this is a polynomial ring $S=\Bbbk[x_\rho \mid \rho \in \Sigma_X(1)]$.
The ring $S$ comes with a natural grading by the class group, $\Cl(X)$.
The $\Cl(X)$-grading induces an action of a group $G$ on the affine space $\Spec(S)$ \cite{CLSToricVarieties}*{Section 5.1}, and this action leads to a GIT framework for describing the birational geometry of $X$ as detailed in \cite{CLSToricVarieties}*{Chapter 14}.
Here, we focus on a more combinatorial description.

The secondary fan $\Sigma_{GKZ}=\Sigma_{GKZ}(X)$ is a fan that is supported on the effective cone of $X$.  For each cone $\Gamma \in \Sigma_{GKZ}$, there is an associated toric variety, and it is convenient to describe these varieties in terms of \defi{generalized fans}, meaning a set of cones satisfying all properties of a fan except that the cones are not required to be strongly convex~\cite{CLSToricVarieties}*{Section 6.2}.

Every cone $\Gamma \in \Sigma_{GKZ}$ is the image in $\Cl(X)_\R$ of a subspace of $\RR^{\Sigma(1)}$ of the form
\begin{equation*}
\left\{ a \in \RR^{\Sigma(1)}
\ \bigg\vert\
\begin{array}{l} \text{there is a concave\footnotemark} \text{ support function } F \text{ on } \overline{\Sigma}_\Gamma \text{ such that} \\ F(u_\rho) = -a_\rho \text{ for all } \rho \not\in I_\Gamma \text{ and } F(u_\rho) \geq -a_\rho  \text{ for all } \rho \in \Sigma(1)  \end{array} \right\}
\end{equation*}
 for some $I_\Gamma \subset \Sigma(1)$
and generalized fan $\overline{\Sigma}_\Gamma$ in $N_\RR$.\footnotetext{We take the opposite convention of \cite{CLSToricVarieties} and say that the function $-x^2$ is concave.}\footnote{We have changed the notation from \cite{CLSToricVarieties} from $I_\emptyset$ to $I_\Gamma$ so that this set keeps track of the cone $\Gamma$.}
The associated $X_\Gamma$ is the toric variety of the generalized fan $\overline{\Sigma}_\Gamma$; see \cite{CLSToricVarieties}*{Section 6.2}.
In particular, if $L_\Gamma$ is the lineality space\footnote{The lineality space of a generalized fan is the minimal cone in the generalized fan, which must be a subspace and a face of all other cones. It is denoted by $\sigma_0$ in \cite{CLSToricVarieties}.} of $\overline{\Sigma}_\Gamma$, the fan of $X_\Gamma$ is
$\Sigma_\Gamma = \{ \sigma/L_\Gamma
\mid \sigma \in \overline{\Sigma}_\Gamma \}$,
which is a fan on $(N_\Gamma)_\RR$, where $N_\Gamma \ce N/N \cap L_\Gamma$.

The toric varieties $X_\Gamma$ are always \defi{semiprojective}.
By \cite{CLSToricVarieties}*{Proposition 7.2.9}, one of several equivalent conditions for a toric variety to be semiprojective is that it is the toric variety associated to a full-dimensional lattice polyhedron.
In terms of the fan, this characterization means that the support $|\Sigma_\Gamma|$ is full-dimensional and convex and that there exists a strictly convex piecewise linear function on $\Sigma_\Gamma$.

We enumerate the maximal cones of $\Sigma_{GKZ}$ as $\Gamma_1, \hdots, \Gamma_r$ with some arbitrary ordering. By \cite{CLSToricVarieties}*{Proposition 14.4.9}, the corresponding $\Sigma_1, \hdots, \Sigma_r$ on $N_\RR$ are honest (not generalized) fans that are simplicial, allowing for the following definition.

\begin{defn}[Cox construction] \label{defn:cox}
    For each maximal chamber $\Gamma_i$ of $\Sigma_{GKZ}$, the associated fan $\Sigma_i$ determines an irrelevant ideal $B_i\subseteq S$ (see~\cite{CLSToricVarieties}*{Chapter 14.5}).  We
    define $\sX_i$ as the smooth toric Deligne--Mumford stack $\sX_i \ce [\Spec(S) - V(B_i)/G]$.  Equivalently, $\sX_i$ can be defined in terms of a stacky fan (see \cref{subsec:cohomology} for a review), given by the fan  $\Sigma_i$ and the homomorphism $\beta\colon \ZZ^{\Sigma_i(1)} \to N$ satisfying $\beta(e_\rho) = u_\rho$ for all $\rho \in \Sigma_i(1)$.
\end{defn}

\begin{example}
\label{ex:P1} If $\Sigma$ is the fan for $\PP^1$, then Definition~\ref{defn:cox} yields that $\sX = [\AA^2 -\{0\} / \GG_m]$, which is isomorphic to $\PP^1$ because $\PP^1 = X_\Sigma$ is smooth.
If $X_{\Sigma}$ is simplicial but not smooth, then this toric stack from the Cox construction will be a smooth Deligne--Mumford stack~\cite{BorisovChenSmith2004}.
\end{example}

\begin{remark}\label{rmk:whyStacks}
    As is common in the investigation of derived categories of toric varieties (for example,~\cites{AKO,BFK_GIT,borisov-horja,BETate}), we focus on the stack $\sX_i$, instead of on the variety $X_i$, because the derived category of the stack is better behaved.  For instance, the stack is smooth while the variety may not be.
    As a consequence, when $X_i$ is projective, (an enhancement of) $D(\sX_i)$ is a smooth, proper dg-category, whereas $D(X_i)$ fails to be proper.

    In more concrete terms, each fan $\Sigma_i$ determines a corresponding irrelevant ideal $B_i\subseteq S$ (see~\cite{CLSToricVarieties}*{Chapter 14.5}), and $D(\sX_i)$ has a simple algebraic description.  
    Since $\sX_i$ equals the quotient stack $[\Spec(S) - V(B_i)/G]$, the derived category $D(\sX_i)$ equals the derived category of finitely generated, $\Cl(X)$-graded $S$-modules modulo the subcategory generated by $B_i$-torsion modules.
    In the smooth case, $\sX_i$ and $X_i$ are equal, and so the derived categories are equal as well.  In the simplicial case, the derived categories remain closely related.  While one needs to distinguish between Cartier and non-Cartier elements of $\Cl(X_i)$, one can still use $\sX_i$ to  effectively study most homological questions about $X_i$. 
\end{remark}

The secondary fan $\Sigma_{GKZ}$ also parametrizes toric morphisms among the corresponding varieties.
Specifically, if $\Gamma \subseteq \Gamma'$ are cones in $\Sigma_{GKZ}$, then there is an induced toric morphism $\pi\colon X_{\Gamma'}\to X_{\Gamma}$ coming from the fact that the generalized fan $\overline{\Sigma}_{\Gamma'}$ is a refinement of $\overline{\Sigma}_\Gamma$.

\begin{remark}
    It is natural to ask why we don't make \cref{defn:cox} for arbitrary cones of $\Sigma_{GKZ}$.  The short answer is that the lower dimensional cones of $\Sigma_{GKZ}$ can correspond to non-simplicial toric varieties, and the related stacks are more subtle in those cases.
\end{remark}

\subsection{Cohomology on toric Deligne--Mumford stacks} \label{subsec:cohomology}

The goal of this section is to take a number of standard results (e.g., Demazure Vanishing) and techniques (e.g., the use of support functions) about sheaf cohomology of line bundles, and to extend them from toric varieties to toric Deligne--Mumford stacks, as will be required for our main results.  
Instead of providing the most general version of each statement, we focus on versions that we will use.
We first briefly recall notation regarding such toric stacks. 

\begin{defn}\label{def:BCStoricStack}
In $\ZZ^{\Sigma(1)}$, we denote the standard basis vectors by $e_\rho$.  
Given a homomorphism $\beta\colon \ZZ^{\Sigma(1)} \to N$ with $\beta(e_\rho) = b_\rho u_\rho$ for some $b_\rho > 0$, let $\sX_{\Sigma,\beta}$ be the corresponding quotient stack of $X_{\widehat{\Sigma}}$ as defined in \cite{BorisovChenSmith2004}*{Section 2} (see also \cites{fantechi2010smooth,GS}).
By \cite{BorisovChenSmith2004}*{Proposition 3.2}, $\sX_{\Sigma, \beta}$ is a smooth Deligne--Mumford stack when $\Sigma$ is simplicial.
\end{defn}

Now fix a simplicial fan $\Sigma$ on $N_\RR$ and $\beta\colon \ZZ^{\Sigma(1)} \to N$ with $\beta(e_\rho) = b_\rho u_\rho$ for some $b_\rho>0$.
Let $\sX = \sX_{\Sigma, \beta}$ be the associated toric stack, and let $\beta^*\colon M \to \ZZ^{\Sigma(1)}$ be the dual map.
There is an exact sequence
\begin{equation} \label{eq:classgroup}
    M \to \ZZ^{\Sigma(1)} \to \Cl(\sX) \to 0, 
\end{equation}
where $\Cl(\sX)$ is the class group of $\sX$ and the first map is given by
\[ 
m \mapsto 
\sum_{\rho \in \Sigma(1)} \langle \beta^*m, e_\rho \rangle D_\rho  
= \sum_{\rho \in \Sigma(1)} \langle m, \beta(e_\rho) \rangle D_\rho  
= \sum_{\rho \in \Sigma(1)} b_\rho \langle m, u_\rho \rangle D_\rho.
\]
As in the previous equation, we will often conflate elements $a \in \ZZ^{\Sigma(1)}$ with torus invariant divisors $\sum a_\rho D_\rho$.
Going further, there is a commutative diagram
\begin{equation}\label{eqn:betastar}
    \begin{tikzcd}
M \arrow[r] \arrow[dr] & \ZZ^{\Sigma(1)} \arrow[r] & \Cl(\sX)  \\
 & \ZZ^{\Sigma (1)} \arrow[u, "\beta^*"] \arrow[r] & \Cl(X), \arrow[u]
\end{tikzcd}
\end{equation}
where $X =X_{\Sigma}$ is the toric variety of the fan $\Sigma$ and
\[ \beta^*\left( \sum_{\rho \in \Sigma(1)} a_{\rho} D_{\rho} \right) = \sum_{\rho \in \Sigma(1)} b_\rho a_{\rho} D_\rho. \]
This diagram allows us to compare the class group of $X$ with $\sX$, but there is even more.
In particular, since $\Sigma$ is simplicial, there is a fan $\widehat{\Sigma}$ in $\RR^{\Sigma(1)}$ given by the preimage of $\Sigma$ under $\beta_\RR$.  
Then, $\beta_\RR$ induces a combinatorial equivalence (cf. \cite{Kuw20}*{Condition 1.1})
between these fans and restricts to a homeomorphism on $|\widehat{\Sigma}|$, which allows us to understand cohomology on $\sX$ as follows.
For any $\widehat{D} = \sum a_\rho D_\rho$ on $\sX$, taking into account the equivariant structure, \cite{CLSToricVarieties}*{Theorem 9.1.3(a)} implies that
\begin{align*}
H^p ( \cO_{\sX} (\widehat{D}) ) 
&= \bigoplus_{m \in M} \widetilde{H}^{p-1}(V_{\widehat{D},m}^{\sX}), 
\\
\text{where}\quad  
V_{\widehat{D},m}^{\sX} 
&= \bigcup_{\sigma \in \Sigma} \mathrm{conv}\left(e_\rho \mid \rho \in \sigma(1) \text{ and } \langle m, \beta(e_\rho) \rangle < - a_\rho \right).  
\intertext{
However, $V_{\widehat{D},m}^{\sX}$ is homeomorphic via $\beta_\RR$ to
}
\beta_\RR \left(V_{\widehat{D},m}^{\sX} \right) 
&= \bigcup_{\sigma \in \Sigma} \mathrm{conv}\left(b_\rho u_\rho \mid \rho \in \sigma(1) \text{ and } b_\rho \langle m, u_\rho \rangle < - e_\rho \right).
\end{align*}
If $\widehat{D} = \beta^*D$, then $\beta_\RR\left(V_{\widehat{D},m}^{\sX}\right)$ is homeomorphic to $V_{D,m}^{X}$.
Applying \cite{CLSToricVarieties}*{Theorem 9.1.3(a)} again, the reduced cohomology of $V_{D,m}^{X}$ computes the cohomology in equivariant grading $m$ of $D$.
Therefore, the following holds.

\begin{prop} \label{prop:semicohomology}
    If $D$ is a torus-invariant Weil divisor on $X_\Sigma$ and $\beta^*$ is as in \eqref{eqn:betastar}, then 
    $H^p(\cO_{\sX_{\Sigma, \beta}} (\beta^* D)) \cong H^p (\cO_{X_\Sigma}(D))$
    for all $p$.
    \hfill $\square$
\end{prop}

In fact, \cref{prop:semicohomology} also follows from the general theory of good moduli morphisms of toric stacks in \cite{GS}.
We refer to \cite{GS}*{Section 6.1} for the relevant definitions regarding good moduli morphisms.
There is a natural toric morphism $\sX_{\Sigma, \beta} \to X_\Sigma$ and the following proposition immediately follows from \cite{GS}*{Corollary 6.5} and implies \cref{prop:semicohomology}.

\begin{prop} \label{prop:semigoodmoduli}
    The natural toric morphism $\sX_{\Sigma, \beta}\to X_{\Sigma}$ is a good moduli morphism.
\end{prop}

\cref{prop:semigoodmoduli} is particularly useful in our case of interest (\cref{defn:cox}), where every divisor on $\sX$ is in the image of $\beta^*$ as $b_\rho = 1$ for all $\rho$. It also allows classical vanishing results on toric varieties such as Demazure Vanishing to be extended.

\begin{thm}[Demazure Vanishing] \label{thm:demazure}
If $\sX = \sX_{\Sigma, \beta}$ is a smooth toric Deligne--Mumford stack, $X_\Sigma$ is semiprojective, and $D$ is a nef $\QQ$-Cartier divisor on $X_{\Sigma}$, then
$H^p( \cO_{\sX}(\beta^* \lfloor D \rfloor )) = 0$
for all $p > 0$.
\end{thm}

\subsection{The Bondal--Thomsen collection}\label{subsec:BTdefinitions}
We now define $\Theta_{\sX}$, within the class group $\Cl(\sX)$ of a toric stack $\sX = \sX_{\Sigma,\beta}$.  We will also prove some essential facts about cohomology and pushforward of elements of $\Theta_{\sX}$.

\begin{defn}\label{defn:BT}
    The \defi{Bondal--Thomsen collection} for $\sX$ is the set $\Theta_{\sX}$ of degrees $-d\in \Cl(\sX)$ that are equivalent to, for some $\theta \in M_\RR$, 
    \begin{equation} \label{eq:BT}
        \sum_{\rho \in \Sigma(1) } \lfloor \langle - \theta, \beta(e_\rho) \rangle \rfloor D_\rho. 
    \end{equation}
\end{defn}

At times, we will denote the class in $\Theta_\sX \subseteq \Cl(\sX)$ corresponding to $\theta \in M_\RR$ by $-d(\theta)$.
The element $-d(\theta)$ depends only on the image of $\theta$ in the torus $M_\RR/M$.
There is also a Bondal--Thomsen collection $\Theta_X \subset \Cl(X)$ for the toric variety $X = X_\Sigma$.
Namely, $\Theta_X$ is the set of all classes described by \eqref{eq:BT} with $\beta(e_\rho)$ replaced by $u_\rho$.

\begin{example}\label{ex:BTonPn}
The set $\Theta_{\PP^n}$ is precisely the Beilinson collection of line bundles $\cO_{\PP^n}(-d)$ for $0\leq d \leq n$.  This is perhaps most easily seen using the equivalent characterizations of \cref{prop:zonotheta} below.
\end{example}

\begin{remark}
    Our naming convention for $\Theta_{\sX}$ differs from that of \cite{HHL}, where it is called the Thomsen collection.
    We have done this in order to recognize that this set first appears in \cite{Thomsen} to describe the Frobenius pushforward of the structure sheaf on a toric variety (see \cref{rem:frobenius}) but also that $\Theta_{\sX}$ was described precisely by \eqref{eq:BT} in \cite{Bondal}, which first indicated the role of $\Theta_{\sX}$ in derived categories.
\end{remark}

There are various other equivalent ways in which the set $\Theta_\sX$ can be described.  It is easiest to geometrically visualize $\Theta_{\sX}$ and relate it to the secondary fan by viewing it as corresponding to lattice points in a zonotope in $\Cl(\sX)_\RR$ following \cite{achinger2015characterization}.

\begin{defn}\label{defn:zonotope}
We define the (partial) \defi{zonotope} $Z_{\sX}$ associated to $\sX$ to be the image of $(-1, 0]^{\Sigma(1)}$ under the map $\RR^{\Sigma(1)} \to \Cl(\sX)_\RR$ induced by \eqref{eq:classgroup}.
\end{defn}

The following proposition summarizes equivalent characterizations of the Bondal--Thomsen collection that we will use (cf. \cite{achinger2015characterization} and \cite{HHL}*{Section 5}).

\begin{prop} \label{prop:zonotheta}
The following subsets of $\Cl(\sX)$ are equal: 
\vspace*{-.5mm}
\begin{enumerate}[noitemsep]
\vspace*{-2mm}
   \item The Bondal--Thomsen collection $\Theta_\sX$.
   \item The collection of Weil divisors in $\Cl(\sX)$ linearly equivalent    in $\Cl(\sX)_\QQ$
to divisors of the form
   $
   \sum_{\rho \in \Sigma(1)} a_\rho D_\rho
   $
   with $-1 < a_\rho \leq 0$.
   \item Elements whose image in $\Cl(\sX)_{\QQ}$ is a lattice point of the zonotope $Z_\sX$.
   \hfill$\square$
\end{enumerate}
\end{prop}

\begin{example}\label{ex:HirzZonotope}
In our running example of the Hirzebruch surface $\cH_3$, the zonotope $Z_{\cH_3}$ is given by the partially-open polytope depicted below. The Bondal--Thomsen collection consists of the degrees corresponding to $-d_0, \dots, -d_5$, with the $d_i$ as illustrated in Figure~\ref{fig:HirzSecondaryFan}. 
\[
\begin{tikzpicture}[scale = .8]
\draw[step=1cm,gray,very thin] (4,1) grid (-3,-3);
\draw[gray, line width = .5mm] (2.91,-.97)--(0,0)--(-1.97,0);
\draw[gray, dashed, line width = .5mm] (3,-1) -- (3,-2)--(1,-2)--(-2,-1)--(-2,0);
\draw[fill = gray, opacity = .3] (0,0) --(3,-1) -- (3,-2)--(1,-2)--(-2,-1)--(-2,-0)--(0,0);
\draw[thick, gray] (3,-1) circle (3pt);
\draw[thick, gray] (-2,0) circle (3pt);
\draw[thick,gray] (3,-1) circle (3pt);
\filldraw[black, very thick]  (-1,1) circle (0pt);
\filldraw[black, very thick]  (0,0) circle (2pt);
\draw[black, very thick]  (0,0) circle (2pt);
\filldraw[black] (0,0) circle (0pt) node[anchor=south]{$(0,0)$};
\end{tikzpicture}
\qedhere
\]
\end{example}

\begin{defn} \label{defn:nef} We will call a divisor $A = \sum a_\rho D_\rho$ on $\sX$ \defi{numerically effective (nef)} if there is a concave piecewise linear function $F$ on $\Sigma$ such that $F(\beta(e_\rho)) = - a_\rho$ for all $\rho \in \Sigma(1)$.
\end{defn}

Note that in \cref{defn:nef} $F$ may not take integral values on the primitive generators of $\Sigma$.
However, any nef divisor $D$ on $X_{\Sigma}$ pulls back to a nef divisor on $\sX_{\Sigma, \beta}$.
We will often use the following fact about vanishing of higher cohomology for Bondal--Thomsen line bundles. The corresponding statement for toric varieties is a corollary of Demazure Vanishing for $\QQ$-divisors (see, for example, \cite{CLSToricVarieties}*{Theorem 9.3.5}).

\begin{lemma}\label{lem:BTvanishingCohomology}
Suppose that $|\Sigma|$ is convex,\footnote{We put the convexity assumption here for clarity as it is used in the proof.  As all toric varieties in this paper are assumed to be semiprojective, this assumption is always satisfied for us.}  $\Sigma$ is simplicial, and $A = \sum a_\rho D_\rho$ is a nef divisor on $\sX = \sX_{\Sigma, \beta}$. Then for any $-d\in \Theta_{\sX}$ and any $p>0$, $H^p(\sX, \cO_{\sX}(A-d))=0$.
\end{lemma}
\begin{proof}
    Choose $\theta \in M_\RR$ so that $-d = -d(\theta)$.
    Applying \cite{CLSToricVarieties}*{Theorem 9.1.3(a)} as in \cref{subsec:cohomology}, 
    \[
        H^p(\cO_{\sX}(A-d)) = \bigoplus_{m \in M} \widetilde{H}^{p-1}(V_{A- \theta,m}), 
    \]
\vspace*{-3mm}
    \begin{align*}
    \text{where}\quad 
        V_{A -\theta, M} &= \bigcup_{\sigma \in \Sigma} \mathrm{conv}\{ e_\rho \mid \rho \in \sigma(1) \text{ and } \langle m, \beta(e_\rho) \rangle < \lceil \langle \theta, \beta(e_\rho) \rangle \rceil - a_\rho \} \\
        &= \bigcup_{\sigma \in \Sigma} \mathrm{conv}\{ e_\rho \mid \rho \in \sigma(1) \text{ and } \langle m, \beta(e_\rho) \rangle < \langle \theta, \beta(e_\rho) \rangle + F(\beta(e_\rho))\},
    \end{align*}
    using that $\langle m, \beta(e_\rho) \rangle, a_\rho$ are in $\ZZ$ and letting $F$ be a concave support function for $A$. Then, $\beta_\RR$ maps $V_{A- \theta, M}$ homeomorphically (see \cref{subsec:cohomology}) onto the set
    \[ 
    \bigcup_{\sigma \in \Sigma} \mathrm{conv} \{ \beta(e_\rho) \mid \rho \in \sigma(1) \text{ and } \langle m, u_\rho \rangle < \langle \theta, u_\rho \rangle + F(u_\rho)) \}, 
    \]
    which, since the inequality $\langle m, u_\rho \rangle < \langle \theta, u_\rho \rangle + F(u_\rho)$ is invariant under positive scaling, is homotopic to
    \begin{equation} \label{eq:homotope}
    \left( \bigcup \{ \sigma \in \Sigma \mid \langle m - \theta, u \rangle < F(u) \text{ for all nonzero } u \in \sigma \right) \setminus \{0\}.
    \end{equation}
    Now, we observe that \eqref{eq:homotope} is a deformation retract of
    \[ 
    V_{F-\theta, m} = \{ u \in |\Sigma| \mid \langle m - \theta, u \rangle < F(u) \}, 
    \]
    which is a convex set since $F$ is concave.
    Namely, for any $u \in |\Sigma|$, there is a minimal cone in $\Sigma$ containing $u$ which has a maximal face $\sigma_u$ such that $\sigma_u\setminus\{0\} \in V_{F-\theta, m}$.
    Moreover, $\sigma_u$ must be nonempty if $u \in V_{F-\theta,m}$.
    Let $P_{\sigma_u}$ be the orthogonal projection onto $\sigma_u$ for some chosen inner product on $N_\RR$.
    Then, a deformation retraction $H\colon V_{F-\theta, m} \times [0,1] \to V_{F-\theta, m}$ is given by
    $H(u, t) =  P_{\sigma_u}(u) + (1-t)\left(u - P_{\sigma_u}(u)\right)$. 
    Putting this all together, $V_{A-\theta,m}$ is contractible when it is nonempty, yielding the desired result.
\end{proof}

\begin{remark} \label{rem:frobenius}
    On a toric stack in any characteristic, there is a toric Frobenius endomorphism $F_\ell$ induced by the linear map of fans given by multiplication by $\ell$.
    In \cite{Thomsen}, it is shown that $\Theta_X$ is the set of all possible summands of $(F_\ell)_* \cO_X$ for a smooth toric variety $X$ (see also \cites{bogvad1998splitting,achinger2015characterization}).
    Thomsen's result was extended to toric Deligne--Mumford stacks in \cite{ohkawa2013frobenius}.
    Yet another standard proof of 
    \cref{lem:BTvanishingCohomology} for varieties uses that $H^p(\cO_{\sX}(A-d))$ includes
    into $H^p\left(\cO_{\sX}(A) \otimes (F_\ell)_* \cO_{\sX}\right) \cong
    H^p\left(\cO_{\sX}(\ell A)\right) = 0$.
\end{remark}

We will need several results on how elements of $\Theta$ push forward.
These results are closely related, and they all follow from a more general statement about pushforwards under appropriately defined birational toric morphisms of toric stacks, but we have chosen to isolate the instances we use for clarity.
First, $\Theta$ pushes forward in the naive way from $\sX$ to $X$.

\begin{prop}[Coarse moduli pushforwards of $\Theta$] \label{prop:pushforwardBTcoarse}
   If $\Sigma$ is simplicial and $\pi\colon \sX_{\Sigma, \beta} \to X_\Sigma$ is the natural coarse moduli space map, then for all $\theta \in M_\RR$,
    \[ 
    \pi_* \cO_{\sX_{\Sigma,\beta}} (-d(\theta)) = \cO_{X_\Sigma}(-d(\theta)). 
    \]
\end{prop}
\begin{proof}
    The higher direct images all vanish as a consequence of \cref{lem:BTvanishingCohomology} with $A= 0$.
    Thus we must compute only $R^0\pi_* \cO_{\sX}(-d(\theta))$. 
    For any $\sigma \in \Sigma$, $\pi^{-1}(U_\sigma) = \sX_{\sigma, \beta|_{\ZZ^{\sigma(1)}}}$.
    Therefore 
    \begin{align*}
        H^0(U_\sigma, R^0\pi_* \cO_{\sX_{\sigma,\beta}} (-d(\theta)) ) &= H^0 \left( \sX_{\sigma, \beta|_{\ZZ^{\sigma(1)}}}, \cO_{\sX} (-d(\theta)) \right) \\
        &= \Bbbk\left\langle m \in M \mid \langle m, \beta(e_\rho) \rangle \geq \lceil \langle \theta, \beta(e_\rho) \rangle \rceil \text{ for all } \rho \in \sigma(1) \right\rangle \\
        &= \Bbbk\left\langle m \in M \mid\langle m, \beta(e_\rho) \rangle \geq  \langle \theta, \beta(e_\rho) \rangle \text{ for all } \rho \in \sigma(1) \right\rangle \\
        &= \Bbbk\left\langle m \in M \mid \langle m, u_\rho \rangle \geq  \langle \theta, u_\rho \rangle  \text{ for all } \rho \in \sigma(1) \right\rangle,
    \end{align*}
    using that $\langle m , \beta(e_\rho) \rangle$ is always an integer and that $\beta(e_\rho)$ is a positive multiple of $u_\rho$.
    Thus this calculation shows that for all $\sigma \in \Sigma$, 
    \[ 
    H^0(U_\sigma, R^0\pi_* \cO_{\sX_{\sigma,\beta}} (-d(\theta))) = H^0(U_\sigma,  \cO_{U_\sigma} (-d(\theta)))
    \]
    in a fashion that respects restriction, which is given by inclusion of polytopes.
\end{proof}

We also need to understand the pushforwards of Bondal--Thomsen elements under toric morphisms induced by refinements of fans.
We start with toric varieties, but due to their natural appearance in \cref{subsec:basics}, we do so in the context of generalized fans.
Note that if $\overline{\Sigma}$ is a generalized fan with lineality space $L$, then the character lattice of $X_{\overline{\Sigma}}$ is
\[
L^\perp \cap M = \{ m \in M \mid \langle m, u \rangle = 0 \text{ for all } u \in L \},
\]
and each $d\in-\Theta_{X_{\overline{\Sigma}}}$ can be represented via some $\theta \in L^\perp$ as
\begin{equation} \label{eq:BTgenfan}
    -d(\theta) := \sum_{\rho \in \overline{\Sigma}(1) } \lfloor \langle - \theta, u_\rho \rangle \rfloor D_\rho.
\end{equation}
Let $\overline{\Sigma}(1)$ denote the set of cones that correspond to rays in the quotient fan, and for each $\rho \in \overline{\Sigma}(1)$, let $u_\rho$ denote the primitive generator of the quotient ray.

\begin{prop}[Pushforwards of $\Theta$ along refinements]\label{prop:pushforwardBTvariety}
Suppose that $\Sigma'$ is a fan on $N_\RR$ refining a generalized fan $\overline{\Sigma}$ on $N_\RR$.
Let $X' = X_{\Sigma'}$, $X = X_{\overline{\Sigma}}$, and let $\pi\colon X' \to X$ be the induced toric morphism.
If $L$ is the lineality space of $\overline{\Sigma}$ and $\theta \in L^\perp + M$, then
\[ \pi_* \cO_{X'} (-d(\theta)) = \cO_{X}(-d(\theta)),\]
and the derived pushforward of $\cO_{X'}(-d(\theta))$ is zero for all $\theta \not\in L^\perp + M$.
\end{prop}
\begin{proof}
    For any $\sigma \in \overline{\Sigma}$, there is a fan $\sigma' \subset \Sigma'$ refining $\sigma$.
    The higher direct images vanish as a consequence of the analogue of \cref{lem:BTvanishingCohomology} with $A=0$ for toric varieties applied to the $X_{\sigma'}$.
    To compute $R^0\pi_* \cO_{X'}(-d(\theta))$, 
    \begin{align}
        H^0(U_\sigma, R^0\pi_* \cO_{X'} (-d(\theta))) &= H^0(X_{\sigma'}, \cO_{X'} (-d(\theta))) \notag \\
        &= \Bbbk \left\langle m \in M \mid \langle m, u_\rho \rangle \geq \lceil \langle \theta, u_\rho \rangle \rceil \text{ for all } \rho \in \sigma'(1) \right\rangle \notag \\
        &= \Bbbk \left \langle m \in M \mid \langle m, u_\rho \rangle \geq  \langle \theta, u_\rho \rangle  \text{ for all } \rho \in \sigma'(1) \right\rangle \label{eq: sections} 
    \end{align}
using that $\langle m, u_\rho \rangle$ is always an integer.

First, suppose that $\theta \not \in L^\perp +M$. 
That is, for any $m$, there exists $u \in L$ such that
$\langle \theta -m, u \rangle > 0$. 
In addition, $\theta \not \in L^\perp +M$ implies that $L \neq \{ 0 \}$; moreover, $\RR_{\geq 0}\cdot \sigma'(1)$ contains $L$. 
Thus, $u$ can be written as a nonnegative combination of $u_\rho$ with $\rho \in \sigma'(1)$, and the fact that $\langle \theta -m, u \rangle > 0$ implies that it is not possible to have $\langle m , u_\rho \rangle \geq \langle \theta,u_\rho \rangle$ for all $\rho \in \sigma'(1)$. 
We conclude that $\pi_* \cO_{X'} (-d(\theta))$ has no local sections on an affine cover.

Now we turn to the case that $\theta \in L^\perp + M$.
Since adding an element of $M$ does not affect the isomorphism class of the line bundle, assume that $\theta \in L^\perp$. 
If $m \not \in L^\perp$, then there exists $u \in L$ such that $\langle m, u \rangle < 0$, and by the same argument as above, it is impossible to have $\langle m, u_\rho \rangle \geq \langle \theta, u_\rho \rangle$ for all $\rho \in \sigma'(1)$. 
Thus, the nonzero sections on $U_\sigma$ must correspond to $m$ that lie in $L^\perp$.  Continuing from \eqref{eq: sections}, 
 \begin{align*}
    H^0(U_\sigma, R^0\pi_* \cO_{X'} (-d(\theta))) &= \Bbbk \left \langle m \in L^\perp \cap M \mid \langle m, u_\rho \rangle \geq  \langle \theta, u_\rho \rangle  \text{ for all } \rho \in \sigma'(1) \right \rangle \\
    & = \Bbbk \left \langle m \in L^\perp \cap M \mid \langle m, u_\rho \rangle \geq  \langle \theta, u_\rho \rangle  \text{ for all } \rho \in \sigma(1) \right \rangle, 
\end{align*}
where the second equality follows from the fact that the image of every primitive generator of $\sigma'(1)$ in the quotient by $L$ is a linear combination of primitive generators of $\sigma(1)$.
That is, our calculation shows that
$H^0(U_\sigma, R^0\pi_* \cO_{X'} (-d(\theta))) 
= H^0(U_\sigma, \cO_{X} (-d(\theta)))$ 
for all $\sigma \in \Sigma$, and in a fashion that respects restriction, which is given by inclusion of polytopes.
\end{proof}

We end this subsection with one more result on pushforwards of Bondal--Thomsen elements for the toric stacks under consideration.
We need an appropriate notion of refinement.

\begin{defn} \label{defn:refinement}
    Let $(\Sigma', \beta')$ and $(\Sigma, \beta)$ be stacky fans on $N_\RR$ in the sense of \cref{def:BCStoricStack}. 
    We say that $(\Sigma', \beta')$ is a \defi{stacky refinement} of $(\Sigma, \beta)$ if $\Sigma'$ is a refinement of $\Sigma$, and there is a homomorphism $\Phi\colon \ZZ^{\Sigma'(1)} \to \ZZ^{\Sigma(1)}$
    such that $\beta \circ \Phi = \beta'$.
\end{defn}

A stacky refinement induces a toric morphism $\sX_{\Sigma',\beta'} \to \sX_{\Sigma, \beta}$ by \cite{BorisovChenSmith2004}*{Remark 4.5} or by viewing a stacky refinement as a special case of a morphism of stacky fans in the sense of \cite{GS}.
Pushforwards of Bondal--Thomsen elements along a stacky refinement behave analogously to the previous two cases.  The proof of the following uses the same argument as in Propositions~\ref{prop:pushforwardBTcoarse} and~\ref{prop:pushforwardBTvariety}, and is omitted.

\begin{prop}[Pushforwards of $\Theta$ along stacky refinements] \label{prop:pushforwardBTstack}
Let $(\Sigma', \beta')$ be a stacky refinement of $(\Sigma, \beta)$ inducing a toric morphism $\pi\colon \sX' \to \sX$ of smooth toric DM stacks, then for all $\theta \in M_\RR$, 
        \[ 
        \hspace*{5.7cm} 
        \pi_* \cO_{\sX'} (-d(\theta)) = \cO_{\sX}(-d(\theta)).
        \hspace*{5.3cm} 
        \square
        \]
\end{prop}

\subsection{The Bondal--Thomsen collection and the GKZ Fan}\label{subsec:BTGKZ}
Recall our standard assumptions that $S$ is the Cox ring of $X$ and the chambers of $\Sigma_{GKZ}(X)$ correspond to $\sX_1, \dots, \sX_r$.  
By \eqref{eq:classgroup}, for each $i$, there is a natural map $\Cl(X)\to \Cl(\sX_i)$.  The map is surjective in general, and it is an isomorphism precisely when $\sX_i$ has the same rays as $X$.  Moreover, under this map, the Bondal--Thomsen collection $\Theta \ce \Theta_X$ surjects onto $\Theta_{\sX_i}$.  
Thus, given any $-d\in \Theta$, there is a natural corresponding element $-d\in\Theta_{\sX_i}$.  
We will generally omit the map of class groups from our notation.  
For instance, by a minor abuse of notation, we will write ``given $-d\in \Theta_X$, consider the bundle $\cO_{\sX_i}(-d)$.''

We next consider the relationship between an $S$-module and the corresponding sheaves on the $\sX_i$. 
In \cref{subsec:BTmonads}, we will go further and  consider sheaves on the varieties $X_{\Gamma}$ corresponding to cones $\Gamma \in \Sigma_{GKZ}(X)$, so we now summarize this construction.  If $P$ is a finitely generated $S$-module and $\Gamma$ is any cone in
$\Sigma_{GKZ}(X)$, then there is a natural coherent sheaf on $X_\Gamma$, which we denote by $P|_{X_{\Gamma}}$ and define as follows.\footnote{While this
sheaf is often denoted $\widetilde{P}$, we will want to distinguish between the sheaves on
the various $X_\Gamma$, which is why we use the notation $P|_{X_\Gamma}$.}   
We will largely mimic the construction and notation from \cite{CLSToricVarieties}*{Section 5.3}, but with two wrinkles due to the fact that we work with generalized fans.  
First, there may be some contracted rays which do not lie in any cone of the generalized fan $\overline{\Sigma}_\Gamma$ of $X_\Gamma$, and second, the generalized fan  may have a lineality space $L\subseteq N$. 
Write $\Sigma_{\Gamma}$ for the corresponding genuine fan of $X_\Gamma$.  
For a cone $\overline{\sigma} \in \overline{\Sigma}_{\Gamma}$, let $\sigma$ denote the corresponding cone in $\Sigma_{\Gamma}$.

Let $\overline{\sigma}$ be a cone in $\overline{\Sigma}_\Gamma$.  
We define 
$x^{\overline{\sigma}^{\mathsf{c}}}$ 
to be the product of the variables $x_\rho$, where $\rho\notin \overline{\sigma}$.  
Note that, if $x_\rho$ corresponds to a ray that has been contracted, then $x_\rho$ will also always divide  $x^{\overline{\sigma}^{\mathsf{c}}}$. 
And if there is a lineality space, then the corresponding variables will divide every $x^{\overline{\sigma}^{\mathsf{c}}}$.  
Define $P|_{X_{\Gamma}}$ as the sheaf whose sections on $U_\sigma \subseteq X_{\Gamma}$ are $P[(x^{\overline{\sigma}^{\mathsf{c}}})^{-1}]_0$.

To first see that when $P=S$, this yields $S|_{X_{\Gamma}} = \cO_{X_{\Gamma}}$, 
monomials in $S[(x^{\overline{\sigma}^{\mathsf{c}}})^{-1}]$ correspond to Laurent monomials with positive exponents on $x_\rho$ for $\rho \in \overline{\sigma}$.  In the short exact sequence
\[
0\to M\to \ZZ^{\Sigma(1)}\to \Cl(X)\to 0,
\]
the map $m\mapsto \prod_\rho x_{\rho}^{\langle m, u_\rho\rangle}$ induces a natural bijection between $m\in \overline{\sigma}^\vee \cap M$ and the monomials in $S[(x^{\overline{\sigma}^{\mathsf{c}}})^{-1}]_0$.  
If $L$ is the lineality space of the generalized fan, then there is also a natural bijection between $\overline{\sigma}^\vee \cap M$ and $\sigma^\vee\cap L^\perp$, since $\overline{\sigma}^\vee \subseteq L^\perp$. 
Since the semigroup $\sigma^\vee\cap L^\perp$ corresponds to the monomials in $H^0(U_{{\sigma}}, \cO_{X_{\Gamma}})$, 
\[
S[(x^{\overline{\sigma}^{\mathsf{c}}})^{-1}]_0 \cong H^0(U_{{\sigma}}, \cO_{X_{\Gamma}}). 
\]

For a finitely generated module $P$, one can then follow \cite{CLSToricVarieties}*{Proposition 5.3.6}, with minor adjustments as in the argument above for $S$, to conclude that $P|_{X_{\Gamma}}$ is a coherent sheaf on $X_{\Gamma}$.  
We record the following for later use.
\begin{lemma}\label{lem:ThetaRestrict}
If $-d=-d(\theta)\in \Theta$, then
\[
S(-d)|_{X_{\Gamma}} =
    \begin{cases}
         \cO_{X_{\Gamma}}(-d) &\text{ if } \theta\in L^\perp + M,\\
         0 & \text{ else}.
    \end{cases}
    \qedhere
\]
\end{lemma}
\begin{proof}
Fix the Laurent monomial $x_{\theta}\ce\prod x_{\rho}^{\lfloor-\langle \theta, u_\rho\rangle \rfloor}$.  
Then 
\begin{align}\label{eq:usigma-monomials}
H^0\left(U_\sigma, S(-d(\theta))\right) 
= S(-d(\theta))[(x^{\overline{\sigma}^{\mathsf{c}}})^{-1}]_{0}\cong S[(x^{\overline{\sigma}^{\mathsf{c}}})^{-1}]_{-d(\theta)}, 
\end{align}
with monomial generators in bijection with the set of Laurent monomials $x^{\alpha}$ in $S$ of degree $-d(\theta)$ such that the exponent of $x_\rho$ is nonnegative for all $\rho \in \sigma(1)$; equivalently, by the correspondence $x^{\alpha}\leftrightarrow x^{\alpha}/x_{\theta}$, these sections correspond to Laurent monomials in $S$ where the exponent of $x_\rho$ is at least the exponent of $x_\theta$ for all $\rho \in \sigma(1)$.  By definition of $x_{\theta}$, this is precisely the set of $m\in M$ from \eqref{eq: sections}. 
By the argument in the proof of \cref{prop:pushforwardBTvariety}, \eqref{eq:usigma-monomials} is empty unless $\theta \in L^\perp + M$, and this proves that $S(-d)|_{X_{\Gamma}}=0$ when $\theta \notin L^\perp +M$. When $\theta \in L^\perp + M$, the arguments in \cref{prop:pushforwardBTvariety} show this can be written as
$
\Bbbk \left \langle m \in L^\perp\cap M \mid \langle m, u_\rho \rangle \geq  \langle \theta, u_\rho \rangle  \text{ for all } \rho \in \sigma(1) \right \rangle.
$ 
These are precisely the global sections of $\cO_{X_{\Gamma}}(-d(\theta))$ on each $U_{\sigma}$, and again, just as in the proof of \cref{prop:pushforwardBTvariety}, the restrictions are the same. 
Therefore we conclude that $S(-d)|_{X_{\Gamma}} \cong \cO_{X_{\Gamma}}(-d)$, as desired. 
\end{proof}

\section{The Cox category}\label{sec:CoxCategory}
In this section, we define the Cox category and establish some of its core properties.

\subsection[Construction of a common resolution]{Construction of $\widetilde{\sX}$}\label{subsec:constructionOfWX}
In this subsection, we construct an explicit smooth toric DM stack $\tsX$ satisfying the necessary requirements of Definition~\ref{defn:DCox}.  The underlying coarse moduli space is simply the toric variety associated to a common simplicial refinement of the fans $\Sigma_1, ..., \Sigma_r$.  However, in order for $\tsX$ to lie over all the associated DM stacks, we need to ``refine'' the stacky structure as well.
Recall that $X_i$ is the toric variety associated to $\sX_i$ with fan $\Sigma_i$.
Let $\Lambda$ be a simplicial fan refining  all of the $\Sigma_i$, $X_\Lambda$ be the corresponding toric variety, and $\sX_\Lambda$ be the toric DM stack obtained from $\Lambda$ via \Cref{defn:cox}.

It may not be the case that $\sX_\Lambda$ admits a morphism to each $\sX_i$.
At the level of stacky fans, the obstacle is that when we lift the map of fans to a map of stacky fans, we might end up sending the generator $u_\rho$ of a ray to a non-lattice point.  We return to the running example of the Hirzebruch surface to illustrate this.

\begin{example}
\label{ex:runningHirzebruch}
We wish to find a toric DM stack $\tsX$ that maps to both $\PP(1,1,3)$ and $\cH_3$.
Since the fan for $\cH_3$ is a refinement of the fan for the weighted projective variety $\PP_v(1,1,3)$, the common refinement would just give back $\cH_3$. 
 However, $\cH_3$  does not map to the stack $\PP(1,1,3)$.  Nevertheless, as there are only 2 chambers in $\Sigma_{GKZ}$ in this example, it is not difficult to see what to do.  Let $\tsX$ be the fiber product (in the category of stacks) of $\PP(1,1,3)$ and $\cH_3$ over the $\PP_v(1,1,3)$.  This turns out to be a root stack of order 3 over the $(-3)$-curve.  It can be obtained as the toric stack coming from the map which triples this toric divisor, namely the map
$\beta\colon \ZZ^4 \to N$ given by $\beta(e_i)=u_i$ for $i\ne 1$ and $\beta(e_1)=3u_1$.
\end{example}

The general case is resolved in a similar manner, by appropriately rescaling the terms in the map $\beta_{\Lambda}$.  Algebraically, this corresponds to regrading some of the variables of the Cox ring; geometrically, it corresponds to taking root stacks of components of the toric boundary divisor.
These rescalings on the rays are chosen as follows. Since each $X_i$ is simplicial, any ray $\rho \in \Lambda$ lies in a unique minimal simplicial cone $\sigma_{\rho i} \in \Sigma_i$ and hence there is a relation
\begin{equation} \label{eq:constructionrelation}
a_{\rho i} u_\rho = \sum_{\tau \in \sigma_{\rho i}(1)} a_{\tau} u_\tau,
\end{equation}
where $a_{\rho i}, a_{\tau} \in \ZZ_{>0}$ and $\gcd(a_{\rho i}, a_{\tau}) = 1$.  Let $c_\rho \ce \operatorname{lcm}(a_{\rho i})$.

\begin{defn}
\label{def:tsX}
Following \cref{def:BCStoricStack}, 
let $\tsX$ be the smooth toric DM stack associated to 
$
(N, \Lambda, \beta_{\Lambda}),
$
where $\beta_{\Lambda}(e_\rho) \ce c_\rho u_\rho$.
\end{defn}

\begin{defn}
\label{def:alphai}
For each $i$, define a linear map
$\alpha_i\colon \ZZ^{\Lambda(1)}  \to \tN$, 
where 
$e_\rho  \mapsto c_\rho u_\rho$. 
Note that, since 
$
c_{\rho}u_\rho = \frac{c_\rho}{a_{\rho i}} \sum_{\tau \in \sigma_{\rho i}(1)} a_{\tau} u_\tau, 
$
the map $\alpha_i$ now satisfies the second condition of \cite{BorisovChenSmith2004}*{Remark~4.5}.  As the first condition was satisfied without this rescaling, we conclude that each $\alpha_i$ determines a morphism of stacky fans.
\end{defn}

\begin{prop}\label{prop:constructionX}
The toric stack $\tsX$ is a smooth DM stack with a simplicial underlying fan, and the maps $\pi_i\colon \tsX\to \sX_i$ are proper and birational with $\pi_{i*}\cO_{\tsX} = \cO_{\sX_i}$ for all $i$.
\end{prop}
\begin{proof}
The morphism is birational because it induces an isomorphism on the open dense torus $T$.  For properness, consider the commutative diagram
\vspace*{-2mm}
\[
\begin{tikzcd}
\tsX \ar[d, "{\pi_i}"] \ar[r, "a"] & X_\Lambda \ar[d, "r"] \\
\sX_i \ar[r, "b"] & X_i,
\end{tikzcd}
\]
where $a,b$ are the coarse moduli maps, and $r$ is the map of toric varieties coming from the refinement of fans.  Now  $a,b$ are proper by the Keel--Mori theorem \cite{KM95}.  Furthermore, $r$ is a proper morphism of toric varieties because it comes from a refinement of fans~\cite{CLSToricVarieties}*{Theorem 3.4.11}.  It follows that $\pi_i$ is proper by~\cite{stacks-project}*{Lemma 0CPT}.
By construction, the maps $\pi_i$ are induced by stacky refinements of fans.
Therefore  \cref{prop:pushforwardBTstack} (with $\theta = 0$) implies that $\pi_{i*} \cO_{\tsX} = \cO_{\sX_i}$ for all $i$.
\end{proof}

We record a more general version of this proposition for refinements of a broader class of stacky fans.  It will be used in \cref{cor:DCoxIndOfX} and may be of independent interest.

\begin{prop} \label{prop:stackyrefinementsexist}
If  $(N,\Sigma_1, \beta_1)$ and $(N, \Sigma_2, \beta_2)$ are a pair of stacky fans (as in \cref{def:BCStoricStack}) with the same support, then there exists $(N,\widetilde{\Sigma}, \widetilde{\beta})$ that is a stacky refinement of both, in the sense of \cref{defn:refinement}, such that $\widetilde{\Sigma}$ is simplicial.
\end{prop} 
\begin{proof} First, refine $\Sigma_1$ and $\Sigma_2$ to simplicial fans by applying \cite{CLSToricVarieties}*{Proposition 11.1.7}. 
From there, the construction and proof are essentially identical to \cref{def:tsX} and \cref{prop:constructionX}, but with $u_\rho$ and $u_\tau$ replaced by $\beta_i(e_\rho)$ and $\beta_{i}(e_\tau)$ in \eqref{eq:constructionrelation}.
\end{proof}

\subsection{The Cox category}\label{subsec:CoxDefinitions}
In this subsection, we establish some basic properties about $D_{\Cox}$; 
see Appendix \ref{sec:grothendieckconstr} for an alternative colimit description. The following is well-known for varieties, and the same essential proof goes through in our context.

\begin{lemma}\label{lem:fullyFaithful}
If $\alpha\colon \sX' \to \sX$ is a proper, birational morphism of smooth toric Deligne--Mumford stacks such that $\alpha_* \cO_{\sX'} = \cO_{\sX}$, then $\alpha^*\colon D(\sX) \to D(\sX')$ is fully faithful.
It follows that the functors $\pi_i^*\colon D(\sX_i)\to D(\tsX)$ are fully faithful for all $i$.
\end{lemma}
\begin{proof}
For any $\cF \in D(\sX)$, 
$\alpha_{\ast}\alpha^\ast \mathcal F \cong \mathcal F \otimes
\alpha_{\ast} \cO_{\sX'}$ by the projection formula, which can be applied as $\cF$ is quasi-isomorphic to a bounded complex of locally free sheaves since $\sX$ is smooth.
Hence the unit map $\Id \to \alpha_\ast \alpha^\ast$ is an
isomorphism.
Consequently, $\alpha^\ast$ is fully faithful.   The application to the $\pi_i$ from \cref{prop:constructionX} is immediate.
\end{proof}

With the standard notation for thick subcategories generated by a 
collection of of subcategories, Definition~\ref{defn:DCox} can be rewritten as
\begin{displaymath} D_{\Cox,\tsX} \ce \langle
\pi_1^\ast D(\sX_1), \ldots, \pi_r^\ast D(\sX_r) \rangle,
\end{displaymath} 
which is the smallest thick subcategory generated by the pullback categories from
each chamber.  For the moment, we include the choice of $\tsX$ in the notation for $D_{\Cox}$ though, as we will confirm in a moment, the category $D_{\Cox}$ does not depend on the choice of $\tsX$. First, we show the following:

\begin{lemma} \label{lem:DCox_under_refinement}
Let $\alpha\colon \tsX'\to \tsX$ be a proper birational map of smooth toric
stacks induced by a stacky refinement of fans.  Write
$\rho_i=\pi_i\circ \alpha \colon \tsX^\prime \to \sX_i$.  Then $\alpha^*$ is an
equivalence of categories between $\langle \pi_1^*D(\sX_1), \dots,
\pi_r^*D(\sX_r)\rangle$ and $\langle \rho_1^*D(\sX_1),  \dots,
\rho_r^*D(\sX_r)\rangle$.  
\end{lemma}

\begin{proof} 
The functor $\alpha^*$ is fully faithful by
\cref{lem:fullyFaithful}.  It induces an equivalence since it essentially
surjects onto the generating set. 
\end{proof}

\begin{cor} \label{cor:DCoxIndOfX}
Given two such $\tsX^\prime$ and $\tsX$, there is a natural equivalence 
$D_{\Cox,\tsX^\prime} \simeq D_{\Cox,\tsX}$. 
\end{cor}
\begin{proof}
By \cref{prop:stackyrefinementsexist}, there exists $\widetilde{\sY}$ with $\pi\colon \widetilde{\sY} \to \tsX$ and $\pi'\colon \widetilde{\sY} \to \tsX'$ both birational and proper. 
Then apply \cref{lem:DCox_under_refinement} to obtain the chain of equivalences.
\end{proof}

In view of \cref{cor:DCoxIndOfX}, we will drop any mention of $\tsX$ from 
the notation for the Cox category.   See also Appendix~\ref{sec:grothendieckconstr}.

\begin{defn}\label{defn:kernelCategory}
Let $\mathcal Q\subseteq D(\tsX)$ be the subcategory of elements $\cE$, where $\pi_{i*}\cE=0$ in $D(\sX_i)$ for all $1\leq i \leq r$.
\end{defn}
We prove in \cref{cor:semiorthogonalDCox} that $D(\tsX)$ has a
semiorthogonal decomposition involving $\mathcal Q$ and $D_{\Cox}$.
For now, we simply observe the following weaker results:

\begin{lemma} \label{lem:semiwarmup} The category $\mathcal Q$ is the right
orthogonal to $D_{\Cox}$. \end{lemma}
\begin{proof} 
Consider an element in $D_{\Cox}$ of the form
$\pi_i^*\cE$ for some $i$, and an arbitrary object $\cF \in \mathcal Q$. 
By adjunction, $\RHom_{\tsX}(\pi_i^*\cE, \mathcal F) =
\RHom_{\sX_i}(\cE,\pi_{i*}\mathcal F) = 0$ because  $\pi_{i*}\cF=0$.  Since $D_{\Cox}$ is generated by the elements of the form $\pi_i^*\cE$ for all $i$, this implies one direction.

Conversely, if $\Hom_{\tsX}(\pi_i^*\cE, \mathcal F) = 0$ for all $i$ and $\cE$, then
$\Hom_{\sX_i}(\cE, $${\pi_i}_*\mathcal F) = 0$ for all $i$ and $\cE$, which implies that
${\pi_i}_*\mathcal F = 0$ for all $i$. 
Thus, $\mathcal F \in \mathcal Q$.
\end{proof}

\begin{lemma}\label{lemma:checkpush} 
Let $\mathcal E, \mathcal F \in D_{\Cox}$
and $\phi \in \Hom_{\tsX}(\mathcal E,\mathcal F)$. If $\pi_{i*}(\phi)$
induces a quasi-isomorphism for all $i$, then $\phi$ induces a quasi-isomorphism
in $D_{\Cox}$. 
\end{lemma} 

\begin{proof}
The cone $C(\phi) = [\mathcal E \overset{\phi}{\to} \mathcal F]$ lies in
$D_{\Cox}$, and our assumption is that $\pi_{i*} C(\phi)=0$ for all~$i$.  Thus,
$C(\phi)\in
D_{\Cox} \cap \mathcal Q$.  If $C(\phi)$ were nonzero, then the identity morphism
would give a nonzero map between an object in $D_{\Cox}$ and one in $\mathcal
Q$, which is impossible by Lemma~\ref{lem:semiwarmup}. \end{proof}

Recall that $(-)^\vee \ce \mathcal Hom(-, \mathcal O_{\widetilde{\sX}})$
is a contravariant auto-equivalence of $D(\tsX)$.

\begin{lemma} \label{lem:self_dual}
The duality functor $(-)^\vee \colon D(\tsX)^{\operatorname{op}} 
\to D(\tsX)$ preserves $D_{\Cox}$. Hence it induces a 
contravariant autoequivalence of $D_{\Cox}$.
\end{lemma}

\begin{proof}
If $\mathcal E \in D_{\Cox}$, then 
\begin{align*}
\mathcal E^\vee & \in \langle (\pi_1^\ast D(\sX_1))^\vee, \ldots, 
(\pi_r^\ast D(\sX_r))^\vee \rangle 
\\
& \qquad 
\quad
= \langle \pi_1^\ast D(\sX_1)^\vee, \ldots, 
\pi_r^\ast D(\sX_r)^\vee \rangle 
= \langle \pi_1^\ast D(\sX_1), \ldots, 
\pi_r^\ast D(\sX_r) \rangle = D_{\Cox}. 
\qedhere
\end{align*} 
\end{proof}

\begin{remark}
The duality from \cref{lem:self_dual} a priori depends on the choice of 
$\tsX$. However, the equivalence between the different versions of $D_{\Cox}$ 
is given by pullback along $\alpha\colon \tsX^\prime \to \tsX$, which commutes with
$(-)^\vee$. 
\end{remark}

\section[The Theta-transform Lemma]{The $\Theta$-Transform Lemma}\label{sec:PushPull}
The main goal of this section is to prove \cref{lem:PushPull}, the $\Theta$-Transform Lemma. 
Recall that $X$ is a semiprojective toric variety with fan $\Sigma$ and we are using an arbitrary indexing $\Gamma_i$ of the chambers in $\Sigma_{GKZ} = \Sigma_{GKZ}(X)$ with corresponding toric stacks $\sX_i$ as in \cref{defn:cox}.
Further, let $\tsX$ be any toric DM stack with birational toric morphisms $\pi_{i} \colon \tsX \to \sX_i$ induced by generalized stacky refinements of fans for all $i$, as shown to exist in \cref{subsec:constructionOfWX}.

Before proving \cref{lem:PushPull} in \cref{subsec:proofofthetatrans}, we first precisely define the Bondal--Thomsen collection in $D_{\Cox}$ in \cref{subsec:BTforDCox}, with various examples in \cref{subsec:examples}. 
Finally, in \cref{subsec:thetacomputations}, we deduce some consequences of the $\Theta$-Transform Lemma.

\subsection[The Bondal--Thomsen collection for the Cox category]{The Bondal--Thomsen collection for the Cox category} \label{subsec:BTforDCox}

We now define a line bundle $\cO_{\Cox}(-d)$ in $D_{\Cox}$ for each $-d \in \Theta \ce \Theta_X$.
Here we face a subtle question: which of the  $\cO_{\sX_j}(-d)$  should we pullback in order to define a representative of $-d$ in  $D_{\Cox}$?  We use the secondary fan as our guide.

\begin{defn}\label{defn:ThetaCox}
Choose $-d \in \Theta$.  Since $d$ is an effective degree, its image in $\Cl(X)_\RR$ lies in a maximal cone $\Gamma_i$ of $\Sigma_{GKZ}$, which corresponds to some $\sX_i$.  We define $\cO_{\Cox}(-d)\ce\pi_i^*\cO_{\sX_i}(-d)$.
\end{defn}
For example, in the case of $\cH_3$ as in Figure~\ref{fig:HirzSecondaryFan}, we would pull back $-d_1$ and $-d_5$ from $\cH_3$, and $-d_2$ and $-d_3$ from $\PP(1,1,3)$; for $-d_0$ and $-d_4$, either choice will work.

For future use, we note that given \cref{defn:ThetaCox}, \cref{lem:PushPull} is equivalent to the statement that $\pi_{j \ast}\cO_{\Cox}(-d) = \cO_{\sX_j}(-d)$ for all $j$ and all $-d \in \Theta$. 
The $\cO_{\Cox}(-d)$ are well-defined by the following result.
\begin{prop} \label{prop:welldefined}
    If the image of $d$ in $\Cl(X)_\RR$ lies in $\Gamma_i \cap \Gamma_j$, then $\pi_i^*\cO_{\sX_i}(-d) = \pi_j^* \cO_{\sX_j}(-d)$.
\end{prop}
\begin{proof}
    Choose $a \in \ZZ^{\Sigma(1)}$ projecting to $d$ in $\Cl(X)$.
    Let $F_i$ be the support function of $d$ on $\Sigma_i$ corresponding to this choice.
    Since the image of $d$ in $\Cl(X)_\RR$ is in $\Gamma_i$, $F_i(u_\rho) \geq -a_\rho$ for all $\rho \in \Sigma (1)$, and $F_i$ is concave on $\Sigma_i$.
    Let $\sigma \in \Sigma_j$.
    Any $u \in \sigma$ can be written as $u = \sum_{\rho \in \sigma(1)} c_\rho u_\rho$ for some $c_\rho \in \RR_{>0}$.
    Using the properties of $F_i$, 
    \[
    F_i(u) = F_i \left( \sum_{\rho \in \sigma(1)} c_\rho u_\rho \right) \\
        \geq \sum_{\rho \in \sigma(1)} c_\rho F_i(u_\rho) \\
        \geq - \sum_{\rho \in \sigma(1)} c_\rho a_\rho \\
        = F_j(u).
    \]
That is, $F_i(u) \geq F_j(u)$ for all $u \in |\Sigma_j| = |\Sigma_i|$. Also, since the image of $d$ in $\Cl(X)_\RR$ lies in $\Gamma_j$, the opposite inequality holds as well, so $F_i = F_j$.
As the pullback of a line bundle is determined by the support function, we obtain the desired result.
\end{proof}

Theorem~\ref{thm:main1} also requires ordering the line bundles in $\Theta$ when $X$ is projective:

\begin{defn}\label{defn:ordering}
When $X$ is projective, the set $\Theta$ is equipped with a partial ordering such that $-d \le -d'$ if and only if $d - d'$ is effective. We choose any refinement of this partial ordering to a total ordering $-d_1, \dots, -d_t$.
\end{defn}

Note that $-d_j - (-d_i)$ is not effective if $i > j$ for any choice of ordering arising from \cref{defn:ordering}.

\subsection{Examples}\label{subsec:examples}
Here we collect some examples to  illustrate properties of $D_{\Cox}$  and the $\Theta$-Transform Lemma.

\begin{example}\label{ex:singleMaximal}
If $\Sigma_{GKZ}$ has a single maximal chamber corresponding to $\sX$, so that the nef cone equals the effective cone, then $D_{\Cox}=D(\sX)$ and Theorem~\ref{thm:main1} implies that $\Theta$ is a full strong exceptional collection for $D(\sX)$ itself.
This happens precisely when $\sX$ is $\PP^n$, a weighted projective stack, or a product and/or finite quotient of these~\cite{fujino2009smooth}*{Proposition 5.3}.  In these cases, $\Theta$ is the standard full strong exceptional collection, e.g., on $\PP^1 \times \PP(1,1,2)$, $\Theta$ is the collection $(0,0), (0,-1), (0,-2), (0,-3),(-1,0), (-1,-1), (-1,-2)$ and $(-1,-3)$.  

More generally, if $\Theta$ lies entirely in the chamber for $\sX$, so that all elements of $\Theta$ are nef, then $D_{\Cox}=D(\sX)$.    
For example, if $X=\cH_1$ is a Hirzebruch of type $1$, then there are two chambers corresponding to $\cH_1$ and $\PP^2$, but the Bondal--Thomsen collection $\Theta$ lies entirely in the chamber corresponding to $\cH_1$ and so $D_{\Cox}=D(\cH_1)$ in this case.
\end{example}

\begin{example}\label{ex:hirzStrongBT}
Revisiting Example~\ref{ex:hirzGKZ} and Figure~\ref{fig:HirzSecondaryFan},  
write $\pi_1\colon \tsX\to \cH_3$ and $\pi_2\colon \tsX \to \PP(1,1,3)$.
By Definition~\ref{defn:ThetaCox}, $\cO_{\Cox}(-1,0) = \pi_1^* \cO_{\cH_3}(-1,0)$.  Since $d=(-2,1)$ lies in the chamber corresponding to $\PP(1,1,3)$, we must use the fact from \cref{subsec:BTGKZ} that there is a natural corresponding element in $\Theta_{\PP(1,1,3)}$.  In this case, the map $\alpha\colon \ZZ^2\to \Cl(\PP(1,1,3))$ is given by $(a,b) \mapsto a+3b$ and thus $\cO_{\Cox}(2,-1) = \pi_2^* \cO_{\PP(1,1,3)}(\alpha(2,-1))=\pi_2^* \cO_{\PP(1,1,3)}(-1)$.  
Here is a sample Hom computation: 
\begin{align*}
\RHom(\cO_{\Cox}(-1,0), \cO_{\Cox}(2,-1))
&=\RHom(\pi_1^*\cO_{\cH_3}(-1,0), \pi_2^*\cO_{\PP}(2,-1)) & \text{by definition,}\\
&=\RHom(\cO_{\cH_3}(-1,0), \pi_{1*}\pi_2^*\cO_{\PP}(2,-1)) & \text{by adjunction,}\\
&=\RHom(\cO_{\cH_3}(-1,0), \cO_{\cH_3}(2,-1)) & \text{by Lemma~\ref{lem:PushPull},}\\
&=\RHom(\cO_{\cH_3}, \cO_{\cH_3}(3,-1))=0 & \text{by direct computation.}
\end{align*}
Note that the inverse computation is nonzero:
\begin{align*}
\RHom(\cO_{\Cox}(2,-1),  \cO_{\Cox}(-1&,0)) \\
\qquad 
&=\RHom(\pi_2^*\cO_{\PP(1,1,3)}(-1), \pi_1^*\cO_{\cH_3}(-1,0))&\text{by definition,}\\
&=\RHom(\cO_{\PP(1,1,3)}(-1), \pi_{1*}\pi_2^*\cO_{\cH_3}(-1,0)) & \text{by adjunction,}\\
&=\RHom(\cO_{\PP(1,1,3)}(-1), \cO_{\PP(1,1,3)}(-1)) & \text{by \cref{lem:PushPull},} \\
&= \Bbbk^1 & \text{by direct computation.}
\end{align*}
Through similar computations for the other pairs, one can directly confirm that $\Theta$ forms a strong, exceptional collection.  Fullness requires an additional argument.
\end{example}

 \begin{example}[Atiyah Flop]\label{ex:atiyahFlop}
Consider the Cox ring $S=\Bbbk[x_0,x_1,y_0,y_1]$ with $\deg(x_i)=1$ and $\deg(y_j)=-1$.  The secondary fan has two chambers: $\RR_{\geq 0}$ and $\RR_{\leq 0}$.
The first corresponds to $Y_+ =[\Spec(S) - V(x_0,x_1) / \GG_m]$, and  the second to $Y_- =[\Spec(S) - V(y_0,y_1) / \GG_m]$. The origin corresponds to $Y_0=\Spec(S_0)$ where $S_0=\Bbbk[x_0y_0, x_0y_1, x_1y_0, x_1y_1]$.  
For $\tY$, we choose the blowup of $Y_+$ at the line $V(y_0,y_1)$ (which is isomorphic to the blowup of $Y_-$ at $V(x_0,x_1)$).  We write $\pi_+\colon \tY\to Y_+$ and similarly for $\pi_-$.
The Cox category is the full subcategory of $D(\tY)$ generated by $\pi_+^* D(Y_+)$ and $\pi_-^* D(Y_-)$.  

The Bondal--Thomsen collection $\Theta$ consists of the degree $-1,0,1$ divisors.  By \Cref{defn:ThetaCox}, we define $\cO_{\Cox}(-1)\ce \pi_+^* \cO_{Y_+}(-1)$ and $\cO_{\Cox}(1)\ce \pi_-^* \cO_{Y_-}(1)$.  The pullbacks of $\cO_{Y_+}$ and $\cO_{Y-}$ coincide and are equal to $\cO_{\Cox}$.  \Cref{thm:main1} implies that $\cO_{\Cox}(-1)\oplus \cO_{\Cox}\oplus \cO_{\Cox}(1)$ is a tilting bundle.

The $\Theta$-Transform Lemma implies that the Fourier--Mukai transform of $\cO_{Y_+}(d)$ from $Y_+$ to $Y_-$ will equal $\cO_{Y_-}(d)$ when $d=0$ or $d=-1$.
A direct computation confirms that these are the {\em only} values of $d$ where this holds.  If $d > 0$ then the degree zero cohomology of the Fourier--Mukai transform will be the ideal sheaf of $(y_0,y_1)^d$, and if $d\leq -2$ then the degree one cohomology will be nonzero.
\end{example}

The example above demonstrates some of the delicacy in the $\Theta$-Transform Lemma.  If a Fourier--Mukai transform $\Phi_{ij}$ is applied to a line bundle $L$ that is too positive, then the $R^0$ term of that Fourier--Mukai transform would be a twist of an ideal sheaf, not necessarily a line bundle.
On the other hand, if the line bundle $L$ is too negative, then there is a risk of acquiring nontrivial higher direct images.
The hypotheses of the lemma are calibrated so as to navigate between these: the fact that the image of $d$ lies in the chamber of $\Gamma_i$ ensures that the $R^0$-term is correct; the fact that $-d$ came from $\Theta$ means it is not so negative as to acquire higher direct images.  The latter is captured in our argument by the existence of $\theta \in M_\RR$ defining $-d \in \Theta$.  
In a certain sense, the support function of $-d$ is approximated by the linear function given by pairing with $\theta$.

One part of Example~\ref{ex:atiyahFlop} is, however, a bit misleading in that $\pi^*_+\cO_{Y^+}(-1)$ is an element of the Bondal--Thomsen collection on $\tY$.  If this were always true, then it would directly imply that the higher pushforwards vanish by \cref{prop:pushforwardBTstack}.
However, Bondal--Thomsen elements are not always sent to Bondal--Thomsen elements under pullback.  (For instance, this can fail in the case of the Hirzebruch surface $\tsX\to \cH_3$.) Thus in general, one must prove the vanishing of higher pushforwards using a different mechanism.

\begin{example}\label{ex:Bl2points}
Let $X$ be the blowup of $\PP^n$ at the points $[1\mathbin:0\mathbin:\cdots \mathbin:0]$ and $[0\mathbin:0\mathbin: \cdots :0\mathbin:1]$.
In~\cite{michalek11} it was shown that $X$ provides a
counterexample to King's Conjecture for $n > 20$.
The Cox ring $S$ of $X$ can be written as $S=\Bbbk[x_0, \dots, x_{n+2}]$ with degrees
$(1, 0, 1),(1, 0, 0), \dots, (1, 0, 0), (1, 1, 0),$   $(0, 1, 0),$ and $(0, 0, 1)$.  The class $(1,0,0)$ corresponds to the hyperplane class from $\PP^n$, and $(0,1,0)$ and $(0,0,1)$ correspond to the two exceptional divisors. The GKZ fan of $X$ has $5$ maximal chambers as illustrated in Figure~\ref{fig:Bl2Pts}. For $n=3$, the chambers correspond to $\PP^3$, the blowup of $\PP^3$ at one of the two points, $X$ itself, and another variety $X'$ that is a $\PP^1$-bundle over a Hirzebruch surface $\cH_1$ of type $1$.  Along the boundary, there is a point, a $\PP^1$, two copies of $\PP^2$, and two copies of $\cH_1$.  Along the internal edge between $X$ and $X'$ is the contraction $X_0$ for the corresponding flop.  The distinct varieties that arise are labeled in \Cref{fig:Bl2Pts}.

The Bondal--Thomsen collection consists of $4n-3$ elements:
\[
\Theta_X = \left\{\begin{matrix}(0,0,0), \dots, (-n-2,0,0), (-1,-1,0), \dots, (-n-1,-1,0), 
\\
(-1,0,-1), \dots, (-n-1,0,-1),
(-1,-1,-1), \dots, (-n,-1,-1),
\end{matrix}\right\}.
\]
Each element lies in the chamber associated to either $X$ or $X'$ and so in this case, $D_{\Cox}=\langle \pi_1^* D(X'), \pi_2^* D(X)\rangle$.
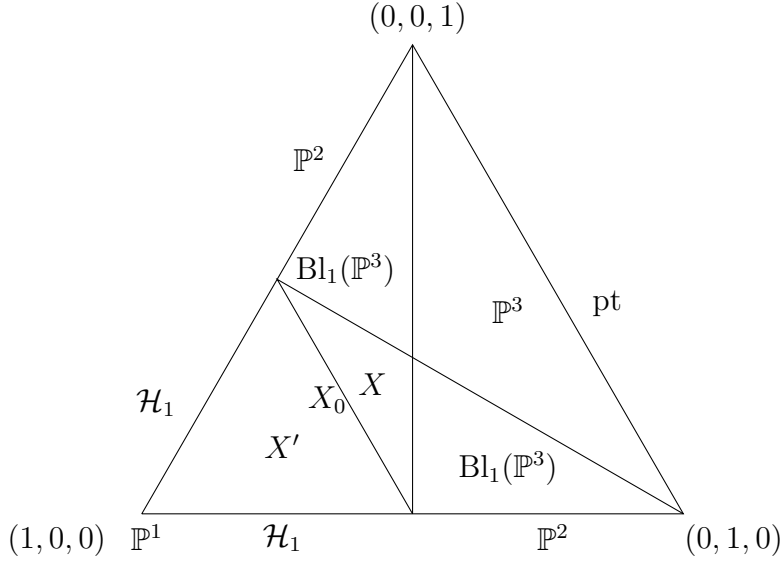
\begin{figure}
    \centering
\begin{tikzpicture}[scale = 1.2]
\draw[black] (0,0) -- (6,0)--(3,5.2)--(0,0);
\draw[black] (3,0)--(3,1.733)--(1.5,2.6)--(3,0);
%
\draw[black] (3,1.733)--(3,5.2);
\draw[black] (6,0)--(3,1.733);
\filldraw[black] (-1,0) circle (0pt) node[anchor=north]{{ $(1,0,0)$}};
\filldraw[black] (6.5,0) circle (0pt) node[anchor=north]{{ $(0,1,0)$}};
\filldraw[black] (3,5.2) circle (0pt) node[anchor=south]{{ $(0,0,1)$}};
\filldraw[black] (2.5,1.2) circle (0pt) node[anchor=south]{{ $X$}};
\filldraw[black] (2.0,1.0) circle (0pt) node[anchor=south]{{ $X_0$}};
\filldraw[black] (1.5,.5) circle (0pt) node[anchor=south]{{ $X'$}};
\filldraw[black] (4,.2) circle (0pt) node[anchor=south]{{ $\operatorname{Bl}_1(\PP^3)$}};
\filldraw[black] (2.2,2.4) circle (0pt) node[anchor=south]{{ $\operatorname{Bl}_1(\PP^3)$}};
\filldraw[black] (4,2.0) circle (0pt) node[anchor=south]{{ $\PP^3$}};
\filldraw[black] (.1,1.5) circle (0pt) node[anchor=north]{{ $\cH_1$}};
\filldraw[black] (1.8,4.2) circle (0pt) node[anchor=north]{{ $\PP^2$}};
\filldraw[black] (5.1,2.6) circle (0pt) node[anchor=north]{{ $\text{pt}$}};
\filldraw[black] (1.5,0) circle (0pt) node[anchor=north]{{ $\cH_1$}};
\filldraw[black] (4.5,0) circle (0pt) node[anchor=north]{{ $\PP^2$}};
\filldraw[black] (0,0) circle (0pt) node[anchor=north]{{ $\PP^1$}};
\end{tikzpicture}
    \caption{The GKZ fan from Example~\ref{ex:Bl2points} has five maximal  chambers.  When $n=3$, we label each chamber and face with the corresponding variety.}
    \label{fig:Bl2Pts}
\end{figure}
\end{example}

\begin{remark}\label{rmk:caveats}
We caution the reader
that there are some categories which might appear similar to $D_{\Cox}$ but are in fact not equivalent.  
For instance, again consider the Hirzebruch surface $\cH_3$.  
The Cox ring is $S=\Bbbk[x_0,x_1,x_2,x_3]$ and the
irrelevant ideal is $B_1\ce(x_0,x_2)\cap (x_1,x_3)$.  Thus $D(\cH_3) = D(S) /
D_{B_1}$ where $D_{B_1}$ is the bounded derived category of finitely generated
$S$-modules whose homology is $B_1$-torsion.  
If $B_2\ce(x_1) \cap
(x_0,x_2,x_3)$, then similarly $D(\PP(1,1,3)) = D(S) / D_{B_2}$.
However, it is not that case $D_{\Cox}$ equals $D(S) / D_{B_1+B_2}$.  In
fact,  $D(S) / D_{B_1+B_2}$ is the derived category of the non-separated
stack one obtains by gluing $\cH_3$ and $\PP(1,1,3)$ along their birational
locus; it is not even a proper dg-category.

Similarly, there is generally not some $\tsX$ with $D_{\Cox}=D(\tsX)$.  This
is perhaps easiest to see for the Atiyah Flop example: with notation as in
\cref{ex:atiyahFlop}, $D_{\Cox}\subseteq D(\tY)$. 
But $K_0(D_{\Cox}) = \ZZ^3$ by Theorem~\ref{thm:main1}, whereas $K_0(D(\tY))=\ZZ^4$
by~\cite{borisov-horja}.  

We also note that $D_{\Cox}$ is not a monoidal category in the ``obvious way;'' it is not closed with respect to the induced monoidal structure from $D(\tsX)$.
For instance, in \cref{ex:atiyahFlop}, tensor products of the line bundles
$\pi_+^*\cO_{Y_+}(-1)$ and $\pi_-^*\cO_{Y_-}(1)$ generate all of the line
bundles on $\tY$.  Since $D(\tY)$ is strictly larger than $D_{\Cox}$, it follows that 
that $D_{\Cox}$ is not closed with respect to this monoidal structure, 
as it would be if it were the category supported along subvariety or the 
pullback from some minimal rational resolution.

One must also be careful in passing from $S$-modules
to elements of $D_{\Cox}$.  For instance, in \cref{ex:atiyahFlop}, there is no
element of $D_{\Cox}$ that pushes forward to both $\cO_{Y_+}(-2)$ and
$\cO_{Y_-}(-2)$.  Said another way, it is not the case that $D_{\Cox}$ can
always be expressed as the Verdier quotient of $D(S)$ along a torsion
subcategory. 
\end{remark}

\subsection[Proof of the Theta-transform Lemma]{Proof of the $\Theta$-Transform Lemma} \label{subsec:proofofthetatrans}

Recall that for any $1\leq i,j\leq r$, there is a commutative diagram as below, where the horizontal arrow is a birational map.  
\[
\xymatrix{
&\tsX\ar[rd]^-{\pi_i}\ar[ld]_-{\pi_j}\\
\sX_i\ar@{<-->}[rr]&&\sX_j.
}
\]
Our goal is to understand how elements of the Bondal--Thomsen collection behave with respect to these birational maps, and the associated Fourier--Mukai transforms.
For convenience, we prove the following restatement of \cref{lem:PushPull}.  

\begin{lemma} \label{lem:thetatransform2}
If $-d \in\Theta$ and has image in $\Gamma_i$, then
       $\pi_{j\ast} \pi_i^* \cO_{\sX_i}(-d) = \cO_{\sX_j} (-d)$ for all $j$.
\end{lemma}

As discussed after \cref{ex:atiyahFlop}, the hypotheses provide a delicate balance. Since $d$ has image in $\Gamma_i$, $\cO_{\sX_i}(-d)$ is anti-nef, and this factors into our proof that $R^0\pi_{j*}$ is the right line bundle. The fact that $-d\in \Theta$ implies that it is ``close to $0$,'' and so it is not so negative as to acquire nonzero higher direct images.

Our proof involves testing the object $\pi_{j\ast} \pi_i^* \cO_{\sX_i}(-d)$ against nef line bundles. 
This allows us to reduce to a cohomology computation for line bundles that can be expressed as a difference of nef line bundles, meaning those of the form $D_1 -D_2$, where both $D_1$ and $D_2$ are nef. 
(Recall that nef line bundles on a toric DM stack are defined in \cref{defn:nef}.) 
Then \cite{altmann-immaculate}*{Theorem III.6} can be used to analyze such divisors in terms of their polytopes.
To reduce to a set of cohomology computations, we  rely on the following lemma.

\begin{lemma} \label{lem:yonedatheta}
Suppose that $\pi\colon \sX' \to \sX$ is a proper morphism of smooth toric DM stacks with quasi-projective coarse moduli spaces. Let $\cF$ be a coherent sheaf on
$\sX'$.
If
    \begin{equation} \label{eq:testnef}
        H^p(\sX,\pi_* \cF \otimes \cO_{\sX}(A) ) \cong 
        H^p(\sX, (R^0\pi_* \cF) \otimes \cO_{\sX}(A))
    \end{equation}
for all $p$ and every nef toric divisor $A$, then $\pi_* \cF = R^0\pi_*
\cF$. 
\end{lemma}
The statement is likely known to experts, but as we could not find a precise reference that implied this result, we provide a separate proof based on a pair of elementary lemmas.

\begin{lemma} \label{lem:generate_plus_pullback} 
If $\sX$ is a smooth toric DM stack with coarse moduli morphism $\pi \colon \mathcal X \to X$, where $X$ is 
quasi-projective,  
then there exist divisors $D_1,\ldots,D_l$ on $\sX$
such that $\pi^\ast (\Perf X) \otimes \bigoplus \mathcal
O_{\sX}(D_i)$ generates $D(\sX)$.
\end{lemma}
\begin{proof}
By \cite{KreschDM}*{Section 5.2}, there exist $D_1, \ldots, D_l$
such that $\cE \ce \bigoplus \mathcal O_{\sX}(D_i)$ is a generating sheaf. By definition, 
$L^0\pi^\ast \pi_\ast \mathcal Hom(\cE,\mathcal F) \otimes \cE \to \mathcal F$ 
is surjective for all quasi-coherent sheaves $\mathcal F$. 
(Note that $\pi_\ast$ is exact.) 
There are enough line bundles on $X$, so there is a surjection $\mathcal V \to \pi_\ast \mathcal Hom(\cE,\mathcal F)$ with
$\mathcal V$ a vector bundle. 
Since $L^0\pi^\ast$ is right exact, the composition 
$
\pi^\ast \mathcal V \otimes \cE = L^0\pi^\ast \mathcal V \otimes \cE \to \mathcal F
$
is surjective. 
Now $\sX$ is smooth and $\mathcal F$ can be resolved by complexes of the 
form $\pi^\ast \mathcal V \otimes \cE$, so $D_1,\dots,D_l$ generate $D(\sX)$, up to splitting 
idempotents. 
\end{proof}

\cref{lem:generate_plus_pullback} amounts to the fact that $\Pic(X)$ is a finite index subgroup of $\Cl(X)$ because $\sX$ corresponds to a simplicial fan~\cite{CLSToricVarieties}*{Proposition 4.2.7}.  For instance, in the case of a weighted projective stack like $\PP(1,1,3)$, one can choose the bundles $\cO(i)$ for $0\leq i\leq 2$.

\begin{lemma} \label{lem:generate_by_nef}
Let $\sX$ be a smooth toric DM stack with quasi-projective coarse moduli space. A map $\phi \colon \cE \to \mathcal F$ in $D(\sX)$ is a quasi-isomorphism 
if and only the maps
\[
\operatorname{Hom}^p(\cO_{\sX}, \mathcal E \otimes \cO_{\sX}(A)) \to 
\operatorname{Hom}^p(\cO_{\sX}, \mathcal F \otimes \cO_{\sX}(A)) 
\]
are isomorphisms for all $p$ and every nef toric divisor $A$.
\end{lemma}

\begin{proof}
It is sufficient to show that the $\mathcal O_{\sX}(A)$ from nef $A$ on 
$X$ generate $D(\sX)$.
To do this, choose an ample line bundle $\mathcal O_X(D)$ on $X$ such that 
$\mathcal O_X(\lambda D-D_i)$ are nef for each $\lambda \geq 1$ and all $i$ from 
\cref{lem:generate_plus_pullback}.
The line bundles $\mathcal O_X(\lambda D)$ for $\lambda \geq 1$ generate 
$\Perf X$ since $\mathcal O_X(D)$ is ample. Thus, the
$\mathcal O_{\sX}(\lambda D-D_i)$ generate $D(\sX)$. 
\end{proof}

\begin{proof}[Proof of \cref{lem:yonedatheta}]
Since $\pi$ is proper, $\pi_*\cF$ is a bounded complex of coherent
sheaves by \cite{olsson}*{Theorem~1.2}. Applying \cref{lem:generate_by_nef} 
to the map $\pi_\ast \mathcal F \to R^0\pi_\ast \mathcal F$ gives the result.
\end{proof}

In light of \cref{lem:yonedatheta}, we  first show that the underived pushforward gives the correct sheaf before addressing the higher direct images.

\begin{lemma} \label{lem:pushforward0}
    If $-d \in \Theta$ and has image in $\Gamma_i$, then $R^0\pi_{j \ast} \pi_i^* \cO_{\sX_i}(-d) = \cO_{\sX_j}(-d)$ for all $j$.
\end{lemma}
\begin{proof}
    Let $\sigma \in \Sigma_j$.
    Since the map $\pi_i$ is induced by a stacky refinement, there is a stacky fan $\tsigma \subset \tSigma$ such that $\pi_i^{-1}(\sX_{\sigma}) = \sX_{\tsigma}$, where $\beta_j, \tbeta$ are suppressed from the notation for these stacks to avoid excessive subscripts. Now 
    \begin{align*}
        H^0 \left( \sX_\sigma, \cO_{\sX_j} (-d) \right) &= \Bbbk \left \langle m \in M \mid \langle m, \beta_j(e_\rho) \rangle \geq F_j(\beta_j(e_\rho)) \text{ for all } \rho \in \sigma(1) \right \rangle \\
        & = \Bbbk \left \langle m \in M \mid \langle m, u_\rho \rangle \geq F_j(u_\rho) \text{ for all } \rho \in \sigma(1) \right \rangle
\\
    \text{and}\quad 
        H^0\big( \sX_\sigma, (\pi_j)_* \pi_i^* \cO_{\sX_i}(-d)\big) &= H^0 \big(\sX_{\tsigma}, \pi_i^* \cO_{\sX_i}(-d) \big) \\
        &= \Bbbk \left \langle m \in M \,\bigg\vert\, \langle m, \tbeta(e_\rho) \rangle \geq F_i ( \tbeta(e_\rho))  \text{ for all } \rho \in \tsigma(1) \right \rangle, 
    \end{align*}
    where $F_i$ and $F_j$ are support functions for $-d$ on $\Sigma_i$ and $\Sigma_j$, respectively.
    We now show that these sets of sections coincide.

    As in the proof of \cref{prop:welldefined}, $F_j(u) \geq F_i(u)$ for all $u \in |\Sigma_j| = |\tSigma|$ because $d \in \Gamma_i$.
    If $\langle m, u_\rho \rangle \geq F_j(u_\rho)$ for all $\rho \in \sigma(1)$, then since $F_j$ is linear on $\sigma$, $\langle m, u \rangle \geq F_j(u)$ for all $u \in |\sigma|$.
    Since $\tbeta(e_\rho) \in |\sigma|$ for all $\rho \in \tSigma(1)$, then for all $\rho \in \tSigma(1)$, 
    $\langle m, \tbeta(e_\rho) \rangle \geq F_j ( \tbeta(e_\rho)) \geq F_i ( \tbeta(e_\rho))$. 

    In the other direction, suppose $\langle m, \tbeta(e_\rho) \rangle \geq F_i ( \tbeta(e_\rho))$ for all $\rho \in \tSigma(1)$.
    Since $\tbeta(e_\rho) = b_\rho u_\rho$ for some $b_\rho > 0$ for all $\rho \in \sigma(1)$ and $F_i$ respects positive scaling, $\langle m , u_\rho \rangle \geq F_i(u_\rho)$ for all $\rho \in \Sigma(1)$.
    Now choose $\theta \in M_\RR$ such that $d = d(\theta)$.
    Then $F_i(u_\rho) = \lceil \langle \theta, u_\rho \rangle \rceil$ for all $\rho \in \Sigma_i(1)$, and the same holds for $F_j$ on $\Sigma_j(1)$.
    If $u \in |\Sigma_i| = |\Sigma_j|$, then $u$ lies in a cone of $\Sigma_i$ and it follows that $F_i(u) \geq \langle \theta, u \rangle$.
    In particular, if $u_\rho \in \Sigma_j(1)$, 
    $\lceil \langle \theta, u_\rho \rangle \rceil = F_j(u_\rho) \geq F_i(u_\rho) \geq \langle \theta, u_\rho \rangle$, 

    so $F_i(u_\rho) = \lceil \langle \theta, u_\rho \rangle \rceil$, since $F_i(u_\rho) \in \ZZ$. As a result, $\langle m, u_\rho \rangle \geq F_j(u_\rho)$ for all $\rho \in \sigma(1)$.

    In summary, $H^0 \left( \sX_j, \cO_{\sX_j} (-d)\right) =  H^0 \left(\sX_j, (\pi_j)_* \pi_i^* \cO_{\sX_i}(-d) \right)$ for all $\sigma \in \Sigma_j$ in a fashion that respects restriction, which is given by inclusion of polytopes, as desired.
\end{proof}

To get the higher vanishing, we will apply \cref{lem:yonedatheta} to $\pi_{j*}\pi_i^* \cO_{\sX_i}(-d)$, and show that for any twist by a nef line bundle on $\sX_j$, this will have the same cohomology as $\cO_{\sX_j}(-d)$ would.  The following lemma provides the baseline, computing the cohomology of the twists of $\cO_{\sX_j}(-d)$.
As in the proof of \cref{lem:pushforward0}, recall that given $\theta \in M_\RR$,  $-d(\theta) \in \Theta$ is obtained by setting
$d(\theta)_\rho = \lceil \langle \theta, u_\rho \rangle \rceil$ 
for all $\rho \in \Sigma(1)$ and that every element of $\Theta$ arises in this fashion by \cref{defn:BT}.

\begin{prop} \label{prop:homzero}
    Suppose that $-d \in \Theta$.  For any $1\leq j \leq r$, if $A = \sum a_\rho D_\rho$ is a nef line bundle on $\sX_j$ with section polytope $P_A$, then
    \[ H^p(\sX_j,\cO_{\sX_j}(A-d)) = \begin{cases} \Bbbk \langle P_A \cap (M - \theta) \rangle & \text{if } p = 0, \\
    0 & \text{if }p \neq 0. \end{cases}\]
\end{prop}
\begin{proof}
    Vanishing outside of degree zero follows from \cref{lem:BTvanishingCohomology}. Thus it remains only to calculate global sections. For this, recall that $H^0(\cO_{\sX_j}(A-d))$ is generated by the monomials indexed by $P_{A-d} \cap M$.
    Now choose $\theta \in M_\RR$ so that $-d = -d(\theta)$.
    Then
    \[ 
    P_{A-d} = \{ k \in M_\RR \mid \langle k , u_\rho \rangle \geq - a_\rho + \lceil \langle \theta, u_\rho \rangle \rceil \text{ for all } \rho \in \Sigma_j (1) \}. 
    \]
    If $m \in P_{A-d} \cap M$, then $m - \theta \in M- \theta$ and
    \begin{align*}
        \langle m - \theta, u_\rho \rangle = \langle m, u_\rho \rangle - \langle \theta, u_\rho \rangle 
        \geq -a _\rho + \lceil \langle \theta, u_\rho \rangle \rceil - \langle \theta, u_\rho \rangle 
        \geq -a_\rho
    \end{align*}
    for all $\rho \in \Sigma_j (1)$.
    That is, $m - \theta \in P_A \cap (M- \theta)$.

    In the other direction, if $k \in P_A \cap (M - \theta)$, then $k + \theta \in M$ and for all $\rho \in \Sigma_j (1)$, 
    $\langle k + \theta, u_\rho \rangle \geq -a_\rho + \langle \theta, u_\rho \rangle$. 
    It follows that for all $\rho \in \Sigma_j (1)$, since $\langle k + \theta, u_\rho \rangle \in \ZZ$, 
    $\langle k + \theta, u_\rho \rangle \geq -a_\rho + \lceil \langle \theta, u_\rho \rangle \rceil$. 
    Therefore $k + \theta \in P_{A-d} \cap M$.
\end{proof}

We note that the computation of global sections ($p = 0$) in \cref{prop:homzero} does not use that $A$ is nef. 
In order to compute the other side of \eqref{eq:testnef} in our case of interest, we will use the following stacky version of \cite{altmann-immaculate}*{Theorem 3.6} (see also \cite{altmann-displaying}).

\begin{theorem} \label{thm:abkw}
    Suppose that $\sX$ is a smooth toric Deligne--Mumford stack arising from a fan $\Sigma$ on $N_\RR$ and homomorphism $\beta\colon \ZZ^{\Sigma(1)} \to N$ such that $X_\Sigma$ is semiprojective and $\beta(e_\rho$) is primitive for all $\rho \in \Sigma(1)$.
    If $A$ and $B$ are nef divisors on $\sX$ with polytopes $P_A, P_B \subset M_\RR$, then for all $m \in M$, 
        \[ 
        H^p_m(\cO_{\sX}(A-B)) \cong \widetilde{H}^{p-1} (P_B \setminus (P_A - m)).
        \]
\end{theorem}
\begin{proof} This follows immediately from \cref{prop:semicohomology} and  \cite{altmann-immaculate}*{Theorem 3.6}.
\end{proof}

\begin{remark}
    While we expect that analogues of \cref{thm:abkw} hold for more general toric stacks, we do not address that here in order to simplify the proof.
\end{remark}

To apply \cref{thm:abkw} with $B = d(\theta)$, the following observation will be essential.

\begin{lemma} \label{lem:stronglystar}
    Let $\alpha \in \Sigma(1)$ and $a \in \ZZ$.
    If $\langle - \theta + m, u_\alpha \rangle \geq - a$, then there is no $k \in M_\RR$ such that $\langle k, u_\alpha \rangle \geq \lfloor - \langle \theta, u_\alpha  \rangle \rfloor $ and $\langle k +m, u_\alpha \rangle < -a $.
\end{lemma}
\begin{proof} It is enough to show that if $\langle - \theta + m, u_\alpha \rangle \geq - a$, then
$-a - \langle m , u_\alpha \rangle \leq \lfloor - \langle \theta, u_\alpha \rangle \rfloor$, 
since then $\langle k, u_\alpha \rangle \geq \lfloor - \langle \theta, u_\alpha \rangle \rfloor$ will imply that $\langle k + m, u_\alpha \rangle \geq - a$.
But, $ -a - \langle m , u_\alpha \rangle \leq - \langle \theta, u_\alpha \rangle $ implies that
$-a - \langle m , u_\alpha \rangle \leq \lfloor - \langle \theta, u_\alpha \rangle \rfloor$ because $a$ and $\langle m , u_\alpha \rangle$ are integers.
\end{proof}

We are now prepared to prove \cref{lem:thetatransform2}.

\begin{proof}[Proof of \cref{lem:thetatransform2}]
Choose $\theta \in M_\RR$ so that $-d = -d(\theta)$.
From \cref{lem:yonedatheta}, \cref{lem:pushforward0}, and \cref{prop:homzero}, it is enough to show that
\begin{equation} \label{eq:thetatransproof}
H^p(\cO_{\sX_j}(A) \otimes (\pi_j)_* \pi_i^* \cO_{\sX_i}(-d)) = \begin{cases} \Bbbk \langle P_{A} \cap (M-\theta) \rangle & \text{if } p = 0, \\ 0 & \text{if } p \neq 0,
\end{cases} 
\end{equation}
where $A$ is any nef divisor on $\sX_j$ with section polytope $P_A$. Now note that
\[ 
H^p(\cO_{\sX_j}(A) \otimes (\pi_j)_* \pi_i^* \cO_{\sX_i}(-d)) 
= H^p (\pi_j^*\cO_{\sX_j}(A) \otimes \pi_i^* \cO_{\sX_i}(-d)) 
\]
by adjunction.
Further observe that $\pi_j^* A$ and $\pi_i^* d$ are nef on $\tsX$ and their section polytopes are preserved under pullback: $P_{\pi_j^* A} = P_A$ and $P_{\pi_i^* d} = P_d$.
Thus, by applying \cref{thm:abkw}, 
$
    H^p_m(\pi_j^*\cO_{\sX_i}(A) \otimes \pi_i^* \cO_{\sX_i}(-d)) \cong \widetilde{H}^{p-1} (P_d \setminus (P_A -m))
$
for all $m \in M$.
Since $d$ has image in $\Gamma_i$, 
\begin{equation} \label{polytoped}
    P_d = \{ k \in M_\RR \mid \langle k , u_\rho \rangle \geq \lfloor - \langle \theta, u_\rho \rangle \rfloor \text{ for all } \rho \in \Sigma(1) \}.
\end{equation}
Therefore 
\begin{equation} \label{eq:polytopedifference}
P_d \setminus (P_A - m) 
= \left\{ k \in M_\RR 
\, \bigg\vert\,  
\begin{array}{l} \langle k, u_\rho \rangle \geq \lfloor - \langle \theta, u_\rho \rangle \rfloor \text{ for all } \rho \in \Sigma(1) \text{ and } \\ \langle k + m, u_\alpha \rangle < - a_\alpha \text{ for some } \alpha \in \Sigma_j(1) \end{array} \right\}.
\end{equation}

Now we check the cohomology in degree zero.
Note that $-\theta \in P_d$. Thus, $P_d \setminus (P_A - m) = \emptyset$ implies that $m - \theta \in P_A$.
The converse is an immediate consequence of \cref{lem:stronglystar}, which implies that $-\theta \in P_d \setminus (P_A - m)$ if $P_d \setminus (P_A - m) \neq \emptyset$.  Therefore \eqref{eq:thetatransproof} holds for $p = 0$.

Finally, we check that the higher cohomology vanishes.
We will show that  $P_d \setminus (P_A - m)$ is star-shaped around $-\theta$ and thus contractible when it is nonempty.
Suppose that $y \in P_d \setminus (P_A - m)$ so that there exists $\alpha \in \Sigma_j(1)$ such that $\langle y +m, u_\alpha \rangle < - a_\alpha$.
As a consequence of \cref{lem:stronglystar}, $\langle - \theta +m, u_\alpha \rangle < -a_\alpha$ so the straight-line path from $y$ to $-\theta$ is contained in $P_d \setminus (P_A - m)$, as claimed.
\end{proof}

\begin{figure}
\begin{center}
\begin{tikzpicture}[scale=0.7]
    \draw[step=3cm, gray, very thin, opacity =0.2] (-12,-6) grid (6,3);
    \draw[gray, very thin, opacity =0.2] (-12,3) -- (-12,-6) -- (6, -6);
    \foreach \x [count=\xi] in {0,1,2,3}
    \draw[gray, very thin,opacity =0.2] (6,\xi-1) -- (-12,\xi-7);
    \draw[black, very thick] (3,0) -- (6,0) --(6,-3) -- (-6, -3) -- (3,0);
    \draw[gray, very thin,opacity =0.2] (6,-1) -- (-9,-6);
    \draw[gray, very thin,opacity =0.2] (6,-2) -- (-6,-6);
    \draw[gray, very thin,opacity =0.2] (6,-3) -- (-3,-6);
    \draw[gray, very thin,opacity =0.2] (6,-4) -- (0,-6);
    \draw[gray, very thin,opacity =0.2] (6,-5) -- (3,-6);
    \draw[gray, very thin,opacity =0.2] (-12,-2) -- (3,3);
    \draw[gray, very thin,opacity =0.2] (-12,-1) -- (0,3);
    \draw[gray, very thin,opacity =0.2] (-12,0) -- (-3,3);
    \draw[gray, very thin,opacity =0.2] (-12,1) -- (-6,3);
    \draw[gray, very thin,opacity =0.2] (-12,2) -- (-9,3);
    \filldraw[gray, opacity=0.3] (3,0) -- (6,0) --(6,-3) -- (-6, -3) -- (3,0);
    \filldraw[black] (3.5,-0.3) circle (2pt);
    \node at (4, -0.3) {$-\theta$};
    \node at (5.5, -0.45) {$P_{d_5}$};
    \node at (2, -5) {$P_{4}$};
    \filldraw[black, opacity=0.35]  (3,-2) --(3, -6) -- (-9,-6) --(3,-2);
    \filldraw[blue, opacity=0.5]  (6,-1) --(6, -3) -- (0,-3) --(6,-1);
    \filldraw[red, opacity=0.2] (3,2) --(-12, -3) -- (3,-3) --(3,2);
\end{tikzpicture}
\end{center}

\caption{This figure illustrates some of the polytopes that appear in the proof of \cref{lem:thetatransform2} in the case that $\sX_i=\cH_3$ is the Hirzebruch surface $\cH_3$ from \Cref{fig:HirzSecondaryFan}.  See \cref{ex:hirzPolytopes} for a detailed description.
}
\label{fig:pushpullproof}
\end{figure}
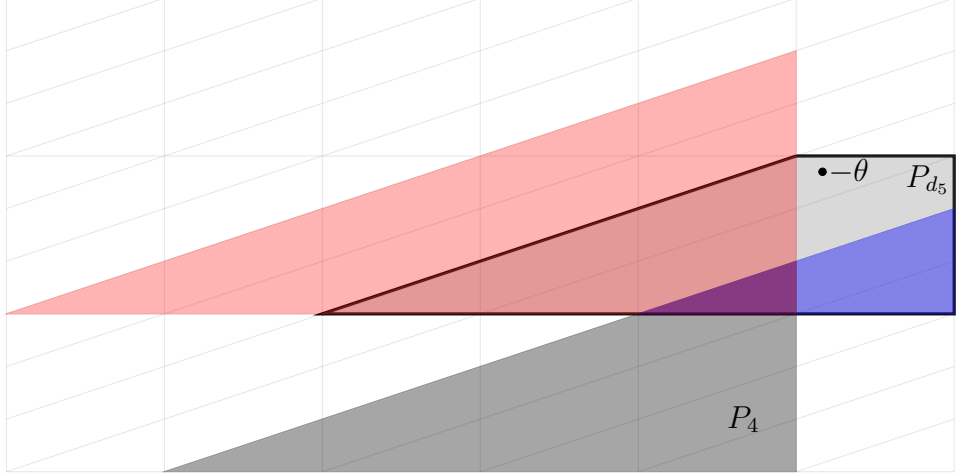

\begin{example}\label{ex:hirzPolytopes}
\Cref{fig:pushpullproof} illustrates some of the polytopes that appear in the proof of \cref{lem:thetatransform2} in the case that $\sX_i=\cH_3$ is the source of the Fourier--Mukai transform and $\sX_j=\PP(1,1,3)$ is the target.  
We choose $-d(\theta) =-d_5=(-1,-1)$, which comes from the chamber corresponding to the Hirzebruch surface $\cH_3$ (see \Cref{fig:HirzSecondaryFan}).  The polytope $P_{d_5}$ on $\cH_3$ is the outlined trapezoid.  The triangles represent $P_A$ for different ample divisors $A$ on $\PP(1,1,3)$, with the bottom triangle $P_4$ corresponding to $\cO_{\PP(1,1,3)}(4)$.
Note that $P_{d_5} \setminus P_4$ is not convex, but is nevertheless star-shaped around $-\theta$, yielding vanishing of higher cohomology.
\end{example}

\begin{remark}
    The fact that the difference of polytopes in \eqref{eq:polytopedifference} is star-shaped appears in \cite{FH2}*{Proposition 5.5} in the case that  $-A$ additionally lies in $\Theta$ by a different argument involving exit paths in stratified spaces.
    That case is enough for many of the computations in the following subsection, but it does not imply the entirety of the $\Theta$-Transform Lemma.
\end{remark}

\begin{remark} \label{rem:borisov}
After this paper was initially posted, Lev Borisov graciously explained to us how to prove \cref{lem:thetatransform2} more directly as we now outline.
As in the proof above, \cref{lem:pushforward0} gives a direct calculation of $R^0 \pi_{j \ast} \pi_i^\ast \cO_{\sX_i}(-d)$, and given that, we need only show that the higher cohomology vanishes. 
A natural approach is to proceed similarly to \cref{prop:pushforwardBTcoarse} and prove that the higher cohomology of $\pi_i^\ast \cO_{\sX_i}(-d)$ vanishes on the open toric substacks of $\tsX$ lying over the toric charts of $\sX_j$. 

To do that, one would like to produce a homotopy equivalence similar to the proof of \cref{lem:BTvanishingCohomology}.
More precisely, if $F_i$ is the support function for $-d = -d(\theta)$ on $\sX_i$ and $\sigma$ is any cone of $\Sigma_j$, we should show that the set 
\[ V_{F_i, m} = \{ u \in \sigma :  \langle m, u \rangle < F_i(u) \} \] 
retracts onto the convex hull of the set of rays of $\sigma(1)$ where $\langle \theta - m, \cdot \rangle$ is positive, which coincides with the rays where $F_i(\cdot) - \langle m, \cdot \rangle$ is positive.
We were unable to find a direct argument as $V_{F_i, m}$ is not convex.
However, one can replace $V_{F_i,m}$ with the complementary set $V_{F_i,m}^c$ where $F_i(\cdot) - \langle m, \cdot \rangle$ is non-positive by an alternative description of the cohomology of a toric line bundle given in \cite{borisov-hua}*{Proposition 4.1}, which is Alexander dual to the description in \cref{subsec:cohomology}.
Then, a linear retraction is possible since the convexity of $F_i$ implies that $V_{F_i,m}^c$ is convex.

Although this argument circumvents the need for \cref{lem:yonedatheta} and \cref{thm:abkw}, we have kept our original argument as it hews closer to how we originally understood the statement to be true through mirror symmetry, and we hope it gives additional insight into the cohomology of the Bondal-Thomsen collection.
\end{remark}

Heuristics and computations lead us to believe that the following also holds.

\begin{conjecture}[Uniform higher $\Theta$ vanishing]
If $-d \in\Theta$ and $i,j\in\{1,\dots,r\}$, then 
$
R^{>0}\pi_{j\ast}\pi_i^\ast \mathcal O_{\sX_i}(-d) = 0
$; 
that is, the Fourier--Mukai transform of any Bondal--Thomsen element from any chamber is always a sheaf.
\end{conjecture}

\subsection[Computations in the Cox category]{Computations in the Cox category using the $\Theta$-Transform Lemma} \label{subsec:thetacomputations}
We now turn to some consequences of \cref{lem:thetatransform2} for $D_{\Cox}$.
In particular, we compute morphisms in $D_{\Cox}$ as follows.
If $-d,-d' \in \Theta$, $d$ has image in $\Gamma_i$, and $d'$ has image in $\Gamma_j$, then \cref{lem:thetatransform2} combined with adjunction yield 
\begin{align*}
        \RHom(\cO_{\Cox}(-d), \cO_{\Cox}(-d')) &= \RHom(\pi_i^* \cO_{\sX_i}(-d), \pi_j^* \cO_{\sX_j}(-d')) \\
        &= \RHom( \cO_{\sX_i}(-d), (\pi_{i})_* \pi_j^* \cO_{\sX_j}(-d')) \\
        &= \RHom( \cO_{\sX_i}(-d), \cO_{\sX_i}(-d')).
\end{align*}
Thus, applying \cref{prop:homzero}, we immediately obtain the following.  This shows that sheaf cohomology computations in $D_{\Cox}$ are combinatorial in nature.

\begin{cor} \label{cor:partialorder}
    Let $-d, -d' \in \Theta$.  For any $\theta'$ such that $d' = d(\theta')$, 
    \begin{equation} \label{eq:Coxhom} \RHom (\cO_{\Cox}(-d), \cO_{\Cox}(-d')) = \Bbbk \langle P_{d} \cap (M-\theta') \rangle,
    \end{equation}
    concentrated in degree zero.
    In particular, if there are any nonzero morphisms from $\cO_{\Cox}(-d)$ to $\cO_{\Cox}(-d')$, then $d - d'$ is effective.
    \hfill$\square$
\end{cor}

These morphism spaces are functorially identified with morphisms of $S$-modules.

\begin{cor}\label{cor:SmodHom}
 If $-d,-d',-d''\in\Theta$, then there are isomorphisms $\phi_{d-d'}$ from $S_{d-d'}$ to $ \RHom(\cO_{\Cox}(-d),\cO_{\Cox}(-d'))$ such that
 \vspace*{-2mm}
    \begin{enumerate}[noitemsep]
        \item Composition of morphisms $ \cO_{\Cox}(-d)\to \cO_{\Cox}(-d')
        \to \cO_{\Cox}(-d'')$ is identified with multiplication of polynomials $S_{d-d'}\otimes S_{d'-d''}\to S_{d-d''}$, and
        \item For all $i$, $(\pi_i)_* \circ \phi_{d-d'}$ coincides with restriction from $S$ to $\sX_i$; that is, there is a commutative diagram where $\rho$ is the natural restriction functor from $D(S)$ to $D(\sX_i)$: 
        \[
        \xymatrix{
        S_{d-d'}\ar[rrr]^-{\phi_{d-d'}}\ar[rrrd]_{\rho}&&&\RHom(\cO_{\Cox}(-d),\cO_{\Cox}(-d'))\ar[d]^{\pi_{i*}}\\
        &&&\RHom(\cO_{\sX_i}(-d),\cO_{\sX_i}(-d')).
        }
        \]
    \end{enumerate}
\end{cor}
\begin{proof}
    Suppose that $\cO_{\Cox}(-d) = \pi_i^* \cO_{\sX_i}(-d)$.
    From \cref{cor:partialorder}, it is enough to show that $P_d \cap (M-\theta')$ is in bijection with monomial generators of $S_{d-d'}$, where $P_d$ is given by \eqref{polytoped}.
    By definition, the monomial generators of $S_{d-d'}$ are in bijection with integral points of
    \[ Q_{d-d'} = \left\{ k \in M_\RR \mid \langle k , u_\rho \rangle \geq - \lceil \langle \theta, u_\rho \rangle \rceil + \lceil \langle \theta', u_\rho \rangle \rceil \text{ for all } \rho \in \Sigma (1) \right\}. \]
    However, $Q_{d-d'} \cap M = P_d \cap (M-\theta')$ by the same argument as in \cref{prop:homzero}.  Thus
    \[ 
    \RHom(\cO_{\Cox}(-d),\cO_{\Cox}(-d')) \cong S_{d-d'}.
    \]

    To see that these identifications respect composition, note that composition in $D_{\Cox}$ is given by addition of integral points in the section polytopes of line bundles on $\widetilde{\sX}$. 
    Similarly, multiplication of polynomials can be realized by multiplying monomials, which corresponds to addition of lattice points.
    Since section polytopes are preserved under pullback, we have shown above that the section polytope of $\cO_{\Cox}(d-d')$ is $Q_{d-d'}$, which demonstrates that composition is identified with multiplication of polynomials.

    Checking the final statement amounts to unraveling the definitions. The polytope $Q_{d-d'}$, whose lattice points index the monomial basis of $S_{d-d'}$, is a subset of the section polytope of $d-d'$ on any $\sX_i$, and the pushforward functor respects the monomial grading.
\end{proof}

\begin{example}\label{ex:CoxHoms}
We illustrate \cref{cor:partialorder} for a pair of Bondal--Thomsen elements $-d(\theta)$ and $-d(\theta')$ on $\mathcal H_3$ in \Cref{fig:BThom}. In \Cref{fig:HirzSecondaryFan}, these correspond to $d(\theta) = d_5=(1,1)$ and $d(\theta') = d_2=(-2,1)$.  
\begin{figure}
\begin{center}
\begin{tikzpicture}
    \draw[step=3cm, gray, very thin, opacity =0.2] (-6,-3 ) grid (6,0);
    \draw[gray, very thin, opacity =0.2] (-6,0) -- (-6,-3) -- (6, -3);
    \foreach \x in {0,1,2,3}
    \draw[gray, very thin,opacity =0.2] (6,\x-3) -- (6-3*\x,-3)
    (3-\x-\x-\x,0) -- (-6, -3 +\x);
    \filldraw[gray, opacity=0.3] (3,0) -- (6,0) --(6,-3) -- (-6, -3) -- (3,0);
    \filldraw[white] (3.5,-0.3) circle (2pt);
    \node at (3.9, -0.3) {$-\theta$};
    \node at (5.5, -0.7) {$P_{d_5}$};
    \foreach \x in {0,1,2,3}
    \filldraw[black] (5.5-3*\x,-2.7) circle (2pt);
    \node at (5.5,-2.3) {$-\theta'$};
\end{tikzpicture}
\end{center}

\caption{This figure demonstrates a sample computation of $\Hom(\cO_{\Cox}(-d),\cO_{\Cox}(-d'))$.  See \cref{ex:CoxHoms} for a detailed description.}
\label{fig:BThom}
\end{figure}
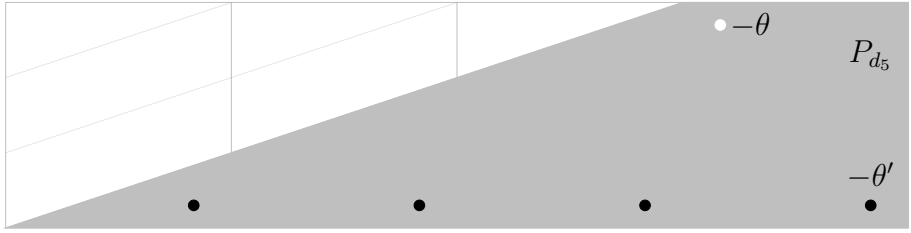
The gray trapezoid is $P_{d_5}$, which contains four translates of $-\theta'$.  Thus 
\begin{equation*} \RHom_{D_{\Cox}} (\cO_{\Cox}(-d(\theta)), \cO_{\Cox}(-d(\theta'))) = \Bbbk \langle P_{d} \cap (M-\theta') \rangle \cong \Bbbk^{4}.
\qedhere
    \end{equation*}
\end{example}

Rephrasing \cref{cor:partialorder}, we obtain:

\begin{cor}\label{cor:strongExceptional}
    If $\mathcal T_{\Cox}\ce\bigoplus_{-d\in \Theta} \cO_{\Cox}(-d)$ as in the statement of Theorem~\ref{thm:main1}, then $\RHom_{D_{\Cox}}(\mathcal T_{\Cox},\mathcal T_{\Cox})$ is concentrated entirely in homological degree $0$.  If $X$ is projective and the order on $\Theta$ is taken as in Definition~\ref{defn:ordering}, then $\Theta$ is a strong exceptional collection.
    \hfill $\square$
\end{cor}

As the morphisms in \eqref{eq:Coxhom} are entirely characterized by the polytope $P_d$ and $\theta'$, it is worth making a few comments on the combinatorial nature of $P_d$.
Let $\theta$ be such that $d = d(\theta)$ and set $a(\theta)_\rho = \lfloor - \langle \theta, u_\rho \rangle \rfloor$ for $\rho \in \Sigma(1)$.
Then since $d \in \sigma$, 
\[ 
P_d = \{ k \in M_\RR \mid \langle k, u_\rho \rangle \geq a(\theta)_\rho \text{ for all } \rho \in \Sigma(1) \}. 
\]
Further, $-\theta \in P_d$ as noted previously and enjoys the special property that, in addition, for all $\rho \in \Sigma(1)$, 
$\langle - \theta, u_\rho \rangle < a(\theta)_\rho + 1$; that is, $a(\theta)_\rho$ is the best possible integral lower bound for $\langle - \theta, u_\rho \rangle$ simultaneously for all $\rho \in \Sigma(1)$.
This provides yet  another combinatorial description of $\Theta$; namely, the Bondal--Thomsen collection corresponds to the set of polytopes with normal fan a subset of $\Sigma(1)$ that contain a point for which the defining hyperplanes are optimal lower bounds simultaneously for all $\rho \in \Sigma(1)$ (including virtual facets).

\section[Generating the Cox category]
{Generating $D_{\Cox}$}\label{sec:generation}

In this section, we prove that $D_{\Cox}$ is generated by the $\cO_{\Cox}(-d)$ for $-d \in \Theta$.
As this is the final ingredient for \cref{thm:main1}, we then record the proof of that result in \cref{sec:proofsMainResults}.

\subsection[Proof that Theta generates the Cox category]{Proof that $\Theta$ generates $D_{\Cox}$}\label{subsec:generation}

\begin{prop}\label{prop:generation}
    The Cox category $D_{\Cox}(X)$ is generated by $\cO_{\Cox}(-d)$ with $-d \in \Theta$.
\end{prop}
\begin{proof}
Let $\theta \in M_{\RR}$, and let $-\widetilde{d}\ce
\widetilde{d}(\theta) \in \Cl(\tsX)$ and $-d\ce -d(\theta)\in \Cl(X)$ be the
corresponding elements of the class groups.  Suppose that the image of $-d$ is
in $\Gamma_i$, the chamber of $\Sigma_{GKZ}$ corresponding to $\sX_i$ so that,
by Definition~\ref{defn:ThetaCox}, 
$\cO_{\Cox}(-d)=\pi_i^*\cO_{\sX_i}(-d)$.  By adjunction and
 \cref{prop:pushforwardBTstack}, 
$\RHom_{\tsX}(\cO_{\Cox}(-d), \cO_{\tsX}(-\widetilde{d}))\cong \RHom_{\sX_i}(\cO_{\sX_i}(-d), \cO_{\sX_i}(-d))$.
The latter $\Hom$ group contains an identity morphism, so write $f\colon \cO_{\Cox}(-d) \to \cO_{\tsX}(-\widetilde{d})$ for the corresponding morphism.
Let $C(f)$ be the cone of $f$.  We next observe that $\pi_{j*}(C(f))=0$ for all $j$.  Namely, by the $\Theta$-Transform Lemma and \cref{prop:pushforwardBTstack} as referenced above,  $\pi_{j*}(C(f))$ is
\[
 \cO_{\sX_j}(-d) \xrightarrow{\ \pi_{j*}(f)\ }\cO_{\sX_j}(-d).
\]
Since $f$ was defined as the adjoint to the identity on $\cO_{\sX_i}(-d)$, and since the identity, $f$, and $\pi_{j*}(f)$ all agree over the torus, it follows that $\pi_{j*}(C(f))=0$.

Now, recall that $\mathcal Q\subseteq D(\tsX)$  was defined in
\cref{subsec:BTdefinitions} as the subcategory of $\cE\in D(\tsX)$ such that 
$\pi_{i*}\cE =0$ for all $i$, and consider the Verdier quotient 
$\pi \colon D(\tsX) \to D(\tsX)/\mathcal Q$. 
Since $D_{\Cox} \subset {}^\perp \mathcal Q$, the composition 
$D_{\Cox} \to D(\tsX) \to D(\tsX)/\mathcal Q$ 
is fully faithful. 
The line bundles $\cO_{\tsX}(-\widetilde{d})$ 
for $-\widetilde{d} \in \Theta_{\tsX}$ generate $D(\tsX)$ by~\cite{HHL}*{Corollary D} or \cite{FH}, and 
hence they also generate $D(\tsX)/\mathcal Q$. From the above, 
$\pi (\mathcal O_{\Cox}(-d)) = 
\pi (\mathcal O_{\tsX}(-\widetilde{d}))$.

Thus both functors
$\langle \mathcal O_{\Cox}(-d) \rangle \subset D_{\Cox} \to 
D(\widetilde{X}) / \mathcal Q$ 
are equivalences.
\end{proof}

We derive some immediate corollaries including a tilting 
equivalence for $D_{\Cox}$ and associated semiorthogonal decompositions.

\begin{defn} \label{def:BT_category}
Let $\mathcal P_{\Theta}$ be the category whose objects are 
$-d \in \Theta$ and whose morphisms are given by inclusion
$\psi_{S} \colon \cP_\Theta \to D(S)$, where 
$-d \mapsto S(-d)$.
\end{defn}

\begin{defn} \label{def:Psimap}
\cref{cor:SmodHom} implies that there is a functor $\psi_{\Cox}\colon
\mathcal P_{\Theta}\to D_{\Cox}$ that sends $S(-d)\mapsto \cO_{\Cox}(-d)$. 
This naturally extends to a functor $\Psi\colon K_{\Theta}(S) \to D_{\Cox}$, where we recall from \cref{conv:functors} that $K_\Theta(S)$ denotes the homotopy category of $\cP_{\Theta}$.
\end{defn}

\begin{cor} \label{cor:tilting_equiv}
The functor $\Psi$ is an equivalence.
\end{cor}

\begin{proof}
\cref{cor:SmodHom} implies that $\Psi$ is fully faithful. 
\cref{prop:generation} implies it is essentially surjective.
\end{proof}

\begin{cor}\label{cor:semiorthogonalDCox}
There are semiorthogonal decompositions 
$
D(\tsX) = \langle \mathcal Q ,D_{\Cox}\rangle = \langle D_{\Cox}, \mathcal Q^\vee \rangle.
$
\end{cor}
\begin{proof}
Since $D_{\Cox} \subseteq {}^\perp \mathcal Q$, the map 
$\pi \colon  \operatorname{Hom}_{D(\tsX)}(E,F) \to 
\operatorname{Hom}_{D(\tsX)/\mathcal Q}(\pi(E),\pi(F))$ 
is an isomorphism for $E \in D_{\Cox}$. Since the composition 
$D_{\Cox} \to D(\tsX)/\mathcal Q$ is an equivalence, 
there is a right adjoint $r$ to the inclusion $i$ of $D_{\Cox}$. 
The counit of the adjunction gives a functorial triangle $(i\circ r )(F) \to F \to C$, 
with $(i\circ r )(F) \in D_{\Cox}$ and $C \in \mathcal Q$ for any 
$C \in D(\tsX).$
Applying $(-)^\vee$ to $\langle \mathcal Q, D_{\Cox}\rangle$ 
produces a new semiorthogonal decomposition 
$
D(\tsX) = \langle D_{\Cox}^\vee, \mathcal Q^\vee \rangle = 
\langle D_{\Cox}, \mathcal Q^\vee \rangle.
$
The last equality is \cref{lem:self_dual}.
\end{proof}

\subsection{Proof of Theorem~\ref{thm:main1}}\label{sec:proofsMainResults}
Having assembled the pieces, we now prove \cref{thm:main1} and \cref{prop:dgThm}.

\begin{proof}[Proof of \cref{thm:main1}] 
By Proposition~\ref{prop:generation}, 
$\Theta$ generates $D_{\Cox}$.
Corollary~\ref{cor:partialorder} shows that if $\mathcal T\ce\bigoplus_{-d\in \Theta} \cO_{\Cox}(-d)$, then $\Hom_{D_{\Cox}}(\mathcal T, \mathcal T)$ is concentrated in degree $0$ and thus $\mathcal T$ is a tilting bundle.
Moreover, that corollary shows that the only morphisms arise from when $d-d'$ is effective.
In the projective case, if we therefore order the elements of $\Theta$ as in Definition~\ref{defn:ordering}, then we immediately deduce the exceptional collection condition.
\end{proof}

\begin{proof}[Proof of \cref{prop:dgThm}]
\cref{lem:self_dual} has already shown that $D_{\Cox}(X)$ is 
self-dual.
The fact that $D_{\Cox}(X)$ is homologically smooth follows from the fact that it is an admissible subcategory of $D(\tsX)$ by \cref{cor:semiorthogonalDCox}. 
This observation also implies that $D_{\Cox}(X)$ has Rouquier dimension equal to $\dim X$ because $D(\tsX)$ has Rouquier dimension equal to $\dim \tsX$ by \cites{HHL, FH}, and $\dim \tsX=\dim X$.
If $X$ is projective, then $D_{\Cox}(X)$ is also proper because it is a subcategory of $D(\tsX)$.
\end{proof}

\section[A diagonal resolution for the Cox category]{A diagonal resolution for $D_{\Cox}$} \label{sec:diagonal}
In this section, we prove that the Hanlon--Hicks--Lazarev resolutions naturally lift to a resolution of the diagonal for the Cox category after recalling the main result of \cite{HHL}. 
In \cref{subsec:uniformityHHL}, we will establish some key properties of the resolutions from \cite{HHL} and the closely related Fourier--Mukai transforms.
Then we construct a resolution of the diagonal for $D_{\Cox}$ in \cref{subsec:variousFunctors} and deduce a few consequences such as \cref{prop:dgThm}.

\subsection{Uniformity properties of the HHL resolution} \label{subsec:uniformityHHL}
\begin{thm}[Theorem A of \cite{HHL}] \label{thm:hhlmain}
    If $\sY$ is a closed toric substack of a toric stack $\sX$, and $\sX$ is covered by smooth stacky charts, then there is a resolution of $\cO_{\sY}$ by direct sums of elements $\cO_{\sX}(-d)$ with $-d\in \Theta$.
\end{thm}

We denote the complex from \cref{thm:hhlmain} by $\HHL_{\cY,\cX}$.
We will focus on the case that $\Delta_{ij} \subset \sX_i \times \sX_j$ is the
graph of the birational morphism from $\sX_i$ to $\sX_j$. Note that, when $i=j$,
$\Delta_{ii}\subseteq \sX_i\times \sX_i$ is simply equal to the diagonal.  

\begin{remark} \label{rmk:characteristic}
In \cite{HHL}*{Appendix B.1.1}, an assumption on the characteristic of the field
was employed to split a 
representation into irreducible representations. However, 
the result stated there holds whenever the group $G$ is 
linearly reductive; see, e.g.,~\cite{Waterhouse}. Its only 
application in \cite{HHL}*{Lemma 2.22} is in this setting.
\end{remark}

\begin{defn}\label{defn:HHLnotation}
    To simplify notation, write $\HHL_{i,j}$ for the Hanlon--Hicks--Lazarev line bundle complex that resolves $\Delta_{ij}$ in $\sX_i\times \sX_j$.
\end{defn}

Solely within this section, we consider an auxiliary stack $\mfX$ that is obtained by gluing the $\sX_i$ along the common birational loci.  A bit more formally, we define $\mfX$ as follows:  recall that each $\sX_i$ has the form $[\Spec(S) - V(B_i) / G]$ for an appropriate irrelevant ideal $B_i\subseteq S$.  We define $B \ce  \sum_i B_i$ and set $\mfX \ce  [\Spec (S) - V(B) / G]$.  We note that each $\sX_i \subseteq \mfX$ is an open subset, and thus $\mfX$ is covered by smooth stacky charts.

We define $\HHL_{S,S}$ to be the free complex of $S\otimes S$-modules that corresponds to the resolution $\HHL_{\mfX, \mfX\times \mfX}$ of the toric subscheme supported on the closure of the diagonal torus. 
(The diagonal itself may fail to be closed as $\mfX$ may fail to be separated.) 
We observe a key uniformity result about these resolutions that was implicit in~\cite{HHL}.
\begin{prop}\label{prop:uniformHHL}
For any $i,j$, $\HHL_{S,S}|_{\sX_i\times \sX_j}$
is homotopy equivalent to $\HHL_{i,j}$  as complexes of sheaves on $\sX_i\times \sX_j$.  Moreover, if both chambers $\Gamma_i$ and $\Gamma_j$ for $\sX_i$ and $\sX_j$ are in the moving cone of $\Sigma_{GKZ}(X)$, then these complexes are equal.
\end{prop}
\begin{proof}
   The key idea is to iteratively apply \cite{HHL}*{Lemma 3.3}, which provides a statement about the functorial behavior of $\HHL_{\cY,\cX}$
   with respect to pullback along an open inclusion.  While \cite{HHL}*{Lemma 3.3} is only stated for removal of a single toric divisor, it can clearly be iterated.  Moreover, the statement also holds when the map $\alpha$ from the cited lemma involves removing higher codimension toric strata, since in that case the complex $\alpha^*\HHL_{\cY,\cX}$ coincides with $\HHL_{\cY',\cX'}$, making the statement trivially true.
   Now, since $\sX_i\times \sX_j$ can be obtained from $\mfX\times \mfX$ by removing the toric strata $V(B_i)\times \mfX \cup \mfX \cup V(B_j)$, both statements follow immediately from these observations. In particular, the latter statement corresponds to the fact that the construction in~\cite{HHL} only depends on the rays of the stacky fan.
\end{proof}

\subsection[An HHL resolution for the Cox category]{An HHL resolution for
$D_{\Cox}$} \label{subsec:variousFunctors}
We now define a version of the Hanlon--Hicks--Lazarev resolution, denoted $\HHL_{\Cox,\Cox}$, that lives in the Cox category, use this to define a corresponding Fourier--Mukai transform, and then prove that this transform is naturally isomorphic to the identity.

Recall that we defined a complex $\HHL_{S,S}$ of $S\otimes S$-modules in
\cref{subsec:uniformityHHL} with terms of the form $S(-d)\boxtimes S(-d')$,
where $-d,-d'\in \Theta$.\footnote{The main result of \cite{BCHSY} shows that $\HHL_{S,S}$ is homotopy equivalent to the free resolution construction in~\cite{short-HST}.  In particular, $\HHL_{S,S}$ is acyclic as a complex of $S\otimes S$-modules.}

\begin{defn}\label{defn:HCoxCox}
Define $\HHL_{\Cox,\Cox}\ce (\Psi \times \Psi)(\HHL_{S,S})$, with $\Psi\colon K_{\Theta}\to D_{\Cox}$ as in \cref{def:Psimap}. 
\end{defn}

\begin{example}
    If $X=\PP^1$, then $\HHL_{S,S}$ is the complex $S\boxtimes  S \longleftarrow S(-1)\boxtimes S(-1)$ given by multiplication by $x_0y_1 - x_1y_0$, and $\HHL_{\Cox,\Cox}$ is simply $\cO_{\Cox}\boxtimes \cO_{\Cox}\longleftarrow  \cO_{\Cox}(-1)\boxtimes \cO_{\Cox}(-1)$ with the same map.
\end{example}

\begin{defn}\label{defn:PhiCoxCox}
Let $\Phi_{\Cox,\Cox}\colon D_{\Cox}\to D_{\Cox}$ be the integral transform with kernel $\HHL_{\Cox,\Cox}$.
\end{defn}

We separate the proof of \cref{thm:resDiagCoxVague} into two steps, proving the following lemma and corollary, each of which strengthens part of that theorem.
\begin{lemma}\label{lem:PhiCoxCD}
For each pair $i,j$, the following holds:
\vspace*{-2mm}
\begin{enumerate}[noitemsep]
    \item  The complex $(\pi_i\times \pi_j)_\ast(\HHL_{\Cox,\Cox})$ is homotopy equivalent to $\HHL_{i,j}$.
    \item  This induces a commutative diagram
\begin{equation}\label{eqn:CDfunc}
\xymatrix{
D(\sX_i)\ar[d]_-{\pi_i^*}\ar[r]^-{\Phi_{i,j}}&D(\sX_j)\\
D_{\Cox}\ar[r]^-{\Phi_{\Cox,\Cox}}&D_{\Cox}.\ar[u]_-{\pi_{j*}}
}
\end{equation}
\end{enumerate}
\end{lemma}

\begin{cor}\label{cor:strongResDiag}
The functor $\Phi_{\Cox,\Cox}\colon D_{\Cox}(X)\to D_{\Cox}(X)$, which is
the Fourier--Mukai transform with kernel $\Psi(\HHL_{S,S})$, is naturally
isomorphic to the identity functor.
\end{cor}

As noted in the introduction, this ``explains'' the uniformity properties of the Hanlon--Hicks--Lazarev resolution: the resolution only depends on the Cox ring (or the rays of the fan)--and not on the irrelevant ideal (or the higher dimensional cones of the fan)--because the Hanlon--Hicks--Lazarev resolution comes from a resolution of the diagonal for $D_{\Cox}$, and $D_{\Cox}$ is an invariant of the Cox ring $S$.

\begin{proof}[Proof of \cref{lem:PhiCoxCD}]
For claim (1), 
the $\Theta$-Transform Lemma shows that
$\pi_{i \ast}(\cO_{\Cox}(-d))=\cO_{\sX_i}(-d)$ for all $i$ and $-d$.  Using \cref{cor:SmodHom}, it follows
that, as functors from $K_\Theta(S)\to D(\sX_i)$,  the composition $\pi_{i \ast}\circ
\Psi$ is naturally isomorphic to $(-)|_{\sX_i}$.  By combining with
\cref{prop:uniformHHL}, we obtain the desired statement.
Claim (2) is a standard consequence; see, e.g.,
\cite{huybrechts06}*{Section 5.1}.
\end{proof}

\begin{proof}[Proof of \cref{cor:strongResDiag}]
By Theorem~\ref{thm:main1}, every object in $D_{\Cox}$ can be represented as a complex $C_{\bullet}$, where each $C_i$ is a direct sum of $\cO_{\Cox}(-d)$ with $-d\in \Theta$.  We will prove that $\Phi_{\Cox}$ is the identity on $\cO_{\Cox}(-d)$ for all $-d\in \Theta$ and that it is the identity on all morphisms between  these bundles.
Since $\Phi_{\Cox}$ is a functor of triangulated categories -- so $\Phi_{\Cox}$ commutes with cones -- it will follow that $\Phi_{\Cox}(C_\bullet)$ is naturally equivalent to $C_\bullet$ itself.

Starting with objects, let $G=\Phi_{\Cox}(\cO_{\Cox}(-d))\in D_{\Cox}$.  We want to show that $G= \cO_{\Cox}(-d)$.  Let $i$ be such that $\cO_{\Cox}(-d) = \pi_i^*\cO_{\sX_i}(-d)$.  Consider Lemma~\ref{lem:PhiCoxCD} in the case $i=j$.  In this case, the Fourier--Mukai transform $\Phi_{i,i}$ is just the identity, and so  $\pi_{i*}(G)$ is isomorphic to $\cO_{\sX_i}(-d)$.
By adjunction, the identity morphism from $\cO_{\sX_i}(-d)$ to $\pi_{i*}(G)$ induces a morphism $f\colon \cO_{\Cox}(-d) \to G$.

Now choose some other $j$, and consider $\pi_{j*}(f)\colon  \pi_{j*}\cO_{\Cox}(-d) \to \pi_{j*}G$.  By \Cref{lem:PhiCoxCD} and the $\Theta$-Transform Lemma, this is simply a map
\[
\pi_{j*}(\cO_{\Cox}(-d)) = \cO_{\sX_j}(-d)\xrightarrow{\ \pi_{j*}f\ }\cO_{\sX_{j}}(-d)=  \pi_{j*}G.
\]
Finally, the morphisms $\pi_{i*}f=\Id$ and $\pi_{j*}f$ must agree when we pass to the open torus inside of $\sX_i$ and $\sX_j$.
Since the restriction from $\sX_j$ to the torus $T$ is injective on homomorphisms of line bundles, this implies that $\pi_{j*}f$ must also equal the identity.
We conclude that the map $f\colon \cO_{\Cox}(-d)\to G$ satisfies $\pi_{j*}f = \Id_{\cO_{\sX_j}(-d)}$ for all $j$, and it follows that $G=\cO_{\Cox}(-d)$ and $f$ is a quasi-isomorphism by \cref{lemma:checkpush}.

Now consider what $\Phi_{\Cox,\Cox}$ does to morphisms.
Let $f\in \Hom(\cO_{\Cox}(-d), \cO_{\Cox}(-d')) = S_{d-d'}$.
The Fourier--Mukai transform $\Phi_{\Cox}$ commutes with the open immersion $\iota\colon T \subseteq \tsX$, yielding a diagram:
\[
\xymatrix{
D_{\Cox}\ar[r]^-{\Phi_{\Cox,\Cox}}\ar[d]_-{\iota^*}&D_{\Cox}\ar[d]_-{\iota^*}\\
D(T)\ar[r]^-{\Phi_{T,T}}&D(T),
}
\]
where $\Phi_{T,T}$ is the Fourier--Mukai transform with respect to the restriction of the Hanlon--Hicks--Lazarev resolution to $T$.  We note that $\iota^*$ is injective on homomorphisms of line bundles, and so it suffices to prove that  $\iota^* f$ equals $\Phi_{T,T}\circ \iota^*(f)$, but this follows from the fact that $\Phi_{T,T}$ is naturally isomorphic to the identity functor.  In particular, $\Phi_{\Cox,\Cox}$ is naturally the identity on morphisms as well.

As noted at the beginning of the proof, this implies that $\Phi_{\Cox,\Cox}$ is naturally isomorphic to the identity functor.
\end{proof} 

\begin{remark}
An alternative natural approach to \cref{cor:strongResDiag} is to project the
Hanlon--Hicks--Lazarev complex resolving the diagonal on $\tsX$ to $D_{\Cox}$. Using the
self-duality of $D_{\Cox}$, one can check that the projected complex resolves
the diagonal of $D_{\Cox}$, and one could prove \cref{cor:strongResDiag} by
showing that $\Psi(\HHL_{S,S})$ is homotopic to this projected complex. However,
to complete that argument, we require stronger functoriality statements about
pushforwards along stacky refinements than can be obtained from
\cite{HHL}*{Lemma 3.3}. Therefore, we take the more direct approach here and leave the interesting
functoriality statements for future work.

On the other hand, the desired homotopy indeed must hold a posteriori as a
consequence of \cref{cor:strongResDiag} and strong exceptionality of $\Theta$ in
$D_{\Cox}$.
\end{remark}

\begin{remark}
Another perspective on Corollary~\ref{cor:strongResDiag} is that
${\HHL}_{\Cox,\Cox}$ is the image of the diagonal bimodule
$\cO_{\Delta_{\Cox}}$ under the equivalence $D_{\Cox}(X\times X)\to
K_{\Theta\times \Theta}(S\otimes S)$.  
Work in progress of Berkesch--Cranton
Heller--Smith--Yang \cite{BCHSY} will show that the Hanlon--Hicks--Lazarev
resolutions are homotopy equivalent to the Brown--Erman resolutions
from~\cite{short-HST}.   The resolutions of Anderson~\cite{anderson}, however, are distinct.
\end{remark}

\begin{remark}
Let $A_\Theta$ be the endomorphism algebra
as in the statement of \cref{thm:noncomm}.
 Theorem~\ref{thm:main1} also implies that $D_{\Cox}(X) \simeq  K_{\Theta}(S) \simeq D(A_\Theta)$ and $D_{\Cox}(X\times X)\to K_{\Theta\times \Theta}(S\otimes S) \simeq D(A_\Theta \otimes_k A_\Theta^{op})$.  Hence, resolutions of $\cO_{\Delta_{\Cox}}$ can also be obtained from resolutions of the diagonal bimodule  $A_\Theta \in D(A_\Theta \otimes_k A_\Theta^{op})$. 
The subtleties of the comparison between noncommutative resolutions and the Hanlon--Hicks--Lazarev resolution are explained in \cite{FS}.
 
 Other resolutions can also be obtained by altering the equivalence and/or tilting object.
 Each such choice yields a different uniform resolution of the diagonal.
 For instance, $\Theta$ is defined in terms of combinations $\sum a_iD_i$, where $a_i\in (-1,0]$.
 However, from the symplectic viewpoint, it is equally natural to allow $a_i \in [-1,0)$.
 Writing $\omega = \sum -\deg(x_i)$ as the canonical degree, the corresponding integral divisors are
 \[
 \Theta^* \ce \omega - \Theta = \{\omega + d \mid -d \in \Theta\}.
 \]
 This yields an equivalence $\Psi'\colon D_{\Cox}(X\times X) \to K_{\Theta^* \times \Theta^*}(S\otimes S)$ and $\Psi'(\cO_{\Delta_{\Cox}})$ is a uniform resolution of the diagonal with terms $S(d)\otimes S(d')$ with $d,d'\in \Theta^*$.
\end{remark}

\section{Applications}\label{sec:applications}

In this section, we explore how our main results about $D_{\Cox}$ connect to prior work on related questions, and we obtain a number of applications. 

\subsection{Representations of the Bondal--Thomsen category}
\label{subsec:BT_reps}

The Bondal--Thomsen collection can be viewed in different categories; for example, $-d\in \Theta$ yields $\cO_{\sX_i}(-d)$, $S(-d)$, and $\cO_{\Cox}(-d)$. Similarly, there are resolutions of the diagonal in $D(\sX_i\times \sX_i)$~\cite{HHL}, $D(S\otimes S)$~\cite{short-HST},  $D(A_{\Theta}\otimes A_{\Theta}^{op})$~\cite{FS}, and the Cox category (as in \cref{sec:diagonal}).
 In this brief subsection, we provide some basic language for comparing these objects and the associated Fourier--Mukai transforms.

Recall from \cref{def:BT_category} that $\cP_\Theta$ denotes the category of graded free $S$-modules whose generators have degrees lying in $\Theta$.  
There are natural functors with $\cP_{\Theta}$ as the source.

\begin{example}\label{ex:BTfunctors}
There is a canonical functor $\psi_S \colon \cP_{\Theta}\to D(S)$; for each $1\leq i \leq r$, there is $\psi_{\sX_i}\colon \cP_{\Theta}\to D(\sX_i)$; and there is the Cox functor $\psi_{\Cox}\colon \cP_{\Theta}\to D_{\Cox}$. 
For all $-d \in \Theta$, each of these functors sends $S(-d)$ to the natural corresponding object.
\end{example}
For each pair of functors $\psi_A, \psi_B$ from \cref{ex:BTfunctors}, there is also a corresponding product category.  For instance, for $\psi_{S}, \psi_{\sX_i}$, consider the  category of graded $S\otimes_k \cO_{\sX_i}$-modules, meaning the category $D(\Spec(S)\times \sX_i)$.  There is some subtlety for pairs like $\psi_{\Cox}, \psi_{\sX_i}$, but one can consider the subcategory of $D(\tsX \times \sX_i)$ generated by $\cE\boxtimes \cF$ where $\cE \in D_{\Cox}$ and $\cF\in D(\sX_i)$.  What we primarily care about are the following Fourier--Mukai transforms. 

\begin{defn}\label{defn:functorialHHLandFMs}
Recall that $\HHL_{S,S}$, defined in \cref{subsec:uniformityHHL}, is a complex of free $S\otimes S$-modules with generators $S(-d)\otimes S(-d')$ for $-d,-d'\in \Theta$.  
For any pair of functors $\psi_A, \psi_B$ from \cref{ex:BTfunctors}, define
$\HHL_{A,B}:= (\psi_A\times \psi_B)(\HHL_{S,S})$, and let 
$\Phi_{A,B}\colon D(A)\to D(B)$ be the corresponding Fourier--Mukai transform.
\end{defn}

For example, the above definition of $\HHL_{i,j}$ is consistent with \cref{defn:HHLnotation}, and the above definition of 
$\Phi_{i,j}$ recovers the Fourier--Mukai transforms  studied throughout this paper.  The functor $\Phi_{\Cox,\Cox}$ was the focus of \cref{sec:diagonal}.  This definition also yields new transforms, such as  $\Phi_{i,S}\colon D(\sX_i)\to D(S)$ and $\Phi_{\Cox,S}\colon D_{\Cox}\to D(S)$.  These functors will be useful for studying window categories and monads in Sections~\ref{sec:windows} and~\ref{subsec:BTmonads}.  

Since all of these functors are transforms with respect to some variant of $\HHL_{S,S}$, they are closely related to one another.  There are two statements that we need in the sequel.

\begin{lemma}\label{lem:PsiInverse}
There is an equivalence $\Psi\circ \Phi_{\Cox,S} \cong \Phi_{\Cox,\Cox}$; in particular, $\Phi_{\Cox,S} \cong \Psi^{-1}$.
\end{lemma}
\begin{proof}
    The key point is that $\Psi \circ \psi_S = \psi_{\Cox}$, and thus
    \[
    (\Id\times \Psi) \circ \HHL_{\Cox,S} 
    = \left((\Id \times \Psi) \circ (\psi_{\Cox} \times \psi_{S})\right)(\HHL_{S,S}) 
    = (\psi_{\Cox}\times \psi_{\Cox})(\HHL_{S,S}).
    \]
The equivalence of transforms is then a standard consequence; see, e.g., \cite{huybrechts06}*{Chapter 5}.  Since $\Phi_{\Cox,\Cox}$ is equivalent to the identity by \cref{cor:strongResDiag}, it follows that $\Phi_{\Cox,S}\cong \Psi^{-1}$.
\end{proof}
\begin{lemma}\label{lem:PhiIS}
There is an equivalence $\Psi\circ \Phi_{i,S} \cong \Phi_{i,\Cox}$. 
\end{lemma}
\begin{proof}
The proof is nearly identical to that of \cref{lem:PsiInverse} and is omitted.
\end{proof}

\subsection{Windows}\label{sec:windows}

Window categories provide a way to use explicit subcategories of $S$-modules to model the derived category of a toric variety.  
\cref{cor:strongResDiag} provides some straightforward yet interesting applications to window categories.  
We begin with a general definition of a window category. 
Recall from \cref{subsec:basics} that $G$ is the group acting on $\Spec(S)$, as induced by its $\Cl(X)$-grading.

\begin{defn}
A full subcategory $W \subset D(S)$ is a \defi{window} for $U \subset
\operatorname{Spec} S$ if the restriction
$r_U^\ast \colon W \to D([U / G])$
is an equivalence. 
\end{defn}

\begin{theorem} \label{thm:HHL_window}
The image of $\Phi_{i,S}$ is a window for $\mathcal X_i$, so
$\Phi_{i,S}\colon D(\mathcal X_i) \to D(S)$ is fully faithful and $r_i^\ast
\circ \Phi_{i,S} \cong \Id$.
\end{theorem}
\begin{proof}
From \cref{lem:PhiIS},  $\Phi_{\Cox,S} \circ \pi_i^\ast \cong
\Phi_{i,S}$. Thus $\Phi_{i,S}$ is the composition of two fully faithful functors.
Since the Hanlon--Hicks--Lazarev complex restricts to a resolution of the diagonal on $\sX_i$, the latter statement holds.
\end{proof}

\begin{theorem} \label{thm:windows_glued}
There is an equality 
    $\langle \operatorname{Im} \Phi_{1,S}, \ldots, \operatorname{Im} \Phi_{r,S} \rangle = K_\Theta(S)$.
\end{theorem}

\begin{proof}
Apply $\Psi$ to deduce this from the definition of $D_{\Cox}$.
\end{proof}

\begin{lemma}[Strong $\Theta$-Transform Lemma]\label{lem:strongThetaTransform}
Fix $-d\in \Theta$. If the image of $d$ lies in the chamber of $\Sigma_{GKZ}$ corresponding to $\sX_i$, then
$\Phi_{i,S}(\cO_{\sX_i}(-d)) = S(-d)$.
\end{lemma}

\begin{proof}
Using \cref{lem:PhiIS}, 
$\Psi \circ \Phi_{i,S} (\mathcal O_{\sX_i}(-d)) = 
\pi_i^\ast \mathcal O_{\sX_i}(-d) = \mathcal O_{\Cox}(-d)$. 
Since $\Psi$ is an equivalence and takes $S(-d)$ to $\mathcal O_{\Cox}(-d)$, 
$\Phi_{i,S}(\cO_{\sX_i}(-d)) = S(-d)$.
\end{proof}

\begin{remark}
In
\cites{halpern2015derived, BFK_GIT}, it is shown that windows exist for the derived categories of GIT quotients, e.g., for all $\sX_i$.  It would be interesting to know if the image of $\Phi_{i,S}$ agrees with the grade restriction associated to any stratification of the unstable locus of $\sX_i$ with weight intervals 
$(-\omega_\lambda^-,0]$, where $-\omega_\lambda^- \ce \sum_{\deg_\lambda (x_i)
<0}\deg_{\lambda}(x_i)$.  Similarly, the image of
$\Phi_{\Cox,S}$ appears to agree with the grade restriction window with weight intervals taken over all one-parameter subgroups $\lambda$, where the $\lambda$ are partially ordered by containment of the corresponding contracting loci; the appropriate interval is $(-\omega_\lambda^-, -\omega_\lambda^+)$ for $\lambda$ where $\omega_\lambda^+ \ce \sum_{\deg_\lambda (x_i)
>0}\deg_{\lambda}(x_i)\ne 0$, or $(-\omega_\lambda^-, 0]$ when $ \omega_\lambda^+=0$.
\end{remark}

\subsection{Noncommutative desingularization}\label{subsec:noncomm}
We now discuss some implications of our results from the perspective of noncommutative algebraic geometry. Recall that $X \ce X_{\Sigma}$ is a semiprojective toric variety; for each Weil divisor class $-d\in \Cl(X)$, write $\mathcal O_{X}(-d)$ for the corresponding reflexive sheaf.  Choose a simplicial refinement $\Sigma'$ of the fan $\Sigma$ with the same rays, which always exists by \cite{CLSToricVarieties}*{Proposition 11.1.7}, and let $X'$ be the corresponding toric variety. Further, let $\sX$ be the toric DM stack corresponding to the stacky fan $(\Sigma', \beta)$, where $\beta \colon \ZZ^{\Sigma'(1)} \to N$ satisfies $\beta(e_\rho) = u_\rho$ for all $\rho \in \Sigma'(1)$.
Since $\Theta$ depends only on the rays of $\Sigma$ (see \cref{subsec:BTdefinitions}), there is a natural bijection between the Bondal--Thomsen collection for $X$ and the Bondal--Thomsen collection of $\sX$.  
By a minor abuse of notation, we will refer to both as $\Theta$.  We define $\mathcal T_{X}\ce\bigoplus_{-d\in \Theta} \cO_{X}(-d)$ and similarly for $\mathcal T_{\sX}$.  We set $T_{\Cox}\ce\bigoplus_{-d\in \Theta} \cO_{\Cox}(-d)$

\begin{lemma}\label{lem:Atheta}
    For any torus invariant divisors $D, D'$, there is an isomorphism
    \[
    \Hom^0(\cO_{X}(D), \cO_{X}(D'))= \Hom^0(\cO_{\sX}(\beta^*D), \cO_{\sX}(\beta^*D')). 
    \]
    Thus there is an isomorphism of underived endomorphism algebras
    $
    \End(\mathcal T_X)=\End(\mathcal T_{\sX}).
    $ 
\end{lemma}
\begin{proof}
Since $\cO_{\sX}(\beta^*D)$ and $\cO_{\sX}(\beta^*D')$ are line bundles on $\mathcal X$, the right hand side is the 0-th cohomology $H^0(\cX, \mathcal O(\beta^*(D'-D)))$. On the other hand, 
\begin{align*}
        \Hom^0(\cO_X(D), \cO_X(D'))
         & = \Hom^0(\cO_X, \cO_X(D)^\vee \otimes \cO_X(D')) & \text{by adjunction,}\\
        & = \Hom^0(\cO_X , \cO_X(D'-D)^{\vee\vee}) & \text{by \cite{CLSToricVarieties}*{Proposition 8.0.6},}\\
        & = \Hom^0(\cO_X, \cO_X(D'-D)) & \text{since $\mathcal O(D'-D)$ is reflexive,}\\
        & = H^0(X, \cO_X (D'-D)).     
    \end{align*}
By \cite{CLSToricVarieties}*{Proposition 4.3.3},
$H^0(X, \cO_X(D'-D))=H^0(X', \cO_{X'}(D'-D))=  \Bbbk \left\langle P_{D'-D}\cap M \right\rangle$.
The statement then follows from \cref{prop:semicohomology}.
\end{proof}

The following will imply Theorem~\ref{thm:noncomm}.  We define $\pi_{\Cox}^*\colon \Perf(X)\to D_{\Cox}(X)$ to be the pullback functor induced by $\tsX \to X$.
\begin{thm}\label{thm:noncommMorePrecise}
The global dimension of $A_\Theta$ is $\dim X$ and 
the functor $\Perf(X) \to D(A_\Theta)$ given by $\cE \mapsto \Hom_{\Cox}(\cT_{\Cox},\pi_{\Cox}^*\cE)$ is fully faithful.  
The algebra $A_{\Theta}$ is uniform for any $X$ with the same Cox ring, that is, for any semiprojective toric variety whose fan has the same rays as $X$. 
\end{thm}
\begin{proof}
The fact that the global dimension of $A_{\Theta}$ is $\dim X$ follows from
\cite{FH}*{Theorem 2.14}.
Since $\cT_{\Cox}$ is a tilting bundle for $D_{\Cox}(X)$ by \Cref{thm:main1}, $D(A_{\Theta})$ is equivalent to $D_{\Cox}(X)$; see, e.g., \cite{keller1994deriving}. 
The pullback $\pi^*_{\Cox} \colon \Perf(X)\to D(\tsX)$ is fully faithful and factors through $D_{\Cox}(X)$.
Finally, the uniformity statement about $A_{\Theta}$ follows from \cref{lem:Atheta}.
\end{proof}

\begin{remark} \label{rem: Sapronov}
Using Grothendieck-Verdier duality for the proper morphism $\pi\colon \tsX \to X$, the fully faithful functor in \Cref{thm:noncommMorePrecise} can alternatively be described as tilting by the complex
$
\pi_*(\mathcal T_{\Cox} \otimes \omega_{\tsX})\otimes \omega_X^\vee$.
\end{remark}
As noted in the introduction, ~\cref{thm:noncommMorePrecise} shows that the algebra $A_\Theta$ is a noncommutative resolution of $X$ in an appropriate sense: the smooth algebra $A_\Theta$ is the {\em underived} endomorphism algebra of the corresponding reflexive sheaf $\mathcal T_X$ on $X$. When $X$ is affine--which is the domain where noncommutative resolutions are generally defined--our result immediately recovers \cite{spenko-van-den-bergh-inventiones}*{Proposition 1.3.6} and the main result of \cite{faber-muller-smith}.
\begin{cor}\label{cor:affineNCR}
    If $X=\Spec R$ is affine, then $A_\Theta$ is a noncommutative resolution of $R$.
\end{cor}
\begin{proof}
    \cref{lem:Atheta} shows that $A_\Theta$ is the underived endomorphism algebra of the reflexive $R$-module $\cT$.   By Theorem~\ref{thm:noncommMorePrecise}, the global dimension of $A_\Theta$ is $\dim X<\infty$. Thus $A_\Theta$ is a noncommutative resolution of $R$ in the sense of Van den Bergh~\cite{van-den-bergh-nccr}.
\end{proof}
Our proof of \cref{cor:affineNCR} has a different flavor than previous proofs. The~\cite{faber-muller-smith} proof is explicit and direct, based on properties of the conic modules corresponding to (what we refer to as) elements of the Bondal--Thomsen collection $\Theta$.  The proof in \cite{spenko-van-den-bergh-inventiones} positions the result as a corollary of a more general result about reductive group actions.  Our work derives the result from $D_{\Cox}$ using the birational geometry perspective offered by $\Sigma_{GKZ}$.

\begin{remark}
    When $X =\Spec(R)$ is affine, one can directly check that for any $\theta \in M_{\RR}$, $\cO_{X}(-d(\theta))$ corresponds to the conic module $A_{\theta}$, in the notation of Faber--Muller--Smith, and hence our $\mathcal T_X$ lines up directly with the module $\AA$ from their work.
\end{remark}

Theorem~\ref{thm:main1} can be articulated in terms of categorical resolutions as defined 
in~\cite{kuzentsov-lefschetz}. 

\begin{theorem}\label{thm:categoricalResolutions}
If $X$ is a semiprojective normal toric variety, then $D_{\Cox}(X)$ is a categorical resolution of $X$.
\end{theorem}
\begin{proof}
As above, we let $\sX$ be the toric DM stack corresponding to a simplicial refinement of $\Sigma$ with the same rays. We write $\pi_0\colon \mathcal X\to X$ for the induced map.  We take $\pi\colon \tsX \to \sX$ as usual.  The pairs $(\pi_0 \circ \pi)^*$ and $(\pi_0 \circ \pi)_*$ form a pseudoadjunction, and $(\pi_0 \circ \pi)^*$  is fully faithful by Lemma~\ref{lem:fullyFaithful}.
It follows that $(\pi_0 \circ \pi)_* \circ (\pi_0 \circ \pi)^* \cong \Id_{\Perf(X)}$. 
Since $D_{\Cox}(X)$ is homologically smooth by Proposition~\ref{prop:dgThm}, this yields the desired result.
\end{proof}

\begin{remark}\label{rmk:crepant}
Faber--Muller--Smith \cite{faber-muller-smith}*{Propositions 7.5 and 7.9} showed that when $X=\Spec(R)$ is affine, the resulting $D(A_\Theta)$ is a noncommutative {\em crepant} resolution if and only if $X$ is simplicial.  When $X$ is not affine, $D_{\Cox}$ is rarely a categorical crepant resolution in the sense of~\cite{kuzentsov-lefschetz}.  For instance, this fails even in for the Hirzebruch surface $\cH_3$. However, with notation as in the proof of \cref{thm:categoricalResolutions}, $D_{\Cox}$ is a categorical crepant resolution of $D(X)$ when $\sX \to X$ is crepant and $\sX$ has Bondal--Ruan type in the sense of~\cite{FH2}. 
\end{remark}

\subsection{Cox rings and Bondal--Thomsen free monads}\label{subsec:BTmonads}
In this section, we show that $D_{\Cox}$ can be used to provide uniform free monads across all toric birational models of a given toric variety.
Recall from \cref{subsec:BTGKZ} that, given an
$S$-module $P$ and a cone $\Gamma\in \Sigma_{GKZ}(X)$, there is a corresponding
sheaf $P|_{X_{\Gamma}}$ on the variety $X_\Gamma$. 
This assignment is functorial and hence extends to complexes and derived categories.

For a sheaf $\cE$ on a toric variety or toric stack $\sY$, a \defi{free monad} is a graded complex $F$ of free modules over the Cox ring such that $F|_{\sY}\cong \cE$; by a standard, minor abuse of notation, we  will refer to $F|_{\sY}$ as a free monad as well.  The Beilinson collection produces free monads for $\PP^n$ that have been used in many contexts; see, e.g.,~\cite{EFS}*{\S8} and related citations, or the complexes arising from one of the Beilinson spectral sequences, e.g., \cite{caldararu}*{\S3.7}.  Similar monads for weighted projective stacks also appear in~\cite{BE-nonstandard}*{Corollary 1.8}.

Recall from \cref{conv:functors} that $K_\Theta(S)$ is the homotopy category of free $S$-modules whose generators have degrees lying in $\Theta$.  
From \cref{cor:tilting_equiv}, the functor $\Psi\colon K_{\Theta}(S)\to D_{\Cox}$ is an equivalence, with inverse $\bM\colon D_{\Cox}\to K_{\Theta}(S)$ (see \cref{lem:PsiInverse}).
We call the free complex of $S$-modules $\bM(\cE)\in K_\Theta(S)$ a \defi{Bondal--Thomsen monad} for $\cE$.

On $\PP^n$, these monads provide an algebraic presentation that allows for a direct computation of sheaf cohomology.  
For instance, for $\cO_{\PP^1}(-3)$, the minimal free resolution is $S(-3)$, whereas the above monad is $\bM(\cO_{\PP^1}(-3)) = [S^2 \gets S(-1)^3]$.  
From the latter (but not the former), the cohomology can be computed by taking the degree $0$ strand:
\vspace*{-1mm}
\[
[S^2\gets S(-1)^3]_0 \cong [k^2 \gets 0] \cong H^*(\PP^1, \cO_{\PP^1}(-3)).
\]
Generalizing to the Bondal--Thomsen collection on 
$X$, it turns out that this structure is available across all cones in $\Sigma_{GKZ}(X)$.  
For each $\Gamma\in \Sigma_{GKZ}$, write $(\pi_{\Gamma})_*$ for the map $D_{\Cox}\to D(X_\Gamma)$ induced by $\pi_\Gamma\colon\tsX\to X_\Gamma$.  From the monad, we show that one can immediately compute the Fourier--Mukai transforms and higher direct images to any of the $X_{\Gamma}$.

\begin{thm}\label{thm:pushforwardGKZ}
Let $\cE\in D_{\Cox}$ with associated  Bondal--Thomsen monad $\bM(\cE)$.   
For each $\Gamma\in \Sigma_{GKZ}$, 
$\bM(\cE)|_{X_\Gamma}$ is a free monad for $(\pi_{\Gamma})_*(\cE)$ on $X_\Gamma$. 
\end{thm}

In particular, a Bondal--Thomsen monad encodes sheaf cohomology: 
\[
H^*(X,\cE) 
= \bM(\cE)|_{\Spec S_0} 
= [\bM(\cE)]_0.
\]
Thus Theorem~\ref{thm:pushforwardGKZ} demonstrates how a complex of $S$-modules can simultaneously encode geometric information for all of the toric varieties in $\Sigma_{GKZ}$.

\begin{example}\label{ex:TwistedCubicMonads}
Let $X$ be the blowup of $\PP^3$ at two points as in Example~\ref{ex:Bl2points}.  Let $C\subseteq X$ be the curve defined by
\[
C = V( x_1x_3-x_2^2x_4, x_0x_2-x_1^2x_5, x_0x_3-x_1x_2x_4x_5).
\]
This is the strict transform of the twisted cubic; namely, $V(x_4)$ and $V(x_5)$ are the exceptional divisors, and setting $x_4=x_5=1$ recovers the traditional equations of the twisted cubic. 
The Bondal--Thomsen monad for $\cO_C$ is the free complex
\[
S \gets
\begin{matrix}S(-2,-1,0)\\ \oplus \\ S(-2,0,-1)\\ \oplus \\ S(-2,-1,-1)  \end{matrix}
\gets S(-3,-1,-1)^2 \gets 0.
\]
This is in fact a free resolution of $S/I$, where $I$ is the ideal defined by the equations above.  The GKZ fan includes five maximal chambers and a number of boundary faces (see \Cref{fig:Bl2Pts}). 
\Cref{thm:pushforwardGKZ} implies that the free monad for $\cO_C$ under the associated pushforwards and Fourier--Mukai transforms can be computed immediately; geometrically, these monads compute the image of $C$ under the various morphisms and rational maps from $\Sigma_{GKZ}$.

The recipe comes from \Cref{lem:ThetaRestrict}.  For instance, to compute the monad for the twisted cubic (the image of the curve in $\PP^3$), \Cref{lem:ThetaRestrict} implies that  $S(a,b,c)|_{\PP^3}\cong \cO_{\PP^3}(a)$ and thus \Cref{thm:pushforwardGKZ} implies that the pushforward of $\cO_C$ from $X$ to $\PP^3$ yields
\[
\cO_{\PP^3}\gets \cO_{\PP^3}(-2)^3 \gets \cO_{\PP^3}(-3)^2\gets 0, 
\]
which is the classical resolution of the twisted cubic.  

Next consider the Hirzebruch surface $\cH_1$ on the bottom row of Figure~\ref{fig:Bl2Pts}. \Cref{lem:ThetaRestrict} implies that only the factors $S$ and $S(-2,-1,0)$ will be nonzero when restricted to this $\cH_1$.  Thus, the monad for the Fourier--Mukai of $\cO_C$ from $X\dashrightarrow X'$ followed by the pushforward to $X'\to \cH_1$ is 
\[
\cO_{\cH_1}\xleftarrow{\ x_1x_3-x_2^2x_4 \ }\cO_{\cH_1}(-2,-1)\gets 0.
\]
Equivalently, this is a locally-free resolution of the structure sheaf of the defining image of $C$ under the rational map from $X$ to $\cH_1$.

In a similar manner, consider the $\PP^2$ from the bottom row of Figure~\ref{fig:Bl2Pts}.  There is a map $f\colon X\to \PP^2$ and the resolution for $f_*\cO_C$ is
\[
\cO_{\PP^2} \xleftarrow{\ x_1x_3-x_2^2 \ }\cO_{\PP^2}(-2)\gets 0,
\]
where $\PP^2$ has coordinates $[x_1\mathbin:x_2\mathbin:x_3]$.  In particular, the image of $C$ is defined by the conic $x_1x_3-x_2^2$.  A resolution can similarly be computed for the pushforward to either copy of $\operatorname{Bl}_1(\PP^3)$ or either copy of $\PP^2$ and so on.
\end{example}

\begin{example}\label{ex:HirzMonads}
Let $Z$ be a set of $5$ generic points in the Hirzebruch surface $\cH_3$.  The Bondal--Thomsen monad $\bM(\cO_Z)$ is the free complex
\[
S^5 \gets \begin{matrix} S(-1,0)^5 \\ \oplus \\ S(0,-1)^5\end{matrix} \gets S(-1,-1)^5 \gets 0.
\]
Theorem~\ref{thm:pushforwardGKZ} implies that the free monad for the various pushforwards and Fourier--Mukai transforms of $\cO_Z$ from this monad can be computed immediately.  
Namely, by \Cref{lem:ThetaRestrict} shows how elements of $\Theta$ behave under restriction to the various $X_{\Gamma}$.  In the following chart, the entry in row $S(-d)$ and column $X_{\Gamma}$ is the sheaf $S(-d)|_{X_{\Gamma}}$.
\begin{center}
    \begin{tabular}{c|c|c|c|c|}
    &$\cH_3$&$\PP(1,1,3)$&$\PP^1$&$\Spec(\Bbbk)$\\ \hline \hline
    $S$&$\cO_{\cH_3}$&$\cO_{\PP(1,1,3)}$&$\cO_{\PP^1}$&$\Bbbk$\\ \hline
   $S(-1,0)$&$\cO_{\cH_3}(-1,0)$ &$\cO_{\PP(1,1,3)}(-1)$ & $\cO_{\PP^1}(-1)$&$0$\\ \hline
   $S(2,-1)$&$\cO_{\cH_3}(2,-1)$ &$\cO_{\PP(1,1,3)}(-1)$ & $0$&$0$\\ \hline
      $S(1,-1)$&$\cO_{\cH_3}(1,-1)$ &$\cO_{\PP(1,1,3)}(-2)$ & $0$&$0$\\ \hline
            $S(0,-1)$&$\cO_{\cH_3}(0,-1)$ &$\cO_{\PP(1,1,3)}(-3)$ & $0$&$0$\\ \hline
            $S(-1,-1)$&$\cO_{\cH_3}(-1,-1)$ &$\cO_{\PP(1,1,3)}(-4)$ & $0$&$0$\\ \hline
    \end{tabular}
\end{center}
For instance, the pushforward of $\cO_Z$ to $\PP^1$ will be
$\cO_{\PP^1}^5 \gets \cO_{\PP^1}(-1)^5$, which is
a monad for $5$ points in $\PP^1$.  
The Fourier--Mukai transform to $\PP(1,1,3)$ will be
\[
\cO_{\PP(1,1,3)}^5 \gets \begin{matrix}\cO_{\PP(1,1,3)}(-1)^5 \\ \oplus \\ \cO_{\PP(1,1,3)}(-3)^5 \end{matrix} \gets \cO_{\PP(1,1,3)}(-4)^5 \gets 0,
\]
which is a monad for $5$ points in $\PP(1,1,3)$ (none of which is the stacky point). 
And, the pushforward to $\Spec(\Bbbk)$ will be
$\Bbbk^5\gets 0$, 
which computes the cohomology of $\cO_Z$.
\end{example}

\begin{proof}[Proof of Theorem~\ref{thm:pushforwardGKZ}]
The theorem amounts to showing that the following diagram of functors commutes: 
\[
\xymatrix{
D_{\Cox}\ar[rr]^-{\bM(-)}\ar[rrd]_{(\pi_{\Gamma})_*}&&K_{\Theta}(S)\ar[d]^-{(-)|_{X_{\Gamma}}}\\
&&D(X_{\Gamma}).
}
\]
It suffices to prove this for $\cO_{\Cox}(-d)$ with $-d\in \Theta$.  Assume first that $\Gamma$ is a chamber, so that $X_{\Gamma}=X_i$ is the coarse moduli space of $\sX_i$.  The $\Theta$-Transform Lemma implies that $(\pi_i)_*\cO_{\Cox}(-d)=\cO_{\sX_i}(-d)$.  Further pushing forward onto $X_i$ then yields $\cO_{X_i}(-d)$ by \cref{prop:pushforwardBTcoarse}.  This gives commutativity on generators.  For morphisms, \cref{cor:partialorder} shows that morphisms $\Hom(\cO_{\Cox}(-d),\cO_{\Cox}(-d')))$ are in bijection with $S_{d-d'}$.  The functor $\Phi_{\Cox,S}$ sends each such morphism to its obvious counterpart based on the fact that $\Phi_{\Cox,S}=\Psi^{-1}$ as proven in \cref{lem:PsiInverse}.  The proof of \cref{lem:ThetaRestrict} then shows that each monomial in $S_{d-d'}$ induces a natural morphism on the corresponding sheaves on $X_{\Gamma}$.    Similarly, the proof of \cref{prop:pushforwardBTvariety} shows that each monomial in $S_{d-d'}$ induces the same natural morphism on the corresponding sheaves on $X_{\Gamma}$

Now, if $\Gamma$ is a lower dimensional cone of $\Sigma_{GKZ}$, then $\pi_{\Gamma}$ factors through a morphism $X_i \to X_{\Gamma}$ for some $i$, then  \cref{prop:pushforwardBTvariety} implies that $(\pi_{\Gamma})_*\cO_{\Cox}(-d) = \cO_{X_\Gamma}(-d)$.
Going the other way around the diagram, \cref{lem:ThetaRestrict} shows that $S(-d)|_{X_{\Gamma}} = \cO_{X_{\Gamma}}(-d)$.  This proves the commutativity on objects.  For morphisms, apply the same argument as in the previous paragraph.
\end{proof}

Working with Bondal--Thomsen monads also yields vanishing results that hold uniformly for all varieties indexed by the secondary fan, echoing ~\cite{BETate}*{Remark 3.16}.
For instance:
\begin{cor}\label{cor:vanishing}
Let $\cE \in D(X)$.  
If $\bM(\cE)$ only has terms in negative cohomological degrees,\footnote{This is equivalent to asking that $\bM(\cE)$ be a virtual resolution of $\cE$ in the sense of~\cite{BES}.} so that $\bM(\cE)$ has the form
\[
0\gets \bM(\cE)^0 \gets \bM(\cE)^{-1} \gets \cdots,
\]
then for any $\pi_{\Gamma}\colon X\to X_\Gamma$ from $\Sigma_{GKZ}(X)$,  $R^i\pi_{\Gamma,*}(\cE)=0$ for $i>0$.
\end{cor}
\begin{proof}   
Writing $\mathrm H^i$ for the $i$-th cohomology of a complex of sheaves, Theorem~\ref{thm:pushforwardGKZ} implies that $R^i(\pi_{\Gamma})_*(\cE)=\mathrm H^i(\bM(\cE)|_{X_\Gamma})$.  But $\bM(\cE)|_{X_\Gamma}$ has the form
\[
\left( \bM(\cE)^0\right)|_{X_{\Gamma}}
\gets
\left( \bM(\cE)^{-1}\right)|_{X_{\Gamma}}
\gets\cdots
\]
and thus equals zero in cohomological degrees $>0$. 
 It follows that $\mathrm H^i=0$ for $i>0$.
\end{proof}

\begin{example}
If $H^i(X,\cE(-d))=0$ for all $i>0$ and $-d\in \Theta$, then by a direct computation involving Demazure Vanishing and the equality $\Phi_{\Cox,S}=\bM$, $\bM(\cE)$ will be a virtual resolution of $\cE$ and \cref{cor:vanishing} will apply.
This is a consequence of the fact that $\bM (\cE)$ can be explicitly computed as a Fourier--Mukai transform with respect to the Hanlon--Hicks--Lazarev resolution of the diagonal and that resolution only involves terms of degree $-d \in\Theta$.

For instance, if $L$ is a nef line bundle on $X$, then $H^j(X, L(-d))=0$ for all $j>0$ and $-d\in \Theta$ by Demazure Vanishing for $\QQ$-divisors as noted before \cref{lem:BTvanishingCohomology}. 
Therefore, Corollary~\ref{cor:vanishing} applies to $L$. 
\end{example}

\begin{remark}
Corollary~\ref{cor:vanishing} has connections with notions of multigraded regularity, such as the one introduced by Maclagan and Smith in~\cite{MS}.  For instance, let $X=\cH_3$ be the Hirzebruch surface from our running example with canonical degree $w_X = (1,-2)$. 
Let $M$ be a graded $S$-module that is $w_X$-regular with respect to the Maclagan--Smith definition of multigraded regularity as in the introduction of~\cite{MS}.
A direct computation confirms that $M|_{X}$ satisfies the hypotheses of \cref{cor:vanishing}.  In other words, by combining multigraded regularity with monads as in \cref{cor:vanishing}, one can obtain new vanishing results for higher direct images under the morphisms and Fourier--Mukai transforms from $\Sigma_{GKZ}$.
\end{remark}

\subsection{Sharpened Bondal--Thomsen generation}\label{sec:sharpenedBondal}
In addition to their influence on mirror symmetry, the ideas from Bondal's Oberwolfach report~\cite{Bondal} were also influential on studies of exceptional collections and generation. The statement that the Bondal-Thomsen collection generates the derived category has been referred to as Bondal's Conjecture in the literature, as it was inspired
by~\cite{Bondal}, however Bondal never made such a conjecture.  The question
nevertheless attracted
interest~\cites{BallardEtAl2019ToricFrobenius,prabhu-naik,ohkawa2013frobenius,uehara} and was proven
in~\cites{FH,HHL}.  Our methods allow us to strengthen those results, proving
that one only needs a subset of the bundles from $\Theta$ to generate $D(X)$.

Throughout this section, assume that $X$ is a simplicial toric variety with corresponding toric DM stack $\sX$ given by $(\Sigma, \beta)$ where $\beta(e_\rho) = u_\rho$ for all $\rho \in \Sigma(1)$.
Further, let $\Gamma$ be a wall of the nef cone of $X$ that is an interior wall of $\Sigma_{GKZ}(X)$.\footnote{If $\Sigma_{GKZ}$ has no interior walls, then we are in the situation of Example~\ref{ex:singleMaximal} where $\Theta$ is a full strong exceptional collection for $D(\sX)$.}
There is a corresponding primitive collection $P=P_{\Gamma}\subseteq \Sigma(1)$, see~\cite{CLSToricVarieties}*{Definition 5.1.15}, and we summarize a few key properties.
Each primitive collection $P$ also determines a circuit relation $\sum_\rho b_\rho u_\rho = 0$, where $b_\rho \in \QQ$~\cite{CLSToricVarieties}*{(6.4.8)}.  
This relation can be used to define a degree function $\deg_\Gamma\colon S\to \ZZ$, where $\deg_\Gamma(x_\rho) \ce b_\rho$.  Again by ~\cite{CLSToricVarieties}*{(6.4.8)},  $\deg_\Gamma(x_\rho)>0$ if and only if $\rho \in P$, and thus the rays can be partitioned into two sets: $P =\{ \rho\mid\deg_\Gamma(x_\rho)>0\}$ and $P^{\mathsf{c}} =\{ \rho\mid\deg_\Gamma(x_\rho)\leq 0\}$.

We define the closed zonotope $\overline{Z_+}=\left\{ \sum_{\rho \in P} a_\rho D_\rho\ \Big\vert\ a_\rho \in [-1,0] \right\} $ and the open zonotope $Z_{-}^\circ \ce \left\{ \sum_{\rho \in P^{\mathsf{c}}} a_\rho D_\rho\ \Big\vert\ a_\rho \in (-1,0) \right\}$.  Note that $Z_{-}^\circ$ is full dimensional because the interior wall of $\Sigma_{GKZ}(X)$ was chosen so there must be a full-dimensional chamber of $\Sigma_{GKZ}(X)$ on the non-positive side of $\Gamma$.

\begin{example}\label{ex:HirzPrimitive}
In our running example of the Hirzebruch surface, let $\Gamma$ be the wall on the $y$-axis in \Cref{fig:HirzSecondaryFan}.  The corresponding primitive collection  is $P=\{\rho_0,\rho_2\}$, with circuit relation $u_0 - 3u_1 +u_2=0$ and degree function $\deg_\Gamma$ where $\deg_\Gamma(x_0)=\deg_\Gamma(x_2) =1, \deg_\Gamma(x_1)=-3$ and $\deg_\Gamma(x_3)=0$.
In this case, $\overline{Z_+}$ is the closed interval from $(0,0)$ to $(-2,0)$, while $Z_-^\circ$ is the open parallelogram spanned by $(3,-1)$ and $(0,-1)$.  See Figure~\ref{fig:zonotopesan}.
\end{example}

\begin{lemma}\label{lem:zonotopes}
There is a containment $\overline{Z_+} + Z_-^\circ \subseteq Z$.
\end{lemma}
\begin{proof}
After taking closures, both sides are equal.  Since $Z_-^\circ$ is full-dimensional and open, for every point $z$ in $Z_-^\circ$, there is an open disc surrounding $z$ and contained in $\overline{Z_-}$.  It follows that, around any point $\widehat{z}$ in the Minkowski sum $\overline{Z_+} + Z_-^\circ$, there is an open disc contained in $\overline{Z_+}+\overline{Z_-}=\overline{Z}$.  In particular, it follows that $\widehat{z}$ cannot lie on the boundary of $\overline{Z}$.
\end{proof}

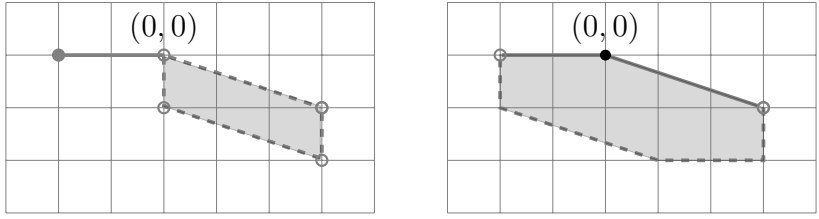
\begin{figure} \centering
\begin{tikzpicture}[scale = .6]
\draw[step=1cm,gray,very thin] (4,1) grid (-3,-3);
\draw[gray, dashed, line width = .5mm] (.09,-.03)--(2.91,-.97);
\draw[gray, dashed, line width = .5mm] (0,-.03)--(0,-.97)--(3,-1.97)--(3,-1);
\draw[fill = gray, opacity = .3] (0,-.03) -- (2.97,-.97) --(2.97,-1.97) -- (0,-.97);
\draw[thick, gray] (0,0) circle (3pt);
\draw[thick, gray] (3,-2) circle (3pt);
\draw[thick, gray] (0,-1) circle (3pt);
\draw[thick, gray, fill = gray] (-2,0) circle (3pt);
\draw[thick, gray] (3,-1) circle (3pt);
\draw[thick,gray] (3,-1) circle (3pt);
\draw[gray, line width = .5mm] (0,0)--(-2,0);
\filldraw[black, very thick]  (-1,1) circle (0pt);
\filldraw[black] (0,0) circle (0pt) node[anchor=south]{$(0,0)$};
\end{tikzpicture}
\qquad
\begin{tikzpicture}[scale = .6]
\draw[step=1cm,gray,very thin] (4,1) grid (-3,-3);
\draw[gray, line width = .5mm] (2.91,-.97)--(0,0)--(-1.97,0);
\draw[gray, dashed, line width = .5mm] (3,-1) -- (3,-2)--(1,-2)--(-2,-1)--(-2,0);
\draw[fill = gray, opacity = .3] (0,0) --(3,-1) -- (3,-2)--(1,-2)--(-2,-1)--(-2,-0)--(0,0);
\draw[thick, gray] (3,-1) circle (3pt);
\draw[thick, gray] (-2,0) circle (3pt);
\draw[thick,gray] (3,-1) circle (3pt);
\filldraw[black, very thick]  (-1,1) circle (0pt);
\filldraw[black, very thick]  (0,0) circle (2pt);
\draw[black, very thick]  (0,0) circle (2pt);
\filldraw[black] (0,0) circle (0pt) node[anchor=south]{$(0,0)$};
\end{tikzpicture}
\caption{In Example~\ref{ex:HirzZonotope}, the region $Z_-^\circ $ is the parallelogram on the left, while $\overline{Z_+}$ is the closed interval $[-2,0]$ on the $x$-axis. The sum $\overline{Z_+} + Z_-^\circ$ is contained in the standard half-open zonotope $Z$ on the right.
}
\label{fig:zonotopesan}
\end{figure}

\begin{defn}\label{defn:ThetaGammaCirc}
Set $\Theta_\Gamma^\circ\ce Z_-^\circ \cap \Cl(X)$.
\end{defn}

We show that these elements are not needed when generating $D(\sX)$. Specifically:
\begin{thm}\label{thm:thetaWgenerates}
The elements in $\Theta \setminus \Theta^\circ_\Gamma$ generate $D(\sX)$; more specifically, 
\[
D(\sX) = \langle \cO_X(-d) \mid -d\in \Theta \setminus \Theta^\circ_\Gamma \rangle.
\]
As a consequence, elements in $\Theta \setminus \Theta^\circ_\Gamma$ generate $D(X)$.
\end{thm}

Recall that the angled bracket notation used in \cref{thm:thetaWgenerates} was introduced in \cref{sec:CoxCategory}, and it indicates the smallest thick subcategory of $D(\sX)$ containing the given elements.

To prove \cref{thm:thetaWgenerates}, we use an algebraic consequence of \cref{lem:zonotopes}. Write $I_P \ce\langle x_\rho\mid\rho \in P\rangle$.  
Since $P$ is a primitive collection, $I_P$ is a minimal prime of the irrelevant ideal of $X$.  
Thus if $K_P$ is the Koszul complex on $I_P$, then $(K_P)|_X\cong (S/I_P)|_{X}=0$ in $D(X)$.

\begin{cor}\label{cor:zonotopes}
If $-d \in \Theta_\Gamma^\circ$, then the twisted Koszul complex $K_P(-d)$ has generators lying entirely in $\Theta$.
\end{cor}
\begin{proof}
The statement is a corollary of Lemma~\ref{lem:zonotopes}.  The free summands of the Koszul complex $K_P$ are of the form $S(-e_I)$, where $I \subseteq P$ and $e_I = \sum_{\rho \in I} \deg(x_\rho)$.  In particular, each $-e_I$ lies in $\overline{Z_+}$.
Since $-d\in Z_-^\circ$, it follows that $-d-e_I$ lies in $\overline{Z_+}+Z_-^\circ$. By Lemma~\ref{lem:zonotopes}, this lies in $Z$, and hence $-d-e_I \in \Theta$.  Thus, $K_P(-d)$ is a Koszul complex where all summands are of the form $S(-d-e_I)$, with $-d-e_I\in \Theta$ and with $-e_I\in \overline{Z_+}$.
\end{proof}

\begin{example}
In our Hirzebruch surface example, if $d=(-2,1)$, then $K_P(-d)$ is the complex
$
S(2,-1) \gets S(1,-1)^2 \gets S(0,-1) \gets 0.
$
\end{example}

\begin{proof}[Proof of Theorem~\ref{thm:thetaWgenerates}]
Since $\Theta$ generates $D(\sX)$, it suffices to prove that each $\cO_{\sX}(-d)$ with $-d\in \Theta_\Gamma^\circ$ can be generated by $\Theta \setminus  \Theta_\Gamma^\circ$. 
Continuing with notation from the proof of ~\cref{cor:zonotopes}, since $\deg_\Gamma(x_\rho)>0$ for all $\rho\in P$, $\deg_\Gamma(d+e_I)> \deg_\Gamma(d)$ for $\emptyset \ne I\subseteq P$.  
Thus, for each $-d\in \Theta_\Gamma^\circ$, the line bundle $\cO_{\sX}(-d)$ can be generated by the elements in $\Theta$ of strictly larger degree with respect to $\deg_\Gamma$.  
In particular, iterating this argument eventually generates all elements of $\Theta_\Gamma^\circ$ by those in $\Theta \setminus \Theta_\Gamma^\circ$. The final claim of the theorem follows from the fact that the coarse moduli morphism realizes $D(X)$ as a quotient of $D(\sX)$.
\end{proof}

\begin{remark}
The results in this section give an indication of both what can go right and what go wrong in the search for full strong exceptional collections for $D(\sX)$ on its own.  In the case of the Hirzebruch surface, the elements of $\Theta \setminus \Theta_\Gamma^\circ$ all lie in the chamber corresponding to the Hirzebruch surface because there is only one wall. Therefore, the generators from \cref{thm:thetaWgenerates} coincide with a well-known full strong exceptional collection of line bundles for the Hirzebruch surface.

On the other hand, imagine that the nef cone is bounded on both sides by interior walls of $\Sigma_{GKZ}$.  Then we will have to consider a pair of primitive collections $P$ and $P'$.  The corresponding reductions could interfere with one another.  For instance, for $-d \in \Theta_\Gamma^\circ$, the Koszul complex $K_P(-d)$ might spill into $\Theta_{\Gamma'}^\circ$, and we might then be unable to generate with only elements from $\Theta \setminus (\Theta_\Gamma^\circ \cup \Theta_{\Gamma'}^\circ)$.
\end{remark}

\begin{appendices}

\section[An invariant description of the Cox category]
{An invariant description of the Cox category}\label{sec:grothendieckconstr}

This appendix provides a 2-categorical perspective on our geometric 
construction of $D_{\Cox}$, via the Grothendieck construction \cite{SGA1}*{\S\.VI.8}. 
Let $\operatorname{Cat}$ be the $2$-category of (small) categories, functors, and natural transformations. We recall the definition of a (contravariant) lax functor from a category $I$ to $\on{Cat}$. Such a lax functor $\cC \colon I^{op} \to \operatorname{Cat}$  is the data of functions
$\cC \colon \operatorname{Obj} I \to \operatorname{Cat}$ and 
$\cC \colon \operatorname{Hom}_{I^{op}}(i,j) \to \operatorname{Fun}(\cC(i), \cC(j))$ 
along with natural transformations $\alpha_i \colon   \Id_{\cC(i)} \to \cC(1_i)$ and, 
for each pair of composable morphisms $u \colon i \to j, v \colon j \to l$ in $I$, a 
natural transformation 
$\gamma_{u,v} \colon \cC(u) \circ \cC(v) \to \cC(u \circ v)$ 
such that the following diagrams commute for all triples of composable morphisms $u \colon i \to j, v \colon j \to l$, and $w \colon l \to m$ in $I$:

\[
\begin{tikzcd}
\cC(u) = \operatorname{Id}_{\cC(i)} \circ\, \cC(u) 
\ar[r, "\alpha_i \cC(u)"] \ar[dr, "1"'] & 
\cC(1_i) \circ \cC(u) \ar[d, "\gamma_{1_i,u}"] \\
& \cC(u),
\end{tikzcd}
\quad
\quad
\begin{tikzcd}
\cC(u) = \cC(u) \circ \operatorname{Id}_{\cC(j)} 
\ar[r, "\cC(u) \alpha_j"] \ar[dr, "1"'] & 
\cC(u) \circ \cC(1_j) \ar[d, "\gamma_{u,1_j}"] \\
& \cC(u),
\end{tikzcd}
\]
\begin{equation} \label{eqn:GC_cocycle}
\begin{tikzcd}
\cC(u) \circ \cC(v) \circ \cC(w) \arrow[rr, "\gamma_{u,v}(\cC(w))"] \arrow[d, "\cC(u) \gamma_{v,w}"'] && \cC(u \circ v) \circ \cC(w) \arrow[d, "\gamma_{u \circ v, w}"] \\
\cC(u) \circ \cC(v \circ w) \arrow[rr, "\gamma_{u, v \circ w}"] && \cC(u \circ v \circ w).
\end{tikzcd}
\end{equation}

The \defi{1-skeleton} of $\cC$ is the data of the categories and 
functors $(\cC(i),\cC(u))$. A \defi{morphism of lax functors} $\nu\colon \cC \to \cD$ consists of functors 
$\nu_i \colon \cC(i) \to \cD(i)$ and natural transformations $\nu_u \colon \cC(u) 
\to \cC(v)$ making the appropriate diagrams commute.

Given $\cC \colon I^{op} \to \operatorname{Cat}$, there is a
category, commonly called the \defi{Grothendieck construction} 
and denoted by $\int \cC$, associated to $\cC$.
The objects and morphisms are as follows:
\begin{align*}
\operatorname{Obj} \int \cC &\ce 
\{ (i,A) \mid A \in \operatorname{Obj} \cC(i) \},  \\
\operatorname{Hom}_{\int\cC}((i,A),(j,B)) &\ce 
\{ (u,\phi) \mid u \colon i \to j,
\phi \colon A \to \cC(u) B \}.
\end{align*}

Composition of morphisms is given by
$(v, \psi) \circ (u, \phi) \ce (v \circ u, \gamma_{u,v} \circ \cC(u) \psi \circ \phi)$,
which is associative by \eqref{eqn:GC_cocycle}.
The category $\int \cC$ is characterized by a universal property \cite{JohnsonYau_2d_cats}*{Theorem 10.2.3}, which realizes it as a lax colimit of $\cC$. 
For algebraic geometers, the Grothendieck construction is 
perhaps most familiar from its use in passing between 
prestacks and fibered categories; see, e.g., \cite{StacksOlsson}*{Chapter~3}. 
It provides a flexible framework for gluing diagrams of 
categories. In particular, it is straightforward to 
realize a semiorthogonal decomposition, interpreted as 
two categories glued along a bimodule, as the 
pretriangulated envelope of 
a Grothendieck construction, where the indexing category 
is the $A_2$-quiver.

A natural source of lax functors are common 
rational resolutions. Before making this precise, we 
record an observation. 
Let $\pi_1 \colon \sX \to \sY_1$ and $\pi_2 \colon \sX \to \sY_2$ be 
maps of concentrated algebraic stacks. Denote the counits of adjunction by 
$
\epsilon_i \colon (\pi_i)^\ast \pi_{i \ast} \to \Id.
$

\begin{lemma} \label{lem:counits_commute}
The following diagram commutes: 
\[
\begin{tikzcd}
(\pi_1)^\ast \pi_{1 \ast} (\pi_2)^\ast \pi_{2 \ast} 
\arrow[r, "\epsilon_1 (\pi_2)^\ast \pi_{2 \ast}"] 
\arrow[d, "(\pi_1)^\ast \pi_{1 \ast} \epsilon_2"'] & 
(\pi_2)^\ast \pi_{2 \ast} \arrow[d, "\epsilon_2"] \\
(\pi_1)^\ast \pi_{1 \ast} \arrow[r, "\epsilon_1"] & 
\Id. 
\end{tikzcd}
\]
\end{lemma}

\begin{proof}
Recall that $\pi_{i \ast}$ and $\pi_i^{\ast}$ are given as Fourier--Mukai 
transforms with kernels equal to the structure sheaf of the graph $\Gamma_{\pi}$ of $\pi_i$,
$\mathcal O_{\Gamma_{\pi_i}} \in D(\sY \times \sX)$, and its transpose 
$\mathcal O_{\Gamma_{\pi_i}}^t \in D(\sX \times \sY)$, respectively. 
Convolution of these kernels underlies the composition of functors.  
To check the diagrams, we verify that the maps of kernels commute. Here, this is 
\[
\begin{tikzcd}
\mathcal O_{\Gamma_{\pi_1}}^t \ast \mathcal O_{\Gamma_{\pi_1}} \ast 
\mathcal O_{\Gamma_{\pi_2}}^t \ast \mathcal O_{\Gamma_{\pi_2}} 
\arrow[r, "\epsilon_1 \ast 1"] 
\arrow[d, "1 \ast \epsilon_2"'] & 
\mathcal O_{\Gamma_{\pi_2}}^t \ast \mathcal O_{\Gamma_{\pi_2}} 
\arrow[d, "\epsilon_2"] \\
\mathcal O_{\Gamma_{\pi_1}}^t \ast \mathcal O_{\Gamma_{\pi_1}}
\arrow[r, "\epsilon_1"] & 
\mathcal O_{\Delta \sY}.
\end{tikzcd}
\]
The upper left convolution is a pushforward from $\sY \times \sX_1 \times 
\sY \times \sX_2 \times \sY$, and one can check that the rest of the diagram 
is also pushed forward from there. 

The counits $\epsilon_i$ are the composition of the natural map from the derived 
tensors to the underived tensors and the natural maps of tensor products of
sheaves. 

Since the maps from the derived to underived tensor products commute, we must 
only check the natural map of sheaves. 
This can be done locally,
assuming $\sX_i=\Spec(B_i)$ and $\sY=\Spec(A)$ are affine.
The latter reduces to
\[
\begin{tikzcd}
A \otimes_{B_1} A \otimes_{B_2} A \arrow[r] \arrow[d] & 
A \otimes_{B_2} A \arrow[d] \\
A \otimes_{B_1} A \arrow[r] & A, 
\end{tikzcd}
\qquad \quad
\begin{tikzcd}
a_1 \otimes a_2 \otimes a_3 \arrow[r, |->] \arrow[d, |->] & 
a_1a_2 \otimes a_3 \arrow[d, |->] \\
a_1 \otimes a_2a_3 \arrow[r, |->] & a_1a_2a_3, 
\end{tikzcd}
\]
where commutativity is straightforward.
\end{proof}

Given cohomologically-proper maps $\pi_i \colon \tsX \to \sX_i$ of concentrated algebraic stacks of finite Tor-dimension for which the natural map
$\mathcal O_{\sX_i} \to \pi_{i \ast} \mathcal O_{\tsX}$
is a quasi-isomorphism for all $i \in I$, 
promote $I$ to a category by taking the trivial 
groupoid on the set $I$, and set 
$\operatorname{Hom}_I(i,j) = \lbrace u_{ij} \rbrace$ 
for each $i,j$ with $u_{ii} = 1_i$. 

\begin{lemma} \label{lem:pullback_lax}
There is a lax functor
$\cC_{\tsX} \colon I^{\operatorname{op}} \to \operatorname{Cat}$, where 
\begin{align*}
\cC_{\widetilde{\sX}}(i)  
&\ce D(\sX_i), 
& \cC_{\widetilde{\sX}}(u_{ij})  \ce 
\pi_{i \ast}\pi_j^\ast 
&\colon D(\sX_j) \to D(\sX_i),
\\ 
\alpha_i  \ce \eta_i 
&\colon \Id \to \pi_{i \ast} \pi_i^\ast, 
& \gamma_{u_{ij}, u_{jl}}  \ce 
\pi_{l\ast} \epsilon_j \pi_i^\ast 
&\colon \pi_{i \ast}\pi_j^\ast \pi_{j \ast} \pi_l^\ast \to \pi_{i \ast} \pi_l^\ast. 
\end{align*} 
\end{lemma}

\begin{proof}
The commutativity of the first two diagrams follows from $\epsilon_i \circ
\eta_i = 1$, and for the final diagram this reduces to \Cref{lem:counits_commute}. 
\end{proof}

By the projection formula, the maps $\alpha_i$ in Lemma~\ref{lem:pullback_lax} are natural isomorphisms since $\mathcal O_{\sX_i} \cong \pi_{i \ast} \mathcal O_{\tsX}$, that is, $\cC_{\tsX}$ is a unitary lax functor. The lax functor obtained from a common rational resolution 
is independent of the resolution, up to isomorphism. 

\begin{lemma} \label{lem:lax_func_unique}
If $\tsX$ is as above and 
$f \colon \tsX^\prime \to \tsX$ is a morphism 
where $\mathcal O_{\tsX} \to f_\ast \mathcal O_{\tsX^\prime}$ is a 
quasi-isomorphism, then there is an isomorphism of lax functors $\cC_{\tsX^\prime} \cong
\cC_{\tsX}$. 
\end{lemma}

\begin{proof}
Set 
$\nu_i \ce \Id \colon D(\sX_i) \to D(\sX_i)$. 
For the other component, use the composition 
\[
\nu_{ij} \colon \pi_{i \ast} \pi_j^\ast \xrightarrow{\ \pi_{i\ast} \eta_f \pi_j^\ast\ } 
\pi_{i\ast} f_{\ast} f^\ast \pi_j^\ast \cong (\pi_i \circ f)_\ast (\pi_i \circ f)^\ast
= \pi_{i \ast}^\prime \pi_j^{\prime \ast}. 
\]
The diagrams 
\[
\begin{tikzcd}
\Id \arrow[r, "\eta_i"] \arrow[dr, "\eta_i^\prime"'] & 
\pi_{i\ast} \pi_i^\ast 
\arrow[d, "\nu_{ii}"] \\ 
& \pi_{i \ast}^\prime \pi_i^{\prime \ast}
\end{tikzcd} \qquad \text{and} \qquad 
\begin{tikzcd}
\pi_{i \ast} \pi_j^\ast \pi_{j \ast} \pi_l^\ast 
\arrow[r, "\pi_{i\ast} \epsilon_j \pi_l^\ast"] 
\arrow[d, "\nu_{ij}\nu_{jl}"'] & 
\pi_{i \ast}\pi_l^\ast \arrow[d,"\nu_{il}"] \\
\pi_{i\ast}^\prime \pi_j^{\prime \ast} \pi_{j \ast}^\prime \pi_l^{\prime\ast}
\arrow[r, "\pi_{i\ast}^\prime \epsilon_j^\prime \pi_l^{\prime\ast}"] &
\pi_{i\ast}^\prime \pi_l^{\prime \ast}
\end{tikzcd}
\]
commute since, similar to the arguments in the proof of
\cref{lem:counits_commute}, we can reduce 
to the affine case, where the necessary arguments amount to standard facts about tensor products.
\end{proof}

\begin{prop} \label{prop: general GC}
The full subcategory of $D(\tsX)$ consisting of 
objects from $\pi_i^\ast D(\sX_i)$ is equivalent to 
$\int \cC_{\tsX}$.
\end{prop}

\begin{proof}
Denote the full subcategory by $\Pi$.  
Thanks to their naturality, the adjunctions 
$\beta_{ij} \colon \operatorname{Hom}_{\sX_i}(E, \pi_{i \ast} \pi_j^\ast F) \cong 
\operatorname{Hom}_{\tsX}(\pi_i^\ast E, \pi_j^\ast F)$ 
induce a functor 
\begin{align*}
    \beta \colon \int \cC_{\tX} & \to \Pi, 
    \quad\text{where \ }
    (E, i) 
    \mapsto \pi_i^*E 
    \text{ \ and \ }
    (\phi, u_{ij}) 
    \mapsto \beta_{ij}(\phi),
\end{align*}
which is essentially surjective by definition and fully faithful since $\beta_{ij}$
is an isomorphism. 
\end{proof}

Now we turn to our situation, where the $\sX_i$ correspond to the chambers of a GKZ fan. 
Choose $\tsX$ as in \cref{subsec:constructionOfWX} and 
set $\cC_{GKZ} \ce \cC_{\tsX}$. 
From \cref{lem:lax_func_unique}, $\cC_{GKZ}$ is independent of the choice of $\tsX$, up to isomorphism.
Recall that for $\sX_i, \sX_j$, 
the closure $\Delta_{ij}$ of the diagonal torus in $\sX_i \times \sX_j$ provides a correspondence with $ p_{j \ast} p_i^* = \Phi_{ij}$: 
\[
\begin{tikzcd}
  & \Delta_{ij} \arrow[ld, " p_i"'] \arrow[rd, " p_j"] & \\
\sX_i & & \sX_j. 
\end{tikzcd}
\] 

\begin{lemma} \label{lem:gkz_lax_functor_skele} 
The 1-skeleton of $\cC_{GKZ}$ is isomorphic to $(D(\sX_i),\Phi_{ij})$.
\end{lemma}

\begin{proof}
Since $\pi_i$ factors as $ p_i \circ f_{ij}$ for some 
$f_{ij}\colon \tsX \to \Delta_{ij}$, similar to the proof of
\cref{lem:lax_func_unique}, there are isomorphisms 
$\Phi_{ij} \overset{\sim}{\to} \pi_{j\ast}\pi_i^\ast$ 
 identifying the functor components of the 1-skeleton.
\end{proof}

One could take another space $\tsX$ lying over the $\sX_i$ to produce natural transformations satisfying \eqref{eqn:GC_cocycle}, as above; however, the choice of $\tsX$ will, in general, change the underlying 1-skeleton 
unless 
$\mathcal O_{\sX_i} \cong f_\ast \mathcal O_{\tsX}$.
This is reminiscent of the fact that rational resolutions yield categorical resolutions in the sense of \cite{kuznetsov-lunts},  but general resolutions do not.

\begin{theorem} \label{thm:cox_groth}
The Cox category is equivalent to the pretriangulated
envelope of (a dg enhancement of) the Grothendieck construction 
on the lax functor $\cC_{GKZ}$.
\end{theorem}

\begin{proof}
By \Cref{prop: general GC}, $\int \cC_{GKZ}$ is the full
subcategory of $D(\tsX)$ consisting of objects from  $\pi_i^\ast D(\sX_i)$.  For
any dg enhancement of $D(\tsX)$, taking the full subcategory 
of objects homotopic to objects from $\pi_i^\ast D(\sX_i)$ 
gives an enhancement of $\int \cC_{GKZ}$.
By definition, its pretriangulated envelope consists of objects 
homotopic to those from $D_{\Cox}$.
\end{proof}

\end{appendices}

\subsection*{Acknowledgments}
This material is based upon work supported by the National Science Foundation
under Grants No. DMS-2001101 (Berkesch), DMS-2048055 (Ganatra), DMS-2200469
(Erman), DMS-2302262 (Favero), DMS-2302263 (Ballard), DMS-2302373 (Brown),
DMS-2549013 \newline (Hanlon, previously as DMS-2404882). We thank the NSF for its support. 

Early discussions among our group occurred at the workshop ``Syzygies and mirror symmetry" at the American Institute of Mathematics in 2023.
Part of this research was performed 
while the authors were visiting 
the Simons Laufer Mathematical Sciences Institute (SLMath)
(formerly the Mathematical Sciences Research
Institute; MSRI), 
which is supported by the National Science Foundation (Grant No. DMS-1928930),
during the SLMath programs on ``Commutative Algebra'' and ``Noncommutative algebraic geometry'' in 2024. We thank both AIM and SLMath for their support and productive research environments.

We are also grateful to many colleagues for many helpful conversations: 
Arend Bayer, 
Lev Borisov,
Andrei C\u{a}ld\u{a}raru, 
David Eisenbud, 
Milena Hering, 
Jeff Hicks,
Oleg Lazarev,
Brian Lehmann,
Martin Olsson,  
Mahrud Sayrafi,
Hal Schenck,
Karl Schwede,
Greg Smith, 
Mark Walker, 
and others.  
In addition, we would like to express particular gratitude to Lev Borisov for explaining an alternative proof of \cref{lem:PushPull} as described in \cref{rem:borisov} as well as Mykola Sapronov for pointing out an error in an earlier version of this manuscript which lead to \Cref{rem: Sapronov}.
The computer algebra system {\texttt Macaulay2}~\cite{M2} provided
valuable assistance.

\bibliographystyle{amsalpha}
\bibliography{refs}

@article{keller1994deriving,
  title={Deriving DG categories},
  author={Keller, Bernhard},
journal ={Ann. Sci. {\'E}c. Norm. Sup{\'e}r. (4)},
  fjournal={Annales scientifiques de l'Ecole normale sup{\'e}rieure},
  volume={27},
  number={1},
  pages={63--102},
  year={1994}
}

@article{KM95,
  title={Quotients by groupoids},
  author={Keel, Se{\'a}n and Mori, Shigefumi},
    journal={Ann. of Math. (2)},
  fjournal={Annals of mathematics},
  volume={145},
  number={1},
  pages={193--213},
  year={1997},
  publisher={JSTOR}
}

@Article{CastravetTevelevI,
 Author = {Castravet, Ana-Maria and Tevelev, Jenia},
 Title = {Derived category of moduli of pointed curves. {I}},
 FJournal = {Algebraic Geometry},
 Journal = {Algebr. Geom.},
 ISSN = {2313-1691},
 Volume = {7},
 Number = {6},
 Pages = {722--757},
 Year = {2020},
 DOI = {10.14231/AG-2020-026},
 zbMATH = {7297283},
 Zbl = {1467.14069}
}

@article {Kuw20,
  title={The nonequivariant coherent-constructible correspondence for toric stacks},
  author={Kuwagaki, Tatsuki},
  journal={Duke Math. J.},
  volume={169},
  number={11},
  pages={2125--2197},
  year={2020},
  publisher={Duke University Press}
}

@article {AKO,
    AUTHOR = {Auroux, Denis and Katzarkov, Ludmil and Orlov, Dmitri},
     TITLE = {Mirror symmetry for weighted projective planes and their
              noncommutative deformations},
   JOURNAL = {Ann. of Math. (2)},
  FJOURNAL = {Annals of Mathematics. Second Series},
    VOLUME = {167},
      YEAR = {2008},
    NUMBER = {3},
     PAGES = {867--943},
      ISSN = {0003-486X},
   MRCLASS = {53D40 (14A22 14J32 18E30 53D12)},
MRREVIEWER = {Justin Sawon},
       DOI = {10.4007/annals.2008.167.867},
       URL = {https://doi.org/10.4007/annals.2008.167.867},
}

@article{BorisovChenSmith2004,
  title={The Orbifold Chow Ring of Toric {D}eligne-{M}umford Stacks},
  author={Borisov, Lev A. and Chen, Linda and Smith, Gregory G.},
  journal={J. Amer. Math. Soc.},
  volume={18},
  number={1},
  pages={193--215},
  year={2004},
  publisher={American Mathematical Society},
  doi={10.1090/S0894-0347-04-00471-0}
}

@incollection {van-den-bergh-nccr,
    AUTHOR = {Van den Bergh, Michel},
     TITLE = {Non-commutative crepant resolutions},
 BOOKTITLE = {The legacy of {N}iels {H}enrik {A}bel},
     PAGES = {749--770},
 PUBLISHER = {Springer, Berlin},
      YEAR = {2004},
      ISBN = {3-540-43826-2},
   MRCLASS = {14A22 (14E15 16S38)},
  mrnumber2 = {2077594},
MRREVIEWER = {Adam\ Nyman},
}

@Article{SVDB-KirwanRes,
 Author = {{\v{S}}penko, {\v{S}}pela and Van den Bergh, Michel},
 Title = {Comparing the {Kirwan} and noncommutative resolutions of quotient varieties},
 FJournal = {Journal f{\"u}r die Reine und Angewandte Mathematik},
 Journal = {J. Reine Angew. Math.},
 ISSN = {0075-4102},
 Volume = {801},
 Pages = {1--43},
 Year = {2023},
 DOI = {10.1515/crelle-2023-0024},
 zbMATH = {7723219},
 Zbl = {1519.14010}
}

@article {spenko-van-den-bergh-inventiones,
    AUTHOR = {{\v{S}}penko, {\v{S}}pela and Van den Bergh, Michel},
     TITLE = {Non-commutative resolutions of quotient singularities for
              reductive groups},
   JOURNAL = {Invent. Math.},
  FJOURNAL = {Inventiones Mathematicae},
    VOLUME = {210},
      YEAR = {2017},
    NUMBER = {1},
     PAGES = {3--67},
      ISSN = {0020-9910,1432-1297},
   MRCLASS = {14A22 (13A50 14L24 16E35)},
MRREVIEWER = {Arvid\ Siqveland},
       DOI = {10.1007/s00222-017-0723-7},
       URL = {https://doi.org/10.1007/s00222-017-0723-7},
}

@article {kuznetsov-lunts,
    AUTHOR = {Kuznetsov, Alexander and Lunts, Valery A.},
     TITLE = {Categorical resolutions of irrational singularities},
   JOURNAL = {Int. Math. Res. Not.},
  FJOURNAL = {International Mathematics Research Notices. IMRN},
      YEAR = {2015},
    NUMBER = {13},
     PAGES = {4536--4625},
      ISSN = {1073-7928,1687-0247},
   MRCLASS = {14F05 (14B05 18E30)},
MRREVIEWER = {Andreas\ Krug},
       DOI = {10.1093/imrn/rnu072},
       URL = {https://doi.org/10.1093/imrn/rnu072},
}

@book {huybrechts06,
    AUTHOR = {Huybrechts, D.},
     TITLE = {Fourier-{M}ukai transforms in algebraic geometry},
    SERIES = {Oxford Mathematical Monographs},
 PUBLISHER = {The Clarendon Press, Oxford University Press, Oxford},
      YEAR = {2006},
     PAGES = {viii+307},
      ISBN = {978-0-19-929686-6; 0-19-929686-3},
   MRCLASS = {14F05 (14-02 18E30)},
  MRNUMBER = {2244106},
MRREVIEWER = {Bal\'{a}zs Szendr\H{o}i},
       DOI = {10.1093/acprof:oso/9780199296866.001.0001},
       URL = {https://doi.org/10.1093/acprof:oso/9780199296866.001.0001},
}

@article {BK89,
    AUTHOR = {Bondal, A. I. and Kapranov, M. M.},
     TITLE = {Representable functors, {S}erre functors, and reconstructions},
   JOURNAL = {Izv. Akad. Nauk SSSR Ser. Mat.},
  FJOURNAL = {Izvestiya Akademii Nauk SSSR. Seriya Matematicheskaya},
    VOLUME = {53},
      YEAR = {1989},
    NUMBER = {6},
     PAGES = {1183--1205, 1337},
      ISSN = {0373-2436},
   MRCLASS = {14F05 (18E30 32C38 32L99 32S60)},
  mrnumber2 = {1039961},
MRREVIEWER = {Alexey N. Rudakov},
       DOI = {10.1070/IM1990v035n03ABEH000716},
       URL = {https://doi-org.spot.lib.auburn.edu/10.1070/IM1990v035n03ABEH000716},
}

@article{michalek11,
title = {Family of counterexamples to {K}ing's conjecture},
journal = {C. R. Math.},
volume = {349},
number = {1},
pages = {67-69},
year = {2011},
issn = {1631-073X},
doi = {https://doi.org/10.1016/j.crma.2010.11.027},
url = {https://www.sciencedirect.com/science/article/pii/S1631073X10003729},
author = {Michałek, Mateusz}
}

@incollection {borisov-horja,
    AUTHOR = {Borisov, Lev A. and Horja, R. Paul},
     TITLE = {On the {$K$}-theory of smooth toric {DM} stacks},
 BOOKTITLE = {Snowbird lectures on string geometry},
    SERIES = {Contemp. Math.},
    VOLUME = {401},
     PAGES = {21--42},
 PUBLISHER = {Amer. Math. Soc., Providence, RI},
      YEAR = {2006},
      ISBN = {0-8218-3663-3},
   MRCLASS = {14C35 (14A20 14M25 19E20)},
MRREVIEWER = {Alicia\ Dickenstein},
       DOI = {10.1090/conm/401/07551},
       URL = {https://doi.org/10.1090/conm/401/07551},
}

@article {GS,
    AUTHOR = {Geraschenko, Anton and Satriano, Matthew},
     TITLE = {Toric stacks {I}: {T}he theory of stacky fans},
   JOURNAL = {Trans. Amer. Math. Soc.},
  FJOURNAL = {Transactions of the American Mathematical Society},
    VOLUME = {367},
      YEAR = {2015},
    NUMBER = {2},
     PAGES = {1033--1071},
      ISSN = {0002-9947,1088-6850},
   MRCLASS = {14D23 (14M25)},
MRREVIEWER = {Fabio\ Perroni},
       DOI = {10.1090/S0002-9947-2014-06063-7},
       URL = {https://doi.org/10.1090/S0002-9947-2014-06063-7},
}

@article{beilinson,
  title={Coherent sheaves on $\mathbb{P}^n$ and problems of linear algebra},
  author={Beilinson, Alexander A.},
  journal={Funct. Anal. Appl.},
  volume={12},
  number={3},
  pages={214--216},
  year={1978},
  publisher={Springer}
}

@incollection {Bondal,
AUTHOR = {Bondal, Alexei},
     TITLE = {Convex and algebraic geometry},
      NOTE = {Abstracts from the workshop held January 29--February 4, 2006,
              Organized by Klaus Altmann, Victor Batyrev and Bernard
              Teissier,
              Oberwolfach Reports. Vol. 3, no. 1},
   JOURNAL = {Oberwolfach Rep.},
  FJOURNAL = {Oberwolfach Reports},
    VOLUME = {3},
      YEAR = {2006},
    NUMBER = {1},
     PAGES = {253--316},
      ISSN = {1660-8933},
   MRCLASS = {14-06 (52-06)},
  mrnumber2 = {2278891},
       DOI = {10.4171/OWR/2006/05},
       URL = {https://doi-org.spot.lib.auburn.edu/10.4171/OWR/2006/05},
}

@article {Thomsen,
    AUTHOR = {Thomsen, Jesper Funch},
     TITLE = {Frobenius direct images of line bundles on toric varieties},
   JOURNAL = {J. Algebra},
  FJOURNAL = {Journal of Algebra},
    VOLUME = {226},
      YEAR = {2000},
    NUMBER = {2},
     PAGES = {865--874},
      ISSN = {0021-8693},
   MRCLASS = {14M25 (13A35 14C20)},
  mrnumber2 = {1752764},
MRREVIEWER = {Alicia Dickenstein},
       DOI = {10.1006/jabr.1999.8184},
       URL = {https://doi-org.spot.lib.auburn.edu/10.1006/jabr.1999.8184},
}

@article{BFK_GIT,
  title={Variation of {G}eometric {I}nvariant {T}heory quotients and derived categories},
  author={Ballard, Matthew and Favero, David and Katzarkov, Ludmil},
  journal={J. Reine Angew. Math.},
  volume={2019},
  number={746},
  pages={235--303},
  year={2019},
  publisher={De Gruyter}
}

@article{BallardEtAl2018arxivDerivedCategories,
  title={Derived categories of centrally-symmetric smooth toric {F}ano varieties},
  author={Ballard, Matthew and Duncan, Alexander and McFaddin, Patrick K.},
  journal={Math. Nachr.},
  volume={295},
  number={2},
  pages={218--241},
  year={2022},
  publisher={Wiley Blackwell}
}

@article{BallardEtAl2019ArithmeticToric,
  author  = {Ballard, Matthew and Duncan, Alexander and McFaddin, Patrick K.},
  title   = {On derived categories of arithmetic toric varieties},
  journal = {Ann. K-Theory},
  volume  = {4},
  number  = {2},
  pages   = {211--242},
  year    = {2019},
}

@article{BallardEtAl2019ToricFrobenius,
  author  = {Ballard, Matthew and Duncan, Alexander and McFaddin, Patrick K.},
  title   = {The toric {F}robenius morphism and a conjecture of {O}rlov},
  journal = {Eur. J. Math.},
  volume  = {5},
  number  = {3},
  pages   = {640--645},
  year    = {2019},
}

@article {HLS,
    AUTHOR = {Halpern-Leistner, Daniel and Sam, Steven V.},
     TITLE = {Combinatorial constructions of derived equivalences},
   JOURNAL = {J. Amer. Math. Soc.},
  FJOURNAL = {Journal of the American Mathematical Society},
    VOLUME = {33},
      YEAR = {2020},
    NUMBER = {3},
     PAGES = {735--773},
      ISSN = {0894-0347},
   MRCLASS = {14F08 (14L24 19E08)},
MRREVIEWER = {Alfonso Zamora},
       DOI = {10.1090/jams/940},
       URL = {https://doi-org.spot.lib.auburn.edu/10.1090/jams/940},
}

@article {bernardi-tirabassi,
    AUTHOR = {Bernardi, Alessandro and Tirabassi, Sofia},
     TITLE = {Derived categories of toric {F}ano 3-folds via the {F}robenius
              morphism},
   JOURNAL = {Matematiche (Catania)},
  FJOURNAL = {Le Matematiche},
    VOLUME = {64},
      YEAR = {2009},
    NUMBER = {2},
     PAGES = {117--154},
      ISSN = {0373-3505,2037-5298},
   MRCLASS = {14F05 (14M25)},
MRREVIEWER = {Daniele\ Faenzi},
}

@article {costa-miro-roig,
    AUTHOR = {Costa, Laura and Mir\'{o}-Roig, Rosa Maria},
     TITLE = {Frobenius splitting and derived category of toric varieties},
   JOURNAL = {Illinois J. Math.},
  FJOURNAL = {Illinois Journal of Mathematics},
    VOLUME = {54},
      YEAR = {2010},
    NUMBER = {2},
     PAGES = {649--669},
      ISSN = {0019-2082,1945-6581},
   MRCLASS = {14F05 (14M25)},
MRREVIEWER = {Pawel\ Sosna},
       URL = {http://projecteuclid.org/euclid.ijm/1318598676},
}

@article {costa-miro-roig2,
    AUTHOR = {Costa, Laura and Mir\'{o}-Roig, Rosa Maria},
     TITLE = {Derived category of toric varieties with small {P}icard
              number},
   JOURNAL = {Cent. Eur. J. Math.},
  FJOURNAL = {Central European Journal of Mathematics},
    VOLUME = {10},
      YEAR = {2012},
    NUMBER = {4},
     PAGES = {1280--1291},
      ISSN = {1895-1074,1644-3616},
   MRCLASS = {14F05 (14M25)},
MRREVIEWER = {Arvid\ Perego},
       DOI = {10.2478/s11533-012-0060-4},
       URL = {https://doi.org/10.2478/s11533-012-0060-4},
}

@article {dey-lason-michalek,
    AUTHOR = {Dey, Arijit and Laso\'{n}, Micha\l and Micha{\l}ek, Mateusz},
     TITLE = {Derived category of toric varieties with {P}icard number
              three},
   JOURNAL = {Matematiche (Catania)},
  FJOURNAL = {Le Matematiche},
    VOLUME = {64},
      YEAR = {2009},
    NUMBER = {2},
     PAGES = {99--116},
      ISSN = {0373-3505,2037-5298},
   MRCLASS = {14F05 (14M25)},
MRREVIEWER = {Maria\ Chiara\ Brambilla},
}

@article {lason-michalek,
    AUTHOR = {Laso\'{n}, Micha\l and Micha{\l}ek, Mateusz},
     TITLE = {On the full, strongly exceptional collections on toric
              varieties with {P}icard number three},
   JOURNAL = {Collect. Math.},
  FJOURNAL = {Collectanea Mathematica},
    VOLUME = {62},
      YEAR = {2011},
    NUMBER = {3},
     PAGES = {275--296},
      ISSN = {0010-0757,2038-4815},
   MRCLASS = {14M25 (14F05 14J45)},
MRREVIEWER = {Lars\ Petersen},
       DOI = {10.1007/s13348-011-0044-x},
       URL = {https://doi.org/10.1007/s13348-011-0044-x},
}

@unpublished{FS,
  title={Line Bundle Resolutions via the Coherent-Constructible Correspondence},
  author={Favero, David and Sapronov, Mykola},
  note={arXiv:2411.17873},
  year={2024}
}

@misc {BCHSY,
author = {Berkesch, Christine and Cranton Heller, Lauren and Smith, Gregory G. and Yang, Jay},
title = {Cellular resolutions for normalizations of toric ideals},
note = {In preparation},
}

@article {EFS,
    AUTHOR = {Eisenbud, David and Fl{\o}ystad, Gunnar and Schreyer, Frank-Olaf},
     TITLE = {Sheaf cohomology and free resolutions over exterior algebras},
   JOURNAL = {Trans. Amer. Math. Soc.},
  FJOURNAL = {Transactions of the American Mathematical Society},
    VOLUME = {355},
      YEAR = {2003},
    NUMBER = {11},
     PAGES = {4397--4426},
      ISSN = {0002-9947},
   MRCLASS = {14F05 (13D02)},
  mrnumber2 = {1990756},
MRREVIEWER = {David A. Jorgensen},
       DOI = {10.1090/S0002-9947-03-03291-4},
       URL = {https://doi-org.spot.lib.auburn.edu/10.1090/S0002-9947-03-03291-4},
}

@article{BETate,
  title={Tate resolutions on toric varieties},
  author={Brown, Michael K. and Erman, Daniel},
  journal={J. Eur. Math. Soc. (to appear)},
  year={2023}
}

@article{FLTZ11,
  title={A categorification of {M}orelli’s theorem},
  author={Fang, Bohan and Liu, Chiu-Chu Melissa and Treumann, David and Zaslow, Eric},
  journal={Invent. Math.},
  volume={186},
  number={1},
  pages={79--114},
  year={2011},
  publisher={Springer}
}

@incollection {caldararu,
    AUTHOR = {C\u{a}ld\u{a}raru, Andrei},
     TITLE = {Derived categories of sheaves: a skimming},
 BOOKTITLE = {Snowbird lectures in algebraic geometry},
    SERIES = {Contemp. Math.},
    VOLUME = {388},
     PAGES = {43--75},
 PUBLISHER = {Amer. Math. Soc., Providence, RI},
      YEAR = {2005},
   MRCLASS = {14F05 (13D03 18E30)},
MRREVIEWER = {Andrei D. Halanay},
       DOI = {10.1090/conm/388/07256},
       URL = {https://doi.org/10.1090/conm/388/07256},
}

@article{FLTZ12,
  title={T-duality and homological mirror symmetry for toric varieties},
  author={Fang, Bohan and Liu, Chiu-Chu Melissa and Treumann, David and Zaslow, Eric},
  journal={Adv. Math.},
  volume={229},
  number={3},
  pages={1873--1911},
  year={2012},
  publisher={Elsevier}
}

@article{FLTZ14,
  title={The {C}oherent--{C}onstructible {C}orrespondence for {T}oric {D}eligne--{M}umford {S}tacks},
  author={Fang, Bohan and Liu, Chiu-Chu Melissa and Treumann, David and Zaslow, Eric},
  JOURNAL = {Int. Math. Res. Not.},
  FJOURNAL = {International Mathematics Research Notices},
  volume={2014},
  number={4},
  pages={914--954},
  year={2014},
  publisher={OUP}
}

@article {olsson,
    AUTHOR = {Olsson, Martin C.},
     TITLE = {On proper coverings of {A}rtin stacks},
   JOURNAL = {Adv. Math.},
  FJOURNAL = {Advances in Mathematics},
    VOLUME = {198},
      YEAR = {2005},
    NUMBER = {1},
     PAGES = {93--106},
      ISSN = {0001-8708,1090-2082},
   MRCLASS = {14A20},
MRREVIEWER = {Charles\ D.\ Cadman},
       DOI = {10.1016/j.aim.2004.08.017},
       URL = {https://doi.org/10.1016/j.aim.2004.08.017},
}

@article {bondal-associative,
    AUTHOR = {Bondal, Alexei I.},
     TITLE = {Representations of associative algebras and coherent sheaves},
   JOURNAL = {Izv. Akad. Nauk SSSR Ser. Mat.},
  FJOURNAL = {Izvestiya Akademii Nauk SSSR. Seriya Matematicheskaya},
    VOLUME = {53},
      YEAR = {1989},
    NUMBER = {1},
     PAGES = {25--44},
      ISSN = {0373-2436},
   MRCLASS = {14F05},
  mrnumber2 = {992977},
MRREVIEWER = {Michael\ M.\ Kapranov},
       DOI = {10.1070/IM1990v034n01ABEH000583},
       URL = {https://doi.org/10.1070/IM1990v034n01ABEH000583},
}

@article {borisov-wang,
    AUTHOR = {Borisov, Lev and Wang, Chengxi},
     TITLE = {On strong exceptional collections of line bundles of maximal
              length on {F}ano toric {D}eligne-{M}umford stacks},
   JOURNAL = {Asian J. Math.},
  FJOURNAL = {Asian Journal of Mathematics},
    VOLUME = {25},
      YEAR = {2021},
    NUMBER = {4},
     PAGES = {505--520},
      ISSN = {1093-6106,1945-0036},
   MRCLASS = {14M25 (14C20 14F08)},
  mrnumber2 = {4413003},
MRREVIEWER = {Pieter\ Belmans},
       DOI = {10.4310/AJM.2021.v25.n4.a3},
       URL = {https://doi.org/10.4310/AJM.2021.v25.n4.a3},
}

@article {borisov-orlov-equivariant,
    AUTHOR = {Borisov, Lev A. and Orlov, Dmitri O.},
     TITLE = {Equivariant exceptional collections on smooth toric stacks},
   JOURNAL = {Izv. Ross. Akad. Nauk Ser. Mat.},
  FJOURNAL = {Izvestiya Rossiiskoi Akademii Nauk. Seriya Matematicheskaya},
    VOLUME = {83},
      YEAR = {2019},
    NUMBER = {4},
     PAGES = {50--85},
      ISSN = {1607-0046,2587-5906},
   MRCLASS = {14F08 (14M25 55U10)},
  mrnumber2 = {3985690},
MRREVIEWER = {P.\ E.\ Newstead},
       DOI = {10.4213/im8830},
       URL = {https://doi.org/10.4213/im8830},
}

@article {borisov-hua,
    AUTHOR = {Borisov, Lev and Hua, Zheng},
     TITLE = {On the conjecture of {K}ing for smooth toric
              {D}eligne-{M}umford stacks},
   JOURNAL = {Adv. Math.},
  FJOURNAL = {Advances in Mathematics},
    VOLUME = {221},
      YEAR = {2009},
    NUMBER = {1},
     PAGES = {277--301},
      ISSN = {0001-8708,1090-2082},
   MRCLASS = {14M25 (14A20 14F05)},
  mrnumber2 = {2509327},
MRREVIEWER = {T.\ Oda},
       DOI = {10.1016/j.aim.2008.11.017},
       URL = {https://doi.org/10.1016/j.aim.2008.11.017},
}

@article {altmann-immaculate,
    AUTHOR = {Altmann, Klaus and Buczy\'{n}ski, Jaros{\l}aw and Kastner, Lars
              and Winz, Anna-Lena},
     TITLE = {Immaculate line bundles on toric varieties},
   JOURNAL = {Pure Appl. Math. Q.},
  FJOURNAL = {Pure and Applied Mathematics Quarterly},
    VOLUME = {16},
      YEAR = {2020},
    NUMBER = {4},
     PAGES = {1147--1217},
      ISSN = {1558-8599,1558-8602},
   MRCLASS = {14M25 (14F06 14F17 14J60 52B20)},
MRREVIEWER = {Luca\ Ugaglia},
       DOI = {10.4310/PAMQ.2020.v16.n4.a12},
       URL = {https://doi.org/10.4310/PAMQ.2020.v16.n4.a12},
}

@article {altmann-displaying,
    AUTHOR = {Altmann, K. and Ploog, D.},
     TITLE = {Displaying the cohomology of toric line bundles},
   JOURNAL = {Izv. Ross. Akad. Nauk Ser. Mat.},
  FJOURNAL = {Izvestiya Rossiiskoi Akademii Nauk. Seriya Matematicheskaya},
    VOLUME = {84},
      YEAR = {2020},
    NUMBER = {4},
     PAGES = {66--78},
      ISSN = {1607-0046,2587-5906},
   MRCLASS = {14M25 (14C20 14F08 52B20)},
       DOI = {10.4213/im8948},
       URL = {https://doi.org/10.4213/im8948},
}

@article {prabhu-naik,
    AUTHOR = {Prabhu-Naik, Nathan},
     TITLE = {Tilting bundles on toric {F}ano fourfolds},
   JOURNAL = {J. Algebra},
  FJOURNAL = {Journal of Algebra},
    VOLUME = {471},
      YEAR = {2017},
     PAGES = {348--398},
      ISSN = {0021-8693,1090-266X},
   MRCLASS = {14F05 (14J45 14M25)},
  mrnumber2 = {3569189},
MRREVIEWER = {Colin\ Diemer},
       DOI = {10.1016/j.jalgebra.2016.09.007},
       URL = {https://doi.org/10.1016/j.jalgebra.2016.09.007},
}

@article {uehara,
    AUTHOR = {Uehara, Hokuto},
     TITLE = {Exceptional collections on toric {F}ano threefolds and
              birational geometry},
   JOURNAL = {Internat. J. Math.},
  FJOURNAL = {International Journal of Mathematics},
    VOLUME = {25},
      YEAR = {2014},
    NUMBER = {7},
     PAGES = {1450072, 32},
      ISSN = {0129-167X,1793-6519},
   MRCLASS = {14F05 (14E30 14J30 14J45 14M25)},
  mrnumber2 = {3238094},
MRREVIEWER = {Andreas\ H\"{o}ring},
       DOI = {10.1142/S0129167X14500724},
       URL = {https://doi.org/10.1142/S0129167X14500724},
}

@article {hille-perling,
    AUTHOR = {Hille, Lutz and Perling, Markus},
     TITLE = {A counterexample to {K}ing's conjecture},
   JOURNAL = {Compos. Math.},
  FJOURNAL = {Compositio Mathematica},
    VOLUME = {142},
      YEAR = {2006},
    NUMBER = {6},
     PAGES = {1507--1521},
      ISSN = {0010-437X,1570-5846},
       DOI = {10.1112/S0010437X06002260},
       URL = {https://doi.org/10.1112/S0010437X06002260},
}

@article{efimov,
author = {Efimov, Alexander I.},
title = {Maximal lengths of exceptional collections of line bundles},
journal = {J. Lond. Math. Soc. (2)},
volume = {90},
number = {2},
pages = {350-372},
doi = {https://doi.org/10.1112/jlms/jdu037},
year = {2014}
}

@book {CLSToricVarieties,
    AUTHOR = {Cox, David A.   and  Little, John B.   and  Schenck, Henry K.},
     TITLE = {Toric varieties},
    SERIES= {Graduate Studies in Mathematics},
    VOLUME= {124},
 PUBLISHER= {American Mathematical Society, Providence, RI},
      YEAR = {2011},
     PAGES= {xxiv+841},
}

@article{GPS,
title={Microlocal Morse theory of wrapped {F}ukaya categories},
  author={Ganatra, Sheel and Pardon, John and Shende, Vivek},
  journal={Ann. of Math. (2)},
  volume={199},
  number={3},
  pages={943--1042},
  year={2024},
  publisher={Department of Mathematics, Princeton University Princeton, New Jersey, USA}
}

@article{ganatra2023sectorial,
    title={Sectorial descent for wrapped {F}ukaya categories},
  author={Ganatra, Sheel and Pardon, John and Shende, Vivek},
  journal={J. Amer. Math. Soc.},
  volume={37},
  number={2},
  pages={499--635},
  year={2024}
}

@article{ganatra2020covariantly,
  title={Covariantly functorial wrapped Floer theory on {L}iouville sectors},
  author={Ganatra, Sheel and Pardon, John and Shende, Vivek},
  journal={Publ. Math. Inst. Hautes \'{E}tudes Sci.},
  volume={131},
  pages={73--200},
  year={2020}
}

@article {BES,
    AUTHOR = {Berkesch, Christine  and Erman, Daniel  and Smith, Gregory G.},
     TITLE = {Virtual resolutions for a product of projective spaces},
   JOURNAL= {Algebr. Geom.},
  FJOURNAL= {Algebraic Geometry},
    VOLUME= {7},
      YEAR = {2020},
    NUMBER= {4},
     PAGES= {460--481},
}

@article{halpern2015derived,
  title={The derived category of a GIT quotient},
  author={Halpern-Leistner, Daniel},
journal={J. Amer. Math. Soc.},
  fjournal={Journal of the American Mathematical Society},
  volume={28},
  number={3},
  pages={871--912},
  year={2015}
}

@article {short-HST,
    title={A short proof of the {Hanlon-Hicks-Lazarev} Theorem},
  author={Brown, Michael K. and Erman, Daniel},
  journal={Forum Math. Sigma},
fjournal ={Forum of Mathematics, Sigma},
  volume={12},
  pages={e56},
  year={2024}
}

@article {FH,
    AUTHOR={Favero, David  and Huang, Jesse},
    TITLE={Rouquier dimension is {K}rull dimension for normal toric varieties},
    JOURNAL={Eur. J. Math.},
   VOLUME={9},
   YEAR={2023},
   NUMBER={91},
   PAGES={40 pages},
}

@unpublished {FH2,
    AUTHOR={Favero, David  and Huang, Jesse},
    TITLE={Homotopy Path Algebras},
    NOTE={arXiv.2205.03730},
    PAGES={40 pages},
    YEAR = {2022},
}

@article {kawamataI,
    AUTHOR = {Kawamata, Yujiro},
     TITLE = {Derived categories of toric varieties},
   JOURNAL = {Michigan Math. J.},
  FJOURNAL = {Michigan Mathematical Journal},
    VOLUME = {54},
      YEAR = {2006},
    NUMBER = {3},
     PAGES = {517--535},
      ISSN = {0026-2285,1945-2365},
   MRCLASS = {14M25 (14F05 18E30)},
  mrnumber2 = {2280493},
MRREVIEWER = {Bal\'{a}zs\ Szendr\H{o}i},
       DOI = {10.1307/mmj/1163789913},
       URL = {https://doi.org/10.1307/mmj/1163789913},
}

@article {kawamataII,
    AUTHOR = {Kawamata, Yujiro},
     TITLE = {Derived categories of toric varieties {II}},
   JOURNAL = {Michigan Math. J.},
  FJOURNAL = {Michigan Mathematical Journal},
    VOLUME = {62},
      YEAR = {2013},
    NUMBER = {2},
     PAGES = {353--363},
      ISSN = {0026-2285,1945-2365},
   MRCLASS = {14F05 (14A20 14E30 14M25)},
  mrnumber2 = {3079267},
MRREVIEWER = {David\ Favero},
       DOI = {10.1307/mmj/1370870376},
       URL = {https://doi.org/10.1307/mmj/1370870376},
}

@article{HHL,
    title={Resolutions of toric subvarieties by line bundles and applications},
  author={Hanlon, Andrew and Hicks, Jeff and Lazarev, Oleg},
  journal={Forum of Mathematics, Pi},
  volume={12},
  pages={e24},
  year={2024},
  organization={Cambridge University Press}
}

@article {HHfunctoriality,
    AUTHOR = {Hanlon, Andrew and Hicks, Jeff},
     TITLE = {Aspects of functoriality in homological mirror symmetry for
              toric varieties},
   JOURNAL = {Adv. Math.},
  FJOURNAL = {Advances in Mathematics},
    VOLUME = {401},
      YEAR = {2022},
     PAGES = {Paper No. 108317, 92},
      ISSN = {0001-8708,1090-2082},
   MRCLASS = {53D37 (14J33 14M25 14T20)},
  mrnumber2 = {4394684},
MRREVIEWER = {Helge\ Ruddat},
       DOI = {10.1016/j.aim.2022.108317},
       URL = {https://doi.org/10.1016/j.aim.2022.108317},
}

@article {kuzentsov-lefschetz,
    AUTHOR = {Kuznetsov, Alexander},
     TITLE = {Lefschetz decompositions and categorical resolutions of
              singularities},
   JOURNAL = {Selecta Math. (N.S.)},
  FJOURNAL = {Selecta Mathematica. New Series},
    VOLUME = {13},
      YEAR = {2008},
    NUMBER = {4},
     PAGES = {661--696},
      ISSN = {1022-1824,1420-9020},
   MRCLASS = {18E30 (14E15)},
MRREVIEWER = {Emanuele\ Macr\'i},
       DOI = {10.1007/s00029-008-0052-1},
       URL = {https://doi.org/10.1007/s00029-008-0052-1},
}

@article {faber-muller-smith,
    AUTHOR = {Faber, Eleonore and Muller, Greg and Smith, Karen E.},
     TITLE = {Non-commutative resolutions of toric varieties},
   JOURNAL = {Adv. Math.},
  FJOURNAL = {Advances in Mathematics},
    VOLUME = {351},
      YEAR = {2019},
     PAGES = {236--274},
      ISSN = {0001-8708,1090-2082},
   MRCLASS = {13C14 (13A35 14A22 14M25 16E10 16S32)},
  mrnumber2 = {3951483},
MRREVIEWER = {Graham\ J.\ Leuschke},
       DOI = {10.1016/j.aim.2019.04.021},
       URL = {https://doi.org/10.1016/j.aim.2019.04.021},
}

@article {MS,
    AUTHOR = {Maclagan, Diane  and Smith, Gregory G.},
     TITLE = {Multigraded {C}astelnuovo-{M}umford regularity},
   JOURNAL= {J. Reine Angew. Math.},
    VOLUME= {571},
      YEAR = {2004},
     PAGES= {179--212},
}

@misc {M2,
    AUTHOR = {Grayson, Daniel R.  and Stillman, Michael E.},
     TITLE = {Macaulay2, a software system for research in algebraic geometry},
    JOURNAL= {Available at \url{http://www.math.uiuc.edu/Macaulay2/}}
}

@unpublished {king,
      AUTHOR={King, Alastair},
       TITLE={Tilting bundles on some rational surfaces},
        YEAR={1997},
}

@unpublished {anderson,
    AUTHOR = {Anderson, Reginald},
    TITLE = {A Resolution of the Diagonal for Toric {D}eligne-{M}umford Stacks},
    NOTE = {arXiv:2303.17497},
    YEAR = {2023},
    JOURNAL = {Preprint},
}

@unpublished {BE-nonstandard,
    AUTHOR = {Brown, Michael K. and Erman, Daniel},
    TITLE = {Positivity and nonstandard graded {B}etti numbers},
    JOURNAL = {Preprint},
    NOTE = {arXiv:2302.07403},
    YEAR={2023},
}

@misc{stacks-project,
    shorthand    = {Stacks},
    author       = {{The Stacks Project Authors}},
    title        = {\textit{Stacks Project}},
    howpublished = {\url{https://stacks.math.columbia.edu}},
    year         = {2018},
  }

@article{abouzaid2006homogeneous,
  title={Homogeneous coordinate rings and mirror symmetry for toric varieties},
  author={Abouzaid, Mohammed},
  journal={Geom. Topol.},
  volume={10},
  number={2},
  pages={1097--1156},
  year={2006},
  publisher={Mathematical Sciences Publishers}
}

@article{abouzaid2009morse,
  title={Morse homology, tropical geometry, and homological mirror symmetry for toric varieties},
  author={Abouzaid, Mohammed},
  journal={Selecta Math. (N.S.)},
  volume={15},
  pages={189--270},
  year={2009},
  publisher={Springer}
}

@article{bai2023rouquier,
  title={On the {R}ouquier dimension of wrapped {F}ukaya categories and a conjecture of {O}rlov},
  author={Bai, Shaoyun and C{\^o}t{\'e}, Laurent},
  journal={Compos. Math.},
  volume={159},
  number={3},
  pages={437--487},
  year={2023},
  publisher={London Mathematical Society}
}

@article {Jerby1,
    AUTHOR = {Jerby, Yochay},
     TITLE = {On {L}andau-{G}inzburg systems, quivers and monodromy},
   JOURNAL = {J. Geom. Phys.},
  FJOURNAL = {Journal of Geometry and Physics},
    VOLUME = {98},
      YEAR = {2015},
     PAGES = {504--534},
      ISSN = {0393-0440,1879-1662},
   MRCLASS = {53D45 (14J33 14M25 53D37)},
MRREVIEWER = {Yisong\ Yang},
       DOI = {10.1016/j.geomphys.2015.09.006},
       URL = {https://doi.org/10.1016/j.geomphys.2015.09.006},
}

@article {Jerby2,
    AUTHOR = {Jerby, Yochay},
     TITLE = {On exceptional collections of line bundles and mirror symmetry
              for toric del-{P}ezzo surfaces},
   JOURNAL = {J. Math. Phys.},
  FJOURNAL = {Journal of Mathematical Physics},
    VOLUME = {58},
      YEAR = {2017},
    NUMBER = {3},
     PAGES = {031704, 19},
      ISSN = {0022-2488,1089-7658},
   MRCLASS = {14J33 (14M25 14T05 18E30 53D37)},
MRREVIEWER = {Howard\ Nuer},
       DOI = {10.1063/1.4977482},
       URL = {https://doi.org/10.1063/1.4977482},
}

@article {Jerby3,
    AUTHOR = {Jerby, Yochay},
     TITLE = {On {L}andau-{G}inzburg systems, co-tropical geometry, and
              {$D^b(X)$} of various toric {F}ano manifolds},
   JOURNAL = {J. Math. Phys.},
  FJOURNAL = {Journal of Mathematical Physics},
    VOLUME = {61},
      YEAR = {2020},
    NUMBER = {6},
     PAGES = {061701, 36},
      ISSN = {0022-2488,1089-7658},
   MRCLASS = {14J33 (14F08 14T20 18F20 18G80)},
MRREVIEWER = {Xiaobin\ Li},
       DOI = {10.1063/1.5096516},
       URL = {https://doi.org/10.1063/1.5096516},
}

@article {achinger2015characterization,
    AUTHOR = {Achinger, Piotr},
     TITLE = {A characterization of toric varieties in characteristic {$p$}},
   JOURNAL = {Int. Math. Res. Not. },
  FJOURNAL = {International Mathematics Research Notices. IMRN},
      YEAR = {2015},
    NUMBER = {16},
     PAGES = {6879--6892},
      ISSN = {1073-7928,1687-0247},
   MRCLASS = {14M25 (13A35 13C20)},
MRREVIEWER = {Ma\l gorzata\ Aneta\ Marciniak},
       DOI = {10.1093/imrn/rnu151},
       URL = {https://doi.org/10.1093/imrn/rnu151},
}

@article {bogvad1998splitting,
    AUTHOR = {B{\o}gvad, Rikard},
     TITLE = {Splitting of the direct image of sheaves under the
              {F}robenius},
   JOURNAL = {Proc. Amer. Math. Soc.},
  FJOURNAL = {Proceedings of the American Mathematical Society},
    VOLUME = {126},
      YEAR = {1998},
    NUMBER = {12},
     PAGES = {3447--3454},
      ISSN = {0002-9939,1088-6826},
   MRCLASS = {14M25 (13A35 14J60)},
MRREVIEWER = {Jesper\ F.\ Thomsen},
       DOI = {10.1090/S0002-9939-98-05000-X},
       URL = {https://doi.org/10.1090/S0002-9939-98-05000-X},
}

@article {ohkawa2013frobenius,
    AUTHOR = {Ohkawa, Ryo and Uehara, Hokuto},
     TITLE = {Frobenius morphisms and derived categories on two dimensional
              toric {D}eligne-{M}umford stacks},
   JOURNAL = {Adv. Math.},
  FJOURNAL = {Advances in Mathematics},
    VOLUME = {244},
      YEAR = {2013},
     PAGES = {241--267},
      ISSN = {0001-8708,1090-2082},
   MRCLASS = {14F05 (14A20)},
MRREVIEWER = {Stefan\ Schr\"oer},
       DOI = {10.1016/j.aim.2013.04.023},
       URL = {https://doi.org/10.1016/j.aim.2013.04.023},
}

@article {fantechi2010smooth,
    AUTHOR = {Fantechi, Barbara and Mann, Etienne and Nironi, Fabio},
     TITLE = {Smooth toric {D}eligne-{M}umford stacks},
   JOURNAL = {J. Reine Angew. Math.},
  FJOURNAL = {Journal f\"ur die Reine und Angewandte Mathematik. [Crelle's
              Journal]},
    VOLUME = {648},
      YEAR = {2010},
     PAGES = {201--244},
      ISSN = {0075-4102,1435-5345},
   MRCLASS = {14M25 (14D23)},
MRREVIEWER = {Hsian-Hua\ Tseng},
       DOI = {10.1515/CRELLE.2010.084},
       URL = {https://doi.org/10.1515/CRELLE.2010.084},
}

@article {fujino2009smooth,
    AUTHOR = {Fujino, Osamu and Sato, Hiroshi},
     TITLE = {Smooth projective toric varieties whose nontrivial nef line
              bundles are big},
   JOURNAL = {Proc. Japan Acad. Ser. A Math. Sci.},
  FJOURNAL = {Japan Academy. Proceedings. Series A. Mathematical Sciences},
    VOLUME = {85},
      YEAR = {2009},
    NUMBER = {7},
     PAGES = {89--94},
      ISSN = {0386-2194},
   MRCLASS = {14M25 (14E30)},
MRREVIEWER = {Alan\ Stapledon},
       DOI = {10.3792/pjaa.85.89},
       URL = {https://doi.org/10.3792/pjaa.85.89},
}

@article {costa-miro-roig0,
    AUTHOR = {Costa, L. and Mir\'o-Roig, R. M.},
     TITLE = {Tilting sheaves on toric varieties},
   JOURNAL = {Math. Z.},
  FJOURNAL = {Mathematische Zeitschrift},
    VOLUME = {248},
      YEAR = {2004},
    NUMBER = {4},
     PAGES = {849--865},
      ISSN = {0025-5874,1432-1823},
   MRCLASS = {14F05 (14M25 18E30)},
MRREVIEWER = {A.\ Prabhakar\ Rao},
       DOI = {10.1007/s00209-004-0684-6},
       URL = {https://doi.org/10.1007/s00209-004-0684-6},
}

@Book{SGA1,
 Title = {S{\'e}minaire de g{\'e}om{\'e}trie alg{\'e}brique du {Bois} {Marie} 1960/61 ({SGA} 1), dirig{\'e} par {Alexander} {Grothendieck}. {Augment{\'e}} de deux expos{\'e}s de {M}. {Raynaud}. {Rev{\^e}tements} {\'e}tales et groupe fondamental. {Expos{\'e}s} {I} {\`a} {XIII}. ({Seminar} on algebraic geometry at {Bois} {Marie} 1960/61 ({SGA} 1), directed by {Alexander} {Grothendieck}. {Enlarged} by two reports of {M}. {Raynaud}. {\`E}tale coverings and fundamental group)},
 FSeries = {Lecture Notes in Mathematics},
 Series = {Lect. Notes Math.},
 ISSN = {0075-8434},
 Volume = {224},
 Publisher = {Springer, Cham, 1971},
 Language = {French},
 DOI = {10.1007/BFb0058656},
 zbMATH = {3370468},
 Zbl = {0234.14002}
}

@Book{StacksOlsson,
 Author = {Olsson, Martin},
 Title = {Algebraic spaces and stacks},
 FSeries = {Colloquium Publications. American Mathematical Society},
 Series = {Colloq. Publ., Am. Math. Soc.},
 ISSN = {0065-9258},
 Volume = {62},
 ISBN = {978-1-4704-2798-6; 978-1-4704-2865-5},
 Year = {2016},
 Publisher = {Providence, RI: American Mathematical Society (AMS)},
 zbMATH = {6561920},
 Zbl = {1346.14001}
}

@Book{Waterhouse,
 Author = {Waterhouse, William C.},
 Title = {Introduction to affine group schemes},
 FSeries = {Graduate Texts in Mathematics},
 Series = {Grad. Texts Math.},
 ISSN = {0072-5285},
 Volume = {66},
 Year = {1979},
 Publisher = {Springer, Cham},
 zbMATH = {3689557},
 Zbl = {0442.14017}
}

@InCollection{KreschDM,
 Author = {Kresch, Andrew},
 Title = {On the geometry of {Deligne}-{Mumford} stacks},
 BookTitle = {Algebraic geometry, {S}eattle 2005. {P}roceedings of the 2005 {S}ummer {R}esearch {I}nstitute, {Seattle}, {WA}, {USA}, {J}uly 25--{A}ugust 12, 2005},
 ISBN = {978-0-8218-4702-2; 978-0-8218-4057-3},
 Pages = {259--271},
 Year = {2009},
 Publisher = {Providence, RI: American Mathematical Society (AMS)},
 zbMATH = {5548152},
 Zbl = {1169.14001}
}

@Book{JohnsonYau_2d_cats,
 Author = {Johnson, Niles and Yau, Donald},
 Title = {2-dimensional categories},
 ISBN = {978-0-19-887137-8; 978-0-19-887138-5},
 Year = {2021},
 Publisher = {Oxford: Oxford University Press},
 DOI = {10.1093/oso/9780198871378.001.0001},
 zbMATH = {7289983},
 Zbl = {1471.18002}
}

@unpublished{sanchez2023derived,
  title={Derived Categories of Permutahedral and Stellahedral Varieties},
  author={Sanchez, Mario},
  note={arXiv:2311.04203},
  year={2023}
}

@article{proudfoot2011all,
  title={All the {GIT} quotients at once},
  author={Proudfoot, Nicholas},
  JOURNAL = {Trans. Amer. Math. Soc.},
  FJOURNAL = {Transactions of the American Mathematical Society},
  volume={363},
  number={4},
  pages={1687--1698},
  year={2011}
}
\Addresses
\end{document}